\documentclass{article}
\usepackage{amsmath,amsthm,amsfonts,amssymb,enumerate,stmaryrd,authblk,framed}
\usepackage{fullpage}
\usepackage{microtype}
\usepackage{hyperref}

\theoremstyle{plain}
\newtheorem{theorem}{Theorem}[section]
\newtheorem{lemma}[theorem]{Lemma}
\newtheorem{corollary}[theorem]{Corollary}
\newtheorem{proposition}[theorem]{Proposition}
\newtheorem{question}[theorem]{Question}
\theoremstyle{definition}
\newtheorem{definition}[theorem]{Definition}

\newtheorem*{theorem:invariance-few-coordinates}{Theorem \ref{thm:invariance-few-coordinates}}

\DeclareMathOperator*{\EE}{\mathbb{E}}
\DeclareMathOperator*{\VV}{\mathbb{V}}
\DeclareMathOperator*{\FF}{\mathbb{F}}
\DeclareMathOperator{\sgn}{sgn}
\DeclareMathOperator{\Inf}{Inf}
\DeclareMathOperator{\Ber}{Ber}
\DeclareMathOperator{\im}{im}
\providecommand{\RR}{\mathbb{R}}
\providecommand{\charf}[1]{\mathbf{1}_{#1}}
\providecommand{\nor}{\mathrm{N}}
\providecommand{\bin}{\mathrm{B}}
\providecommand{\cG}{\mathcal{G}}
\providecommand{\rpi}{\boldsymbol\pi}
\providecommand{\rX}{\mathbf{X}}
\providecommand{\rU}{\mathbf{U}}
\providecommand{\rD}{\mathbf{D}}
\providecommand{\rC}{\mathbf{C}}
\providecommand{\rs}{\mathbf{s}} 
\providecommand{\rt}{\mathbf{t}}
\providecommand{\rI}{\mathbf{I}}
\providecommand{\rf}{\mathbf{f}}
\providecommand{\rg}{\mathbf{g}}
\providecommand{\eqdef}{\triangleq}
\providecommand{\ib}[1]{\llbracket #1 \rrbracket}

\providecommand{\xsum}{S}
\providecommand{\xstd}{\tilde{S}}
\providecommand{\fc}[2]{#1_{#2}}
\providecommand{\fs}[2]{#1[#2]}
\providecommand{\cF}{\mathcal{F}}
\providecommand{\symm}{\mathrm{S}}
\providecommand{\rmd}{\mathrm{d}}

\title{Harmonicity and Invariance on Slices of the Boolean Cube}
\author[1]{Yuval Filmus}
\author[2]{Elchanan Mossel}
\affil[1]{Department of Computer Science, Technion --- Israel Institute of Technology \\
\texttt{yuvalfi@cs.technion.ac.il}}
\affil[2]{Mathematics and IDSS, Massachusetts Institute of Technology \\
\texttt{elmos@mit.edu}}

\begin{document}

\maketitle

\begin{abstract}
In a recent work with Kindler and Wimmer we proved an invariance principle for the slice for low-influence, low-degree harmonic multilinear polynomials (a polynomial in $x_1,\ldots,x_n$ is \emph{harmonic} if it is annihilated by $\sum_{i=1}^n \frac{\partial}{\partial x_i}$). Here we provide an alternative proof for \emph{general} low-degree harmonic multilinear polynomials, with no constraints on the influences. We show that any real-valued harmonic multilinear polynomial on the slice whose degree is $o(\sqrt{n})$ has approximately the same distribution under the slice and cube measures.

Our proof is based on ideas and results from the representation theory of $S_n$, along with a  novel decomposition of random increasing paths in the cube in terms of martingales and reverse martingales. While such decompositions have been used in the past for stationary reversible Markov chains, our decomposition is applied in a non-stationary non-reversible setup. 
We also provide simple proofs for some known and some new properties of harmonic functions which are crucial for the proof.

Finally, we provide independent simple proofs for the known facts that 1) one cannot distinguish between the slice and the cube based on functions of $o(n)$ coordinates and 2) 
Boolean symmetric functions on the cube cannot be approximated under the uniform measure by functions whose sum of influences is $o(\sqrt{n})$. 

\end{abstract}


\pagebreak

\tableofcontents

\pagebreak

\section{Introduction} \label{sec:introduction}

The basic question motivating our work is the following:

\begin{question} \label{q:1} 
Assume $n$ is even. How distinguishable are the uniform measure $\mu$ on $\{0,1\}^n$ and the measure $\nu$ given by the uniform measure on $\{0,1\}^n$ conditioned on $\sum_i x_i = n/2$?

More generally: how distinguishable are the product measure $\mu_p$ on $\{0,1\}^n$ in which each coordinate takes the value $1$ independently with probability $p$ and 
$\nu_{pn}$ given by the uniform measure on $\{0,1\}^n$ conditioned on $\sum_i x_i = p n$ (assuming $pn$ is an integer)?
\end{question} 

Note that the two measures are easily distinguished using the simple sum-of-coordinates test. 
How does the answer change if we ``do not allow" the sum-of-coordinate test? 
From a computational perspective we might be interested in restricted families of tests, such as \emph{low-depth circuits} or \emph{low-degree polynomials}. 
Furthermore, it turns out that the canonical representation of functions with respect to $\nu_{pn}$, called harmonic representation, does not include the function 
$\sum x_i$ (indeed, under $\nu_{pn}$ it can be represented as a constant). 


We call $\{0,1\}^n$ the \emph{cube}, the support of the distribution $\nu_{pn}$ the \emph{slice}, and the support of $\nu$ the \emph{middle slice}.
For exposition purposes, the introduction will only discuss the middle slice, though all results (previous and ours) extend for the case of $\mu_p$ and $\nu_{pn}$ for every fixed $p$. 

\subsection{Low-degree polynomials} 
In a recent joint work with Kindler and Wimmer~\cite{FKMW} we provided a partial answer to Question~\ref{q:1} by extending the non-linear invariance principle of~\cite{MOO}. 
As mentioned earlier, any function on the slice has a canonical representation as a 
{\em harmonic}~\footnote{This somewhat unfortunate terminology is borrowed from Bergeron~\cite[Section 8.4]{Bergeron}, in which an $S_n$-harmonic polynomial is one which is annihilated by $\sum_{i=1}^n \frac{\partial^k}{\partial x_i^k}$ for all $k$. For multilinear polynomials, both definitions coincide.} multilinear polynomial. The harmonic property means that $\sum_{i=1}^n \frac{\partial f}{\partial x_i} = 0$. 
Suppose that $f$ is a harmonic low-degree low-influence multilinear polynomial. 
The invariance principle of~\cite{FKMW} establishes that the distribution of $f$ under the measure $\nu$ is close
to its distribution under the product measure $\mu$ on $\{0,1\}^n$,
as well as to its distribution under the product space $\RR^n$ equipped with the product Gaussian measure  $\cG = \nor(1/2,1/4)^{\otimes n}$. 

The restriction to multilinear harmonic polynomials is quite natural in the slice --- as every function on the slice has a unique representation as a harmonic multilinear polynomial (this fact, due to Dunkl~\cite{Dunkl76}, is proved in Section~\ref{sec:harmonic}). It is the analog of the implicit restriction to multilinear polynomials in the original non-linear invariance principle. Further motivations, from commutative algebra and representation theory, are described in Section~\ref{sec:harmonic-use}.

\smallskip

Both the invariance principle proven in~\cite{MOO} and the one proven in~\cite{FKMW}
require that the functions have low influences. Indeed, a function like $x_1$ has a rather different distribution under $\mu$ compared to $\cG$.  Similarly, the function $x_1 - x_2$ has a rather different distribution under $\nu$ compared to $\cG$.
However, note that the distribution of $x_1 - x_2$ under $\nu$ is quite similar to its distribution under $\mu$. It is natural to speculate that low-degree harmonic functions have similar distributions under $\nu$ and $\mu$.
Unfortunately, the proof of the invariance principle in~\cite{FKMW} goes through Gaussian space, rendering the low-influence condition necessary even when comparing $\nu$ and $\mu$.

\smallskip

Our main result in this paper is a direct proof of the invariance principle on the slice, showing that the distribution of a low-degree harmonic function on the slice is close to its distribution on the corresponding cube. Our results do not require the condition of low influences.

\begin{theorem} \label{thm:invariance_intro} 
Let $f\colon \{-1,1\}^n \to \RR$ be a harmonic multilinear polynomial of degree $o(\sqrt{n})$ and variance~$1$. For any $1$-Lipschitz function $\varphi$ (i.e., one satisfying $|\varphi(x)-\varphi(y)| \leq |x-y|$),
\[
 |\EE_{\nu}[\varphi(f)] - \EE_{\mu}[\varphi(f)]| = o(1), 
\]
and the L\'evy distance\footnote{The L\'evy distance between two real random variables $X,Y$ is the infimum value of $\epsilon$ such that for all $t \in \RR$ it holds that $\Pr[X \leq t - \epsilon] - \epsilon \leq \Pr[Y \leq t] \leq \Pr[X \leq t + \epsilon] + \epsilon$.} between the distribution of $f$ under $\mu$ and its distribution under $\nu$ is 
$o(1)$. 

\end{theorem}

See Theorem~\ref{thm:invariance} as well as Corollary~\ref{cor:levy} for more quantitative bounds and more general statements (which apply in particular to any i.i.d.\ measure on the cube and the corresponding slice).  In Subsection~\ref{sec:high-degree} we show that the results cannot be extended to polynomials whose degree is much bigger than $\sqrt{n}$. 

While Theorem~\ref{thm:invariance_intro} cannot hold for arbitrary multilinear polynomials, we are able to recover a similar theorem for such polynomials by considering the distribution of $f$ on several coupled slices at once; see Theorem~\ref{thm:interpolation-2} and its corollaries. 

\paragraph*{Between representation theory and probability}
An interesting aspect of our work is the interplay between the $L_2$ theory of the slice and probabilistic arguments based on coupling.\footnote{The informal term ``$L_2$ theory'' refers to studying functions via $L_2$ norms and inner products. For example, Fourier analysis is an $L_2$ theory, since its central concept is the Fourier basis, which is an orthonormal basis with respect to a given inner product. In contrast, total variation distance, which is inherently related to coupling, is an $L_1$ concept.} The $L_2$ theory of the slice, which was developed by Dunkl~\cite{Dunkl76,Dunkl79}, is intimately related to representations of $S_n$. In Section~\ref{sec:harmonic} we provide an elementary and self-contained approach to this theory. 
The $L_2$ theory is used to control the norm of low-degree harmonic functions with respect to several different measures, in particular by using a two-sided Poincar\'e inequality. This, in turn, allows us to introduce the main probabilistic argument. 

\paragraph*{Novel proof ingredients} The basic idea of the proof is to use the coupling method by showing that the distribution of $f$ on different slices of the cube is almost identical, as long as the slices are of distance at most roughly $\sqrt{n}$ (here the distance between the $k$th slice and the $\ell$th slice is $|k-\ell|$). In fact, for our proof to work we crucially need to allow distances which are somewhat larger than $\sqrt{n}$.

To construct the coupling, we use a uniform random increasing path to couple level $\ell$ to level $k$ above or below it. The main novel technique is representing the difference between the two levels as a difference of two martingales. Such representations have been used before in the analysis of {\em stationary reversible} Markov chains in Banach spaces~\cite{NPSS}, and even earlier in the analysis of stochastic integrals~\cite{LZ}. However, all previous decompositions were for stationary reversible chains, while ours is neither. Our novel representation of the differences might be of interest in other applications. We outline this argument in the introduction to Section~\ref{sec:invariance}, and carry it out in the rest of the section.


\paragraph*{Applications}
Except for the natural interpretation of Theorem~\ref{thm:invariance_intro} in terms of distinguishing between distributions, it can be used to prove results in extremal combinatorics in the same way the main result of~\cite{FKMW} is used. For example, in Proposition~\ref{pro:kindler-safra} we give a proof of the Kindler--Safra theorem on the slice, first proved in~\cite{FKMW}. 

\subsection{Influences, symmetric functions and circuits}
We prove a few other results that give partial answers to Question~\ref{q:1}: 
\begin{itemize}

\item Using direct computation of the total variation distance we prove the following theorem:
\begin{theorem:invariance-few-coordinates}
 Let $f$ be a function on $\{0,1\}^n$  depending on $o(n)$ coordinates and satisfying $\|f\|_\infty \leq 1$. Then 
 \[
  |\EE_{\nu}[f] - \EE_{\mu}[f]| = o(1).
 \]
\end{theorem:invariance-few-coordinates}

\item 
We prove that symmetric functions cannot be approximated by functions whose total influence is~$o(\sqrt{n})$ (see Proposition~\ref{prop:balanced-symmetric} for a more general formulation): 

\begin{proposition} \label{prop:balanced-symmetric_intro}
 There exists a constant $\delta > 0$ such that if $f$ is a symmetric Boolean function such that $\frac{1}{3} \leq \EE_{\mu_p}[f] \leq \frac{2}{3}$ then $\Pr_{\mu_p}[f \neq g] > \delta$ for every Boolean function $g$ satisfying $\Inf[g] = o(\sqrt{n})$.
\end{proposition}
Since it is well-known~\cite{Bop}, based on arguments from \cite{LMN,Has}, that a depth~$d$ size~$m$ circuit has total influence at most $O((\log m)^{d-1})$, our result immediately implies circuit lower bounds for such functions. However, much better bounds are known, see e.g.~\cite{Tal} for general symmetric functions and~\cite{ODonnellWimmer} for the case of the majority function.
Nevertheless, Proposition~\ref{prop:balanced-symmetric_intro} is more general as it holds for functions $f$ that are not necessarily the majority function and for functions $g$ that are not necessarily in $\mathsf{AC}^0$. Moreover, the proof of Proposition~\ref{prop:balanced-symmetric_intro} is based on a new and simple probabilistic argument. 

\end{itemize} 

\subsection{Other results comparing the cube to the slice}
Question~\ref{q:1} is not a new question. So we conclude the introduction with a few classical results related to this question: 
\begin{itemize}
\item The limiting behaviors of the partial sum $W(s) = \frac{1}{\sqrt{s}} \sum_{i=1}^s (x_i - \frac{1}{2})$ as $s \to \infty$ under the cube and the slice measures are well-studied. It is well-known that under the cube measure $W(s)$ converges to Brownian motion, while under the slice measure it converges to a Brownian bridge.
\item It is well-known that the partial sums $W(s)$ are at least as concentrated in the slice as they are in the cube~\cite{Hoe}. 
\item It is well-known that Lipschitz functions of the random variables $x_1,\ldots,x_n$ are concentrated both in the cube and in the slice. 
The results for the slice follow from the hypercontractive estimates by Lee and Yau~\cite{LeeYau}. These are also needed in our proofs. 
\end{itemize} 

\paragraph*{Paper organization} Following Section~\ref{sec:definitions}, which contains several useful definitions, the paper is composed of four major parts:

\begin{enumerate}
\item \textbf{Harmonic analysis on the slice (Section~\ref{sec:harmonic}).} In this part we provide a self-contained introduction to harmonic analysis on the slice. Most of the results proved in this section are not new, but our proofs are novel and elementary.
\item \textbf{Invariance principles (Sections~\ref{sec:invariance}--\ref{sec:multilinear}).} In this part we prove several invariance principles. Section~\ref{sec:invariance} contains the major result of this paper, an invariance principle for Lipschitz functions. This theorem applies to polynomials of degree $o(\sqrt{n})$, a condition discussed in Section~\ref{sec:high-degree}. Section~\ref{sec:approx-bool} extends the invariance principle to the non-Lipschitz function $\varphi(x) = (|x|-1)^2$, appearing in many applications. As an application of this invariance principle, we give an alternative proof of a non-optimal Kindler--Safra theorem for the slice, first proved in~\cite{FKMW} (an optimal version of the theorem has since been proved by Keller and Klein~\cite{KellerKlein}). Both of these results are for harmonic polynomials. The case of non-harmonic polynomials is tackled in Section~\ref{sec:multilinear}, in which an analog of the invariance principle is proved for arbitrary multilinear polynomials.
\item \textbf{Minor invariance principles (Sections~\ref{sec:few-coordinates}--\ref{sec:low-influence}).} In this part we prove two simpler invariance principles for restricted classes of functions: bounded functions depending on $o(n)$ coordinates (Section~\ref{sec:few-coordinates}) and functions whose total influence is $o(\sqrt{n})$ (Section~\ref{sec:low-influence}).
\item \textbf{Connections (Sections~\ref{sec:decomposition}--\ref{sec:harmonic-use}).} In this final part, we provide some connections to other mathematical fields. Section~\ref{sec:decomposition} provides a representation-theoretic angle on the slice. This point of view serves to demystify some of the important properties of the slice proved in Section~\ref{sec:harmonic}. Section~\ref{sec:harmonic-use} justifies our use of harmonic polynomials to ``lift'' functions from slice to cube. Two justifications are given: a Sperner-theoretic justification (Section~\ref{sec:harmonic-algebra}) and a representation-theoretic justification (Section~\ref{sec:harmonic-rep}).

Section~\ref{sec:harmonic-algebra} also contains an alternative proof of the important property that every function on the slice has a canonical representation as a harmonic multilinear polynomial. In fact, we prove a stronger result, Blekherman's theorem~\cite{Blekherman}, which is crucial in Section~\ref{sec:multilinear}. A corollary of Blekherman's theorem is also used in the proof of the Kindler--Safra theorem in Section~\ref{sec:approx-bool}.
\end{enumerate}


\subsection*{Acknowledgements}

Both authors would like to thank the referees for their extensive and helpful comments.

Y.F. would like to mention that this material is based upon work supported by the National Science Foundation under agreement No.~DMS-1128155. Any opinions, findings and conclusions or recommendations expressed in this material are those of the authors, and do not necessarily reflect the views of the National Science Foundation. Part of the work was done while at the Institute for Advanced Study, Princeton, NJ. The research was also funded by ISF grant 1337/16. The author is a Taub Fellow, and supported by the Taub Foundations. 

E.M. would like to acknowledge the support of the following grants: NSF grants DMS 1106999 and CCF 1320105, DOD ONR grant N00014-14-1-0823, and grant 328025 from the Simons Foundation.

\section{Definitions} \label{sec:definitions}

\paragraph*{Notation} We employ the falling power notation $n^{\underline{k}} = n(n-1)\cdots(n-k+1)$. The notation $\charf{E}$ equals $1$ if the condition $E$ holds, and $0$ otherwise. The sign function is denoted $\sgn$. The \emph{$L_2$ triangle inequality} is $(a+b)^2 \leq 2(a^2+b^2)$ or its generalization $(\sum_{i=1}^n a_i)^2 \leq n\sum_{i=1}^n a_i^2$.

A monomial is \emph{squarefree} if it is not divisible by a square of a variable. (Thus there are $2^n$ squarefree monomials on $n$ variables.) A polynomial is \emph{multilinear} if all monomials are squarefree. A polynomial is \emph{homogeneous} if all monomials have the same total degree. The $d$th homogeneous part of a polynomial $f = \sum c_m m$, denote $f^{=d}$, is the sum of $c_m m$ over all monomials $m$ of total degree $d$. A polynomial $f$ over $x_1,\ldots,x_n$ is \emph{harmonic} if $\sum_{i=1}^n \frac{\partial f}{\partial x_i} = 0$.

A univariate function $f$ is \emph{$C$-Lipschitz} if $|f(x) - f(y)| \leq C|x-y|$. A function is \emph{Lipschitz} if it is $1$-Lipschitz.

The expectation and variance of a random variable are denoted $\EE,\VV$, and $\|\cdot\|$ denotes its $L_2$ norm $\|X\| = \sqrt{\EE[X^2]}$. To signify that expectation is taken with respect to a distribution $\alpha$, we write $\EE_\alpha[X]$, $\VV_\alpha[x]$, and $\|\cdot\|_\alpha$. A normal distribution with mean $\mu$ and variance $\sigma^2$ is denoted $\nor(\mu,\sigma^2)$. A binomial distribution with $n$ trials and success probability $p$ is denoted $\bin(n,p)$.

The symmetric group on $[n] = \{1,\ldots,n\}$ is denoted $S_n$. A distribution on $\RR^n$ is \emph{exchangeable} if it is invariant under the action of $S_n$ (that is, under permutation of the coordinates); a discrete distribution is exchangeable if the probability of $(x_1,\ldots,x_n)$ depends only on $x_1+\cdots+x_n$. For a function $f$ on $\RR^n$ and a permutation $\pi$, we define $f^\pi(x_1,\ldots,x_n) = f(x_{\pi(1)},\ldots,x_{\pi(n)})$. We compose permutations left-to-right, so that $(f^\alpha)^\beta = f^{\alpha\beta}$.


\paragraph*{Asymptotic notation} The notation $O(f)$ means a function $g$ such that $g \leq Cf$ for some positive constant $C > 0$. In particular, $g = O(f)$ means that $g \leq Cf$ for some positive constant $C > 0$. The notation $\Omega(f)$ means a function $g$ such that $g \geq cf$ for some positive constant $c > 0$. The notation $\Theta(f)$ means a function $g$ such that $cf \leq g \leq Cf$ for some positive constants $C \ge c > 0$.

We write $g(n) = o(f(n))$ if $\lim_{n \to \infty} g(n)/f(n) = 0$. The quantity $n$ should be clear from context (in some cases, for example, it is $p(1-p)n$). We write $g(n) = \omega(f(n))$ if $\lim_{n \to \infty} g(n)/f(n) = \infty$.

The notation $O_a(f)$ means a function $g$ such that $g \leq C(a)f$ for some everywhere positive function $C(a)$. We similarly define $O_{a,b}(f)$, $\Omega_a(f)$, and so on. For example, $n/p(1-p) = O_p(n)$.

\paragraph*{The slice} The $n$-dimensional Boolean cube is the set $\{0,1\}^n$. For an integer $0 \leq k \leq n$, the $k$th slice of the $n$-dimensional Boolean cube is the set
\[
 \binom{[n]}{k} = \left\{ (x_1,\ldots,x_n) \in \{0,1\}^n : \sum_{i=1}^n x_i = k \right\}.
\]

\paragraph*{Probability measures} Our work involves two main probability measures, where $n$ is always understood:
\begin{itemize}
 \item The uniform measure on the slice $\binom{[n]}{k}$ is $\nu_k$.
 \item The product measure $\mu_p$ on the Boolean cube is given by $\mu_p(x) = p^{\sum_i x_i} (1-p)^{\sum_i (1-x_i)}$.
\end{itemize}
Note that $\nu_k,\mu_{k/n}$ have the same marginal distributions.

\section{Harmonic functions} \label{sec:harmonic}

A basic and easy result states that every function on $\{-1,1\}^n$ has a unique representation as a multilinear polynomial, known as the \emph{Fourier expansion}. It is easy to see that a multilinear polynomial has the same mean and variance with respect to the uniform measure on $\{-1,1\}^n$ and with respect to the standard $n$-dimensional Gaussian measure. In this section we describe the corresponding canonical representation on the slice, due to Dunkl~\cite{Dunkl76,Dunkl79} and elaborated by Srinivasan~\cite{Srinivasan}, Filmus~\cite{F} and Ambainis et al.~\cite{ABRW}. Most of the results in this section are already known, though the proofs presented in this section are novel. A possible exception is the two-sided Poincar\'e inequality for derivatives, Lemma~\ref{lem:poincare-derivative}.

The canonical representation of functions on the slice is described in Subsection~\ref{sec:canonical-representation}. We decompose this representation into orthogonal parts in Subsection~\ref{sec:orthogonality}, where we also deduce that the mean and variance of a low-degree function is similar on the slice and on the Boolean cube. The analog of the Poincar\'e inequality is proved in Subsection~\ref{sec:poincare} alongside results of a similar flavor. Finally, we prove that degree is subadditive with respect to multiplication, and monotone with respect to substitution, in Subsection~\ref{sec:cometric}. 

\subsection{Canonical representation} \label{sec:canonical-representation}

Every function on the slice $\binom{[n]}{k}$ can be represented as a multilinear polynomial, but this representation is not unique. However, as found by Dunkl~\cite{Dunkl76,Dunkl79}, we can make it unique by demanding that it be \emph{harmonic} in the sense of the following definition.

\begin{definition} \label{def:harmonic}
A polynomial $P$ over $x_1,\ldots,x_n$ is \emph{harmonic} if
\[ \sum_{i=1}^n \frac{\partial P}{\partial x_i} = 0. \]
In other words, $P$ is harmonic if $\Delta P = 0$, where $\Delta$ is the differential operator $\sum_{i=1}^n \frac{\partial}{\partial x_i}$.
\end{definition}

\begin{definition} \label{def:basic-function}
A \emph{basic function} is a (possibly empty) product of factors $x_i - x_j$ on disjoint indices. A function is \emph{elementary} if it is a linear combination of basic functions.
\end{definition}

Most, but not all, of the harmonic polynomials we consider will be multilinear. In particular, notice that all elementary functions are multilinear. Here are some basic properties of harmonic polynomials.

\begin{lemma} \label{lem:harmonic-algebra}
The set of harmonic polynomials is an algebra of polynomials, and is closed under partial derivatives, under permutations of the coordinates, and under taking homogeneous parts. In particular, all elementary functions are harmonic.
\end{lemma}
\begin{proof}
Suppose $f,g$ are harmonic. Then $\Delta (\alpha f + \beta g) = \alpha \Delta f + \beta \Delta g = 0$; $\Delta(fg) = f\Delta g + g\Delta f = 0$; $\Delta \frac{\partial f}{\partial x_i} = \frac{\partial \Delta f}{\partial x_i} = 0$; and $\Delta(f^\pi) = (\Delta f)^\pi = 0$. Finally, since $\Delta(\sum_{d=0}^n f^{=d}) = \sum_{d=0}^n \Delta f^{=d}$ and $\Delta f^{=d}$ is homogeneous of degree $d-1$, we see that $\Delta f^{=d} = 0^{=d-1} = 0$.
\end{proof}

\begin{lemma} \label{lem:harmonic-alt}
A polynomial $f$ is harmonic if and only if for all $x_1,\ldots,x_n,c$ we have
\[ f(x_1+c,\ldots,x_n+c) = f(x_1,\ldots,x_n). \]
In particular, if $f^\pi = f$ for all $\pi \in S_n$ then $\deg f = 0$.
\end{lemma}
\begin{proof}
Given $x_1,\ldots,x_n$, define a function
\[ \phi(x_1,\ldots,x_n,c) = f(x_1+c,\ldots,x_n+c). \]
The chain rule implies that $\frac{\partial \phi}{\partial c} = \Delta f$. Hence $\Delta f = 0$ iff $\phi$ is independent of $c$.

If $f$ moreover satisfies $f^\pi = f$ for all $\pi \in S_n$, then $f$ is a symmetric polynomial, and so $f = \phi(x_1 + \cdots + x_n)$ for some univariate polynomial $\phi$. The polynomial $\phi$ satisfies $\phi(x) = \phi(x + nc)$ for all $c$, and so both $\phi$ and $f$ must be constant.
\end{proof}

Our proofs of Theorem~\ref{thm:harmonic-representation} and Lemma~\ref{lem:harmonic-elementary} proceed by analyzing ``derivatives'' of the form $f - f^{(i\;j)}$, and using the following simple lemma to conclude the same property for $f$ itself.

\begin{lemma} \label{lem:symmetrization}
Let $X$ be a ring of characteristic zero with an $S_n$~action, and let $V \subseteq X$ be a vector space satisfying the following properties:
\begin{enumerate}[(a)]
\item If $f \in X$ then $f - f^{(i\;j)} \in V$ for all $i \neq j$.
\item If $f \in X$ then 	$\symm(f) := \EE_\pi[f^\pi] \in V$.
\end{enumerate}
Then $V = X$.	
\end{lemma}
\begin{proof}
 Every permutation $\pi \in S_n$ can be written as $\pi = (i_1\;j_1) \cdots (i_\ell\;j_\ell)$. Since $V$ is a vector space, the first property shows that
\[
 f - f^\pi = \sum_{t=1}^\ell f^{(i_1\;j_1) \cdots (i_{t-1}\;j_{t-1})} - f^{(i_1\;j_1) \cdots (i_t\;j_t)} \in V.
\]
Taking average over all $\pi \in S_n$, we get that $f - \symm(f) \in V$. The second property now implies that $f \in V$.
\end{proof}

Our first theorem states that every function on the slice has a unique representation as a harmonic multilinear polynomial of degree at most $\min(k,n-k)$.

\begin{theorem} \label{thm:harmonic-representation}
Let $0 \leq k \leq n$. Every function on the slice $\binom{[n]}{k}$ has a unique representation as a harmonic multilinear polynomial of degree at most $\min(k,n-k)$.
\end{theorem}
\begin{proof}
We use the notation $f \equiv g$ (read $f$ \emph{agrees with} $g$) to denote that $f$ agrees with $g$ on the slice $\binom{[n]}{k}$.

We start by proving, by induction on $\min(k,n-k)$, that every function on a slice $\binom{[n]}{k}$ has some representation as a harmonic multilinear polynomial. The base cases, $k=0$ and $k=n$, are trivial, since in these cases all functions are constant.

Consider now a function $f$ on the slice $\binom{[n]}{k}$, where $0 < k < n$. Define a function $g$ on the slice $\binom{[n-2]}{k-1}$ by $g(x_1,\ldots,x_{n-2}) = f(x_1,\ldots,x_{n-2},0,1) - f(x_1,\ldots,x_{n-2},1,0)$. By the induction hypothesis, $g$ can be represented as a harmonic multilinear polynomial $G \equiv g$ of degree at most $\min(k-1,(n-2)-(k-1)) = \min(k,n-k)-1$. Let $H = (x_n - x_{n-1}) G$. Note that $H$ is a multilinear polynomial of degree at most $\min(k,n-k)$, and by Lemma~\ref{lem:harmonic-algebra} it is harmonic. We claim that $f - f^{(n-1\;n)} \equiv H$ (recall that $f^{(n-1\;n)}$ is obtained from $f$ by permuting coordinates $n-1$ and $n$). Indeed, when $x_{n-1} = x_n$, both sides vanish, and when $(x_{n-1},x_n) = (0,1)$ or $(x_{n-1},x_n) = (1,0)$, this is true by definition (checking both cases separately). Thus $f - f^{(n-1\;n)}$ can be represented as a harmonic multilinear polynomial of degree at most $\min(k,n-k)$. For short, we say that $f - f^{(n-1\;n)}$ is \emph{representable}.

The same argument implies that $f - f^{(i\;j)}$ is representable for all $i,j$, satisfying the first property in Lemma~\ref{lem:symmetrization}, with $X$ consisting of all functions on the slice $\binom{[n]}{k}$, and $V$ consisting of all representable functions. Since $\symm(f) = \EE[f]$ is a constant and so representable, the lemma shows that all functions are representable.

It remains to prove that the representation is unique. To that end, it is enough to show that if $P$ is a harmonic multilinear polynomial of degree at most $\min(k,n-k)$ such that $P \equiv 0$ then $P = 0$. We prove this by induction on $\min(k,n-k)$. The base cases, $k=0$ and $k=n$, are trivial, since if $P$ is a constant polynomial agreeing with the zero function, then $P = 0$.

Suppose now that $P \equiv 0$ on the slice $\binom{[n]}{k}$ for some harmonic multilinear polynomial $P$ of degree at most $\min(k,n-k)$, where $0 < k < n$. Write $P = A + x_{n-1} B + x_n C + x_{n-1}x_n D$, where $A,B,C,D$ are polynomials over $x_1,\ldots,x_{n-2}$, and notice that $P - P^{(n-1\;n)} = (x_{n-1} - x_n)(B - C)$. Since $\Delta [(x_{n-1} - x_n)(B - C)] = (x_{n-1} - x_n) \Delta (B-C)$, we see that $B-C$ is harmonic. Considering the substitution $(x_{n-1},x_n) = (0,1)$, we see that $B-C$ is a harmonic multilinear polynomial of degree at most $\min(k,n-k)-1$ which agrees with $0$ on the slice $\binom{[n-2]}{k-1}$. The induction hypothesis implies that $B-C = 0$, and so $P = P^{(n-1\;n)}$.

The same argument implies that $P = P^{(i\;j)}$ for all $i,j$, and so $P = P^{\pi}$ for all $\pi$. Lemma~\ref{lem:harmonic-alt} implies that $P$ is constant. Since $P$ agrees with $0$, we must have $P = 0$, completing the proof.
\end{proof}

Similarly, we can prove that every harmonic multilinear polynomial is elementary.

\begin{lemma} \label{lem:harmonic-elementary}
A multilinear polynomial is harmonic iff it is elementary. In particular, a harmonic multilinear polynomial over $x_1,\ldots,x_n$ has degree at most $n/2$.
\end{lemma}
\begin{proof}
Lemma~\ref{lem:harmonic-algebra} implies that every elementary polynomial is harmonic. We now prove that every multilinear harmonic polynomial is elementary by induction on $n$. If $n = 0$ then any polynomial is constant and so elementary. If $n = 1$ then any harmonic polynomial is constant and so elementary. Consider now any $n \geq 2$ and any harmonic multilinear polynomial $f$ over $x_1,\ldots,x_n$. For any $i \neq j$ we can write $f = A + x_i B + x_j C + x_ix_j D$, where $A,B,C,D$ don't involve $x_i,x_j$, so that $f - f^{(i\;j)} = (x_i - x_j)(B - C)$. Notice that $f - f^{(i\; j)}$ is harmonic, and since $\Delta [(x_i-x_j)(B-C)] = (x_i-x_j)\Delta(B-C)$, we see that $B-C$ is also harmonic. By induction, we get that $B-C$ is elementary, and so $f - f^{(i\;j)} = (x_i-x_j)(B-C)$ is elementary.

This shows that the first property in Lemma~\ref{lem:symmetrization} is satisfied, where $X$ consists of all multilinear harmonic polynomials, and $V$ consists of all elementary polynomials. Lemma~\ref{lem:harmonic-alt} shows that $\symm(f)$ is constant and so elementary, and so the second property is satisfied as well. Lemma~\ref{lem:symmetrization} therefore implies that every multilinear harmonic polynomial is elementary.
\end{proof}

The algebraically inclined reader may enjoy the following reformulation of Theorem~\ref{thm:harmonic-representation}.

\begin{corollary} \label{cor:harmonic-representation}
 Let $0 \leq k \leq n$, and fix a field $\FF$ of characteristic zero. Consider the polynomial ideal
\[
 I = \left\langle x_1^2 - x_1, \ldots, x_n^2 - x_n, \sum_{i=1}^n x_i - k \right\rangle = I\bigl(\binom{[n]}{k}\bigr),
\]
 where we think of the slice $\binom{[n]}{k}$ as an affine variety.
 Then $\FF[x_1,\ldots,x_n]/I$ is isomorphic to the ring of harmonic multilinear polynomials of degree at most $\min(k,n-k)$ over $x_1,\ldots,x_n$.
\end{corollary}

We discuss this reformulation in Section~\ref{sec:harmonic-algebra}, giving an alternative proof of the theorem.

We conclude this section by computing the dimension of the space of harmonic multilinear polynomials of given degree.

\begin{corollary} \label{cor:dimension}
 Let $d \leq n/2$. Then
\begin{enumerate}[(a)]
 \item The linear space $H_{\leq d}$ of harmonic multilinear polynomials of degree at most $d$ has dimension $\binom{n}{d}$.
 \item The linear space $H_d$ of harmonic multilinear polynomials which are homogeneous of degree $d$ has dimension $\binom{n}{d} - \binom{n}{d-1}$, where $\binom{n}{-1} = 0$.
\end{enumerate}
\end{corollary}
\begin{proof}
 The first item follows directly from Theorem~\ref{thm:harmonic-representation}, taking $k = d$. Since $\dim H_{\leq d} = \sum_{e=0}^d \dim H_e$, the second item follows from the first.
\end{proof}

For another proof, see the proof of Theorem~\ref{thm:harmonic-representation-2}.

\subsection{Orthogonality of homogeneous parts} \label{sec:orthogonality}

As stated in the introduction to this section, multilinear polynomials enjoy the useful property of having the same mean and variance with respect to all product measures with fixed marginal mean and variance. The corresponding property for harmonic multilinear polynomials is stated in the following theorem, which also follows from the work of the first author~\cite{F}. A representation-theoretic proof of the theorem is outlined in Section~\ref{sec:coarse-decomposition}.

\begin{theorem} \label{thm:spherical}
Let $f,g$ be homogeneous harmonic multilinear polynomials of degree $d_f,d_g$, respectively, and let $\alpha$ be an exchangeable measure. If $d_f \neq d_g$ then $\EE_\alpha[fg] = 0$. If $d_f = d_g = d$ then there exists a constant $C_{f,g}$ independent of $\alpha$ such that
\[
 \EE_\alpha[fg] = C_{f,g} \EE_\alpha[(x_1-x_2)^2 \cdots (x_{2d-1}-x_{2d})^2].
\]
\end{theorem}
\begin{proof}
Let $h = fg$, and note that Lemma~\ref{lem:harmonic-algebra} implies that $h$ is harmonic.
Let $H = \EE_{\pi \in S_n}[h^\pi]$, and note that $H$ is also harmonic.
Since $\alpha$ is exchangeable, $\EE_\alpha[fg] = \EE_\alpha[H]$.

We first note that $H$ is a linear combination of the functions
\[ b_t = \EE_{\pi \in S_n} x_{\pi(1)}^2 \cdots x_{\pi(t)}^2 x_{\pi(t+1)} \cdots x_{\pi(d_f+d_g-t)}, \]
say $H = \sum_t \beta_t(f,g) b_t$. 

Suppose first that $d_f \neq d_g$. It is easy to check that $\EE_{\nor(0,1)}[b_t] = 0$ for all $t$, and so $\EE_{\nor(0,1)}[H] = 0$. Since $H$ is harmonic, Lemma~\ref{lem:harmonic-alt} implies that $\EE_{\nor(\mu,1)}[H]$ doesn't depend on $\mu$, and so
\[
 0 = \EE_{\nor(0,1)}[H] = \EE_{\nor(\mu,1)}[H] = \sum_{t=0}^{\min(d_f,d_g)} \beta_t(f,g) (1 + \mu^2)^t \mu^{d_f+d_g-2t}.
\]
The polynomial $P_t(\mu) = (1 + \mu^2)^t \mu^{d_f+d_g-2t}$ has minimal degree term $\mu^{d_f+d_g-2t}$, and so the polynomials $P_0,\ldots,P_{\min(d_f,d_g)}$ are linearly independent. This shows that $\beta_t(f,g) = 0$ for all $t$, and so $H = 0$. In particular, $\EE_\alpha[fg] = \EE_\alpha[H] = 0$.

When $d_f = d_g = d$, it is still true that $\EE_{\nor(0,1)}[b_t] = 0$ for all $t < d$, but now $\EE_{\nor(0,1)}[b_d] = 1$. Therefore the same argument as before shows that
\[
 \beta_d(f,g) = \sum_{t=0}^d \beta_t(f,g) P_t.
\]
The linear independence of $P_0,\ldots,P_d$ implies that $\beta_t(f,g)/\beta_d(f,g)$ depends only on $d$, and so
\[
 H = \beta_d(f,g) \sum_{t=0}^d \kappa_t b_t
\]
 for constants $\kappa_0,\ldots,\kappa_t$ depending only on $d$.
 In particular,
\[
 \EE_\alpha[fg] = \EE_\alpha[H] = \beta_d(f,g) \EE_\alpha \left[ \sum_{t=0}^d \kappa_t b_t \right].
\]
Applying the same argument to $f' = g' = (x_1-x_2)\cdots(x_{2d-1}-x_{2d})$, we obtain
\[
 \EE_\alpha[(x_1-x_2)^2\cdots(x_{2d-1}-x_{2d})^2] = \beta'_d \EE_\alpha \left[ \sum_{t=0}^d \kappa_t b_t \right],
\]
 where $\beta'_d$ depends only on $d$. The theorem follows with $C_{f,g} = \beta_d(f,g)/\beta'_d$.
\end{proof}

\begin{corollary} \label{cor:approximate-norm}
Let $f$ be a harmonic multilinear polynomial of degree at most $d$ with constant coefficient $f^{=0}$. Suppose that $\alpha,\beta$ are exchangeable measures and $C > 0$ is a constant that for $t \leq d$ satisfies
\[
 \EE_\alpha[(x_1-x_2)^2 \cdots (x_{2t-1}-x_{2t})^2] \leq C \EE_\beta[(x_1-x_2)^2 \cdots (x_{2t-1}-x_{2t})^2].
\]
Then $\EE_\alpha[f] = f^{=0}$, $\|f\|^2_\alpha \leq C\|f\|^2_\beta$, and $\VV_\alpha[f] \leq C\VV_\beta[f]$.
\end{corollary}
\begin{proof}
Write $f = \sum_{t=0}^d f^{=t}$, where $f^{=t}$ is homogeneous of degree $t$. Theorem~\ref{thm:spherical} implies that $f^{=t_1},f^{=t_2}$ are orthogonal with respect to all exchangeable measures. This implies that $\EE_\alpha[f] = f^{=0}$ and
\[
 \|f\|^2_\alpha = \sum_{t=0}^d \|f^{=t}\|^2_\alpha.
\]
The theorem also implies that for some $K_0,\ldots,K_d$ we moreover have
\[
 \|f\|^2_\alpha = \sum_{t=0}^d K_t \EE_\alpha[(x_1-x_2)^2 \cdots (x_{2t-1}-x_{2t})^2] \leq C \sum_{t=0}^d K_t \EE_\beta[(x_1-x_2)^2 \cdots (x_{2t-1}-x_{2t})^2] = C\|f\|^2_\beta.
\]
Finally, since $f - f^{=0}$ is also harmonic, we deduce that $\VV[f]_\alpha \leq C \VV[f]_\beta$.
\end{proof}

The following lemma computes $\EE[(x_1-x_2)^2 \cdots (x_{2d-1}-x_{2d})^2]$ for the measures $\nu_k,\mu_p$.

\begin{lemma} \label{lem:norms}
Let $p = k/n$. We have
\begin{align*}
 \EE_{\nu_k}[(x_1-x_2)^2 \cdots (x_{2d-1}-x_{2d})^2] &= 2^d \frac{k^{\underline{d}} (n-k)^{\underline{d}}}{n^{\underline{2d}}} = (2p(1-p))^d \left(1 \pm O\left(\frac{d^2}{p(1-p)n}\right)\right), \\
 \EE_{\mu_p}[(x_1-x_2)^2 \cdots (x_{2d-1}-x_{2d})^2] &= (2p(1-p))^d.
\end{align*}
\end{lemma}

This straightforward computation appears in~\cite[Theorem 4.1]{F} and~\cite[Lemma 2.9]{FKMW}. Qualitatively, the lemma states that the norm of a low-degree basic function is similar in both $\nu_k$ and $\mu_p$. This is not surprising: the coordinates in the slice are almost independent, and a low-degree basic function depends only on a small number of them.

\subsection{Poincar\'e inequalities} \label{sec:poincare}

We proceed by proving the so-called two-sided Poincar\'e inequality, starting with the following fact.

\begin{lemma} \label{lem:transpositions}
Let $f$ be a harmonic multilinear polynomial. Then
\[
 \sum_{i<j} f^{(i\;j)} = \sum_{d=0}^{n/2} \left[ \binom{n}{2} - d(n-d+1) \right] f^{=d},
\]
where $f^{=d}$ is the $d$th homogeneous part of $f$.
\end{lemma}
\begin{proof}
In view of Lemma~\ref{lem:harmonic-elementary}, $f$ is elementary. Therefore it is enough to consider the case that $f$ is a basic function of some degree $d$, say $f = \prod_{t=1}^d (x_{a_t} - x_{b_t})$. We split the transpositions into four kinds. The first kind is transpositions which do not involve any $a_t,b_t$. There are $\binom{n-2d}{2}$ such transpositions $(i\;j)$, and they all satisfy $f^{(i\;j)} = f$. The second kind is transpositions of the form $(a_t\;b_t)$. There are $d$ of these, and they satisfy $f^{(a_t\;b_t)} = -f$. The third kind is transpositions of the form $(a_t\;j)$ or $(b_t\;j)$, where $j \neq a_s,b_s$. There are $d(n-2d)$ pairs of these. Since
\[ (x_{a_t} - x_{b_t})^{(a_t\;j)} + (x_{a_t} - x_{b_t})^{(b_t\;j)} = x_{a_t} - x_{b_t}, \]
each such pair contributes one multiple of $f$.
The fourth kind is transpositions involving two pairs $(a_t\;b_t),(a_s\;b_s)$, which we group in the obvious way into $\binom{d}{2}$ quadruples.
Direct computation shows that
\[ \sum_{\pi \in \{(a_t\;a_s),(a_t\;b_s),(b_t\;a_s),(b_t\;b_s)\}} [(x_{a_t} - x_{b_t})(x_{a_s} - x_{b_s})]^\pi = 2(x_{a_t} - x_{b_t})(x_{a_s} - x_{b_s}), \]
and so the contribution of each such quadruple is two multiples of $f$. In total, we obtain
\[ \sum_{i<j} f^{(i\;j)} = \left[ \binom{n-2d}{2} - d + d(n-2d) + 2\binom{d}{2} \right] f = \left[ \binom{n}{2} - d(n-d+1) \right] f. \qedhere \]
\end{proof}

\begin{lemma} \label{lem:poincare}
Let $f$ be a harmonic multilinear polynomial of degree at most $d$. Then with respect to any exchangeable measure,
\[
 n\VV[f] \leq \frac{1}{2} \sum_{i<j} \|f-f^{(i\;j)}\|^2 \leq d(n-d+1)\VV[f].
\]
\end{lemma}
\begin{proof}
Write $f = \sum_{t=0}^d f^{=t}$. Theorem~\ref{thm:spherical} implies that the homogeneous parts are orthogonal. We have
\[
 \frac{1}{2} \sum_{i<j} \|f-f^{(i\;j)}\|^2 = \binom{n}{2} \|f\|^2 - \big\langle f, \sum_{i<j} f^{(i\;j)} \big\rangle.
\]
Lemma~\ref{lem:transpositions} implies that
\[
 \frac{1}{2} \sum_{i<j} \|f-f^{(i\;j)}\|^2 = \sum_{t=0}^d t(n-t+1) \|f^{=t}\|^2.
\]
The lemma now follows from the observation that for $1 \leq t \leq d$ we have $n \leq t(n-t+1) \leq d(n-d+1)$, since $t(n-t+1)$ is increasing for $t \leq (n+1)/2$.
\end{proof}

Finally, we prove another two-sided Poincar\'e inequality, this time for derivatives. We start with the following surprising corollary of Theorem~\ref{thm:spherical}.

\begin{lemma} \label{lem:spherical-derivative}
Let $f,g$ be homogeneous harmonic multilinear polynomials of degree $d$. Then for any exchangeable measure $\alpha$,
\[
 \frac{\sum_{i=1}^n \EE_\alpha \left[ \frac{\partial f}{\partial x_i} \frac{\partial g}{\partial x_i} \right]}{\EE_\alpha [fg]} = 2d \frac{\EE_\alpha [(x_1-x_2)^2\cdots(x_{2d-3}-x_{2d-2})^2]}{\EE_\alpha [(x_1-x_2)^2\cdots(x_{2d-1}-x_{2d})^2]}.
\]
\end{lemma}
\begin{proof}
Since the functions $\frac{\partial f}{\partial x_i},\frac{\partial g}{\partial x_i}$ are harmonic, Theorem~\ref{thm:spherical} implies that there are constants $C_{f,g},D_{f,g}$ such that
\[
 \frac{\sum_{i=1}^n \EE_\alpha \left[ \frac{\partial f}{\partial x_i} \frac{\partial g}{\partial x_i} \right]}{\EE_\alpha [fg]} = \frac{D_{f,g}}{C_{f,g}} \frac{\EE_\alpha [(x_1-x_2)^2\cdots(x_{2d-3}-x_{2d-2})^2]}{\EE_\alpha [(x_1-x_2)^2\cdots(x_{2d-1}-x_{2d})^2]}.
\]
 Here $D_{f,g} = \sum_{i=1}^n C_{\frac{\partial f}{\partial x_i},\frac{\partial g}{\partial x_i}}$. We can evaluate the ratio $D_{f,g}/C_{f,g}$ by considering the distribution $\alpha = \nor(0,I_n)$:
\[
 \frac{\sum_{i=1}^n \EE_{\nor(0,I_n)} \left[ \frac{\partial f}{\partial x_i} \frac{\partial g}{\partial x_i} \right]}{\EE_{\nor(0,I_n)} [fg]} = \frac{D_{f,g}}{C_{f,g}} \frac{2^{d-1}}{2^d} = \frac{D_{f,g}}{2C_{f,g}}.
\]
 On the other hand, with respect to $\nor(0,I_n)$ we have
\[
 \sum_{i=1}^n \EE\left[ \frac{\partial f}{\partial x_i} \frac{\partial g}{\partial x_i} \right] = \sum_{i=1}^n \sum_{S \ni i} \hat{f}(S) \hat{g}(S) = \sum_S |S| \hat{f}(S) \hat{g}(S) = d \EE[fg].
\]
 We conclude that $D_{f,g}/(2C_{f,g}) = d$, and so $D_{f,g}/C_{f,g} = 2d$.
\end{proof}

We deduce the following two-sided Poincar\'e inequality.

\begin{lemma} \label{lem:poincare-derivative}
 Let $f$ be a harmonic multilinear polynomial of degree $d$, and let $\alpha$ be an exchangeable measure. Suppose that for $1 \leq t \leq d$ we have
\[ m \leq 2t \frac{\EE_\alpha [(x_1-x_2)^2\cdots(x_{2t-3}-x_{2t-2})^2]}{\EE_\alpha [(x_1-x_2)^2\cdots(x_{2t-1}-x_{2t})^2]} \leq M. \]
Then also
\[ m\VV[f] \leq \sum_{i=1}^n \left\|\frac{\partial f}{\partial x_i}\right\|^2 \leq M\VV[f]. \]
\end{lemma}
\begin{proof}
 Write $f = \sum_{t=0}^d f^{=t}$, and notice that $\frac{\partial f}{\partial x_i} = \sum_{t=1}^d \frac{\partial f^{=t}}{x_i}$, the latter sum being a decomposition into homogeneous parts. Theorem~\ref{thm:spherical} implies that the homogeneous parts are orthogonal with respect to $\alpha$, and so
\[
 \sum_{i=1}^n \left\|\frac{\partial f}{\partial x_i}\right\|^2 = \sum_{t=1}^d \sum_{i=1}^n \left\|\frac{\partial f^{=t}}{\partial x_i}\right\|^2.
\]
 Using Lemma~\ref{lem:spherical-derivative} we can upper bound
\[
 \sum_{t=1}^d \sum_{i=1}^n \left\|\frac{\partial f^{=t}}{\partial x_i}\right\|^2 \leq M \sum_{t=1}^d \|f^{=t}\|^2 = M \VV[f].
\]
 The lower bound is obtained in the same way.
\end{proof}

The following lemma computes $m,M$ for the measures $\nu_k,\mu_p$.

\begin{lemma} \label{lem:norm-ratios}
Let $p = k/n$. We have
\begin{align*}
 2d \frac{\EE_{\nu_k}[(x_1-x_2)^2 \cdots (x_{2d-3}-x_{2d-2})^2]}{\EE_{\nu_k}[(x_1-x_2)^2 \cdots (x_{2d-1}-x_{2d})^2]} &= d \frac{(n-2d+2)(n-2d+1)}{(k-d+1)(n-k-d+1)} \\ &= \frac{d}{p(1-p)} \left(1 \pm O\left( \frac{d}{p(1-p)n} \right)\right), \\
 2d \frac{\EE_{\mu_p}[(x_1-x_2)^2 \cdots (x_{2d-3}-x_{2d-2})^2]}{\EE_{\mu_p}[(x_1-x_2)^2 \cdots (x_{2d-1}-x_{2d})^2]} &= \frac{d}{p(1-p)}.
\end{align*}
\end{lemma}

The proof is a straightforward corollary of Lemma~\ref{lem:norms}.

\subsection{Cometric property} \label{sec:cometric}

Theorem~\ref{thm:harmonic-representation} states that every function on the slice $\binom{[n]}{k}$ can be represented as a harmonic multilinear polynomial. The following result shows that if the original function can be represented as a polynomial of degree $d$, then its harmonic representation has degree at most $d$.

\begin{lemma} \label{lem:harmonic-degree}
 Let $0 \leq k \leq n$. Let $f$ be a polynomial of degree $d$ on the variables $x_1,\ldots,x_n$. The unique harmonic multilinear polynomial agreeing with $f$ on the slice $\binom{[n]}{k}$ has degree at most $d$.
\end{lemma}
\begin{proof}
 We can assume without loss of generality that $f$ is multilinear, since on the slice $x_i^2 = x_i$. Hence it suffices to prove the theorem for the monomial $f = x_1 \cdots x_d$. Theorem~\ref{thm:harmonic-representation} states that there is a unique harmonic polynomial $g$ of degree at most $\min(k,n-k)$ agreeing with $f$ on the slice $\binom{[n]}{k}$. If $d \geq \min(k,n-k)$ then the lemma is trivial, so assume that $d < \min(k,n-k)$.
 
 Decompose $g$ into its homogeneous parts: $g = g^{=0} + \cdots + g^{=\min(k,n-k)}$. We will show that if $e$ satisfies $d < e \leq \min(k,n-k)$ and $h$ is a basic function of degree $e$ then $\langle f, h \rangle = 0$, where the inner product is with respect to the uniform measure on the slice. Since $f$ agrees with $g$ on the slice, it follows that $\langle g, h \rangle = 0$. Lemma~\ref{lem:harmonic-elementary} shows that $g^{=e}$ is a linear combination of basic functions of degree $e$ (since basic functions are homogeneous), and so $\langle g, g^{=e} \rangle = 0$. Theorem~\ref{thm:spherical} implies that $g^{=e} = 0$. Since this is the case for all $e > d$, we conclude that $\deg g \leq d$.
 
 Consider therefore a basic function $h = (x_{a_1} - x_{b_1})\cdots(x_{a_e} - x_{b_e})$. We have
\[
 \binom{n}{k} \langle f,h \rangle = \sum_{\substack{x \in \binom{[n]}{k}\colon \\ x_1=\cdots=x_d=1}} (x_{a_1} - x_{b_1}) \cdots (x_{a_e} - x_{b_e}).
\]
 Since $e > d$, for some $i$ it holds that $a_i,b_i > d$. The only non-zero terms in the sum (if any) are those for which $x_{a_i} \neq x_{b_i}$. We can match each term in which $x_{a_i} - x_{b_i} = 1$ with a term in which $x_{a_i} - x_{b_i} = -1$, obtained by switching $x_{a_i}$ and $x_{b_i}$. The two terms have opposite signs and so cancel. This shows that the entire sum vanishes, completing the proof.
\end{proof}

As a corollary, we deduce the so-called \emph{Q-polynomial} or \emph{cometric} property of the slice. For a function $f$ on a slice, denote by $\deg f$ the degree of its unique harmonic multilinear representation on the slice.

\begin{corollary} \label{cor:cometric}
 Let $f,g$ be functions on a slice. Then $\deg (fg) \leq \deg f + \deg g$.
\end{corollary}

Similarly, we have the following substitution property.

\begin{corollary} \label{cor:substitution}
 Let $f$ be a function on a slice, and let $g$ be the function on a smaller slice obtained by fixing one of the coordinates. Then $\deg g \leq \deg f$.
\end{corollary}

Both results are non-trivial since the product of two harmonic multilinear polynomials need not be multilinear, and substitution doesn't preserve harmonicity.

\section{Invariance principle} \label{sec:invariance}

In this section we prove an invariance principle showing that the distribution of a low-degree harmonic function on a slice $\binom{[n]}{k}$ is similar to its distribution on the Boolean cube with respect to the measure $\mu_{k/n}$. For convenience, we analyze the similarity in distribution via Lipschitz test functions, and derive similarity in more conventional terms as a corollary. The basic idea is to show that the distribution of a low degree function on a given slice $\binom{[n]}{k}$ is similar to its distribution on nearby slices $\binom{[n]}{\ell}$. If we can show this for all slices satisfying $|k-\ell| \leq B$ for some $B = \omega(\sqrt{n})$, then the invariance follows by decomposing the Boolean cube into a union of slices.

\paragraph*{The argument} Before giving the formal proof, we provide the intuition underlying the argument. Our argument concerns the following objects:
\begin{itemize}
 \item A harmonic multilinear polynomial $f$ of degree $d$ and unit norm. We think of $d$ as ``small''.
 \item A Lipschitz function $\varphi$.
 \item A slice $\binom{[n]}{pn}$. We think of $p$ as constant, though the argument even works for sub-constant $p$.
\end{itemize}
Our goal is to show that $\EE_{\mu_p}[\varphi(f)] \approx \EE_{\nu_{pn}}[\varphi(f)]$. The first step is to express $\mu_p$ as a mixture of $\nu_\ell$ for various $\ell$:
\[
 \EE_{\mu_p}[\varphi(f)] = \sum_{\ell=0}^n \binom{n}{\ell} p^\ell (1-p)^{n-\ell} \EE_{\nu_\ell}[\varphi(f)].
\]
Applying the triangle inequality, this shows that
\[
 |\EE_{\mu_p}[\varphi(f)] - \EE_{\nu_{pn}}[\varphi(f)]| \leq \sum_{\ell=0}^n \binom{n}{\ell} p^\ell (1-p)^{n-\ell} |\EE_{\nu_\ell}[\varphi(f)] - \EE_{\nu_{pn}}[\varphi(f)]|.
\]
In general we expect $|\EE_{\nu_\ell}[\varphi(f)] - \EE_{\nu_{pn}}[\varphi(f)]|$ to grow with $|\ell - pn|$, and our strategy is to consider separately slices close to $pn$, say $|pn - \ell| \leq \delta$, and slices far away from $pn$, say $|pn - \ell| > \delta$. We will bound the contribution of slices close to $pn$ directly. If $\delta$ is large enough then we expect the contribution of slices far away from $pn$ to be small, essentially since $\mu_p$ is concentrated on slices close to $pn$. For this argument to work, we need to choose $\delta$ so that $\delta = \omega(\sqrt{n})$.

It remains to bound $|\EE_{\nu_\ell}[\varphi(f)] - \EE_{\nu_{pn}}[\varphi(f)]|$ for $\ell$ close to $pn$. One strategy to obtain such a bound is to bound instead $|\EE_{\nu_s}[\varphi(f)] - \EE_{\nu_{s+1}}[\varphi(f)]|$ for various $s$, and use the triangle inequality. To this end, it is natural to consider the following coupling: let $(\rX(s),\rX(s+1)) \in \binom{[n]}{s} \times \binom{[n]}{s+1}$ be chosen uniformly at random under the constraint $\rX(s) \subset \rX(s+1)$. We can then bound
\begin{multline*}
 |\EE_{\nu_s}[\varphi(f)] - \EE_{\nu_{s+1}}[\varphi(f)]| = |\EE[\varphi(f(\rX(s))) - \varphi(f(\rX(s+1)))]| \leq \\ \EE[|\varphi(f(\rX(s))) - \varphi(f(\rX(s+1)))|] \leq \EE[|f(\rX(s)) - f(\rX(s+1))|].
\end{multline*}
Denoting $\rpi(s+1) = \rX(s+1) \setminus \rX(s)$ and using the multilinearity of $f$, this shows that
\[
 |\EE_{\nu_s}[\varphi(f)] - \EE_{\nu_{s+1}}[\varphi(f)]| \leq \EE\left[ \left| \frac{\partial f}{\partial x_{\rpi(s+1)}} (\rX(s)) \right| \right] =
 \EE\left[ \frac{1}{n-s} \sum_{i \notin \rX(s)} \left| \frac{\partial f}{\partial x_i} (\rX(s)) \right| \right].
\]
While we cannot bound $\sum_i |\frac{\partial f}{\partial x_i}|$ directly, Lemma~\ref{lem:poincare-derivative} implies that $\sum_i \bigl(\frac{\partial f}{\partial x_i}\bigr)^2 = O(d)$. Applying Cauchy--Schwarz, we get that for $s$ close to $pn$,
\begin{align*}
 |\EE_{\nu_s}[\varphi(f)] - \EE_{\nu_{s+1}}[\varphi(f)]| &\leq \frac{1}{\Theta(n)} \EE\left[ \sum_{i=1}^n \left| \frac{\partial f}{\partial x_i} (\rX(s)) \right| \right] \\ &\leq
 \frac{1}{\Theta(n)} \EE\left[\sqrt{n} \sqrt{\sum_{i=1}^n \frac{\partial f}{\partial x_i} (\rX(s))^2}\right] = O\left(\sqrt{\frac{d}{n}}\right).
\end{align*}
Recall now that our original goal was to bound $|\EE_{\nu_\ell}[\varphi(f)] - \EE_{\nu_{pn}}[\varphi(f)]|$ for $|\ell - pn| \leq \delta$, and our intended $\delta$ satisfied $\delta = \omega(\sqrt{n})$. Unfortunately, the idea just described only gives a bound of the form $|\EE_{\nu_\ell}[\varphi(f)] - \EE_{\nu_{pn}}[\varphi(f)]| = O(\delta\sqrt{d/n})$, which is useless for our intended $\delta$.

One way out is to take $\delta = C\sqrt{n}$. This allows us to obtain meaningful bounds both on the contribution of slices close to $pn$ and on the contribution of slices far away from $pn$. Although this only gives a constant upper bound on $|\EE_{\mu_p}[\varphi(f)] - \EE_{\nu_{pn}}[\varphi(f)]|$ if applied directly, this idea can be used in conjunction with the invariance principle for the Boolean cube~\cite{MOO} to give an invariance principle for the slice, and this is the route chosen in the prequel~\cite{FKMW}. One drawback of this approach is that the invariance principle for the Boolean cube requires all influences to be small.

\smallskip

Our approach, in contrast, considers a coupling $(\rX(0),\ldots,\rX(n))$ of \emph{all} slices. Analogous to $f(\rX(s+1)) - f(\rX(s))$, we consider the quantity
\[
 \rC(s) = (n-s) (f(\rX(s+1)) - f(\rX(s))) - s (f(\rX(s-1)) - f(\rX(s))).
\]
As before, we can bound $\EE[|\rC(s)|] = O(\sqrt{dn})$. Moreover,
\[
 \sum_{u=s}^t \rC(u) = (n-t) f(\rX(t+1)) + (t+1) f(\rX(t)) - (n-s+1) f(\rX(s)) - s f(\rX(s-1)),
\]
and so we can bound $|\EE_{\nu_\ell}[\varphi(f)] - \EE_{\nu_{pn}}[\varphi(f)]|$ by bounding the expectation of $\sum_{u=pn}^\ell \rC(u)$ or of $\sum_{u=\ell}^{pn} \rC(u)$. The triangle inequality gives $|\sum_{u=s}^t \rC(u)| = O(|s-t|\sqrt{dn})$, which suffers from the same problem that we encountered above. However, by expressing $\rC(s)$ as a difference of two martingales, we are able to improve on the triangle inequality, showing that
\[
 \left|\sum_{u=s}^t \rC(u)\right| = O(\sqrt{|s-t|dn}),
\]
a bound which is useful for $|s-t| = o(n/d)$ rather than for $|s-t| = o(\sqrt{n/d})$ as before.

In more detail, we define
\begin{align*}
 \rU(u) &= f(\rX(u+1)) - f(\rX(u)) - \EE[f(\rX(u+1)) - f(\rX(u)) | \rX(u)], \\
 \rD(u) &= f(\rX(u-1)) - f(\rX(u)) - \EE[f(\rX(u-1)) - f(\rX(u)) | \rX(u)], 
\end{align*}
both martingales by construction, $\rU(u)$ for increasing $u$, and $\rD(u)$ for decreasing $u$. We claim that $\rC(u) = (n-u)\rU(u) - u\rD(u)$. If this holds, then using the fact that $\EE[\rU(u)\rU(v)] = \EE[\rD(u)\rD(v)] = 0$ for $u \neq v$ and the $L_2$ triangle inequality $(a+b)^2 \leq 2a^2 + 2b^2$, we get
\begin{align*}
 \EE\left[\left( \sum_{u=s}^t \rC(u) \right)^2\right] &\leq 
 2\EE\left[\left( \sum_{u=s}^t (n-u)\rU(u) \right)^2\right] +
 2\EE\left[\left( \sum_{u=s}^t u\rD(u) \right)^2\right] \\ &=
 2\sum_{u=s}^t (n-u)^2 \EE[\rU(u)^2] +
 2\sum_{u=s}^t u^2 \EE[\rD(u)^2].
\end{align*}
This shows that $\EE[(\sum_{u=s}^t \rC(u))^2]$ scales linearly in $t-s$ rather than quadratically in $t-s$, which is what we would get if we just applied the triangle inequality. Since the $L_1$ norm is bounded by the $L_2$ norm, we conclude that $\EE[|\sum_{u=s}^t \rC(u)|] = O(\sqrt{|s-t|dn})$.

Finally, let us explain why $\rC(u) = (n-u)\rU(u) - u\rD(u)$. In view of our previous expression for $\rC(u)$, this boils down to proving that
\[
 (n-u) \EE[f(\rX(u+1)) - f(\rX(u)) | \rX(u)] - u \EE[f(\rX(u-1)) - f(\rX(u)) | \rX(u)] = 0.
\]
We can rewrite the left-hand side as
\[
 \EE\left[\sum_{i \notin \rX(u)} [f(\rX(u) \cup \{i\}) - f(\rX(u))] - \sum_{i \in \rX(u)} [f(\rX(u) \setminus \{i\}) - f(\rX(u))]\right].
\]
Since $f$ is multilinear, we can replace the differences with derivatives:
\[
 \EE\left[\sum_{i \notin \rX(u)} \frac{\partial f}{\partial x_i}(\rX(u)) - \sum_{i \in \rX(u)} -\frac{\partial f}{\partial x_i}(\rX(u)) \right] =
 \EE\left[\sum_{i=1}^n \frac{\partial f}{\partial x_i}(\rX(u)) \right].
\]
However, the last expression clearly vanishes, since $f$ is harmonic. This completes the outline of the proof.

\subsection{The proof} \label{sec:invariance-proof}

The basic setup of our argument is described in the following definition.

\begin{definition} \label{def:setup}
 We are given a harmonic multilinear polynomial $f$ of degree $d \geq 1$.

 Let $\rpi \in S_n$ be chosen uniformly at random, and define random variables $\rX(s) \in \binom{[n]}{s}$ for $0 \leq s \leq n$ as follows:
\[
 \rX(s)_i = \begin{cases} 1 & \text{if } i \in \rpi(\{1,\ldots,s\}), \\ 0 & \text{if } i \in \rpi(\{s+1,\ldots,n\}). \end{cases}
\]
 For $0 \leq s \leq n$, define random variables $\rU(s)$ (for $s \neq n$), $\rD(s)$ (for $s \neq 0$) and $\rC(s)$ by
\begin{align*}
 \rU(s) &= f(\rX(s+1)) - f(\rX(s)) - \EE[f(\rX(s+1)) - f(\rX(s)) | \rX(s)], \\
 \rD(s) &= f(\rX(s-1)) - f(\rX(s)) - \EE[f(\rX(s-1)) - f(\rX(s)) | \rX(s)], \\
 \rC(s) &= (n-s) \rU(s) - s \rD(s).
\end{align*}
\end{definition}

In words, $\rX(0),\ldots,\rX(n)$ is a random maximal chain in the Boolean cube. These random variables form a coupling of the uniform distributions on all slices. The random variable $\rU(s)$ measures the effect of moving up from $\rX(s)$ to $\rX(s+1)$ on $f$, normalized so that it has zero mean given $\rX(s)$. The random variable $\rD(s)$ similarly measures the effect of moving down from $\rX(s)$ to $\rX(s-1)$. Finally, $\rC(s)$ measures the effect of moving away from $\rX(s)$. The usefulness of this representation stems from the fact that $\rU(s)$ and $\rD(s)$  are martingale differences and are therefore orthogonal while $\rC(s)$ is on one hand easily expressed in terms of $\rD(s)$ and $\rU(s)$ and on the other hand is useful for bounding differences of $f$. This representation of a random directed path in terms of martingales and reverse martingales is inspired by
previous work using such representations for stationary reversible Markov chains and for stochastic integrals~\cite{LZ,NPSS}.
Part of the novelty of our proof is using such a representation in a non-stationary setup. In particular, the formula for the representation is significantly different from previous applications of the method.  The basic properties of the representation are stated formally in the following lemma.

\begin{lemma} \label{lem:setup-basic}
 The following properties hold:
\begin{enumerate}[(a)]
 \item \label{lem:setup-basic-a} If $s \neq t$ then $\EE[\rU(s) \rU(t)] = \EE[\rD(s) \rD(t)] = 0$.
 \item \label{lem:setup-basic-b} For all $s$ we can bound
\begin{align*}
 \EE[\rU(s)^2] &\leq 4\EE[(f(\rX(s+1)) - f(\rX(s)))^2] \leq \frac{4}{n-s}
 \sum_{i=1}^n \left\|\frac{\partial f}{\partial x_i}\right\|^2_{\nu_s}, \\
 \EE[\rD(s)^2] &\leq 4\EE[(f(\rX(s-1)) - f(\rX(s)))^2] \leq
 \frac{4}{s} \sum_{i=1}^n \left\|\frac{\partial f}{\partial x_i}\right\|^2_{\nu_s}.
\end{align*}
 \item \label{lem:setup-basic-c} For all $s$ we have $\rC(s) = (n-s) f(\rX(s+1)) - (n-2s) f(\rX(s)) - s f(\rX(s-1))$.
 \item \label{lem:setup-basic-d} For all $s \leq t$ we have
\[ \sum_{u=s}^t \rC(u) = (n-t) f(\rX(t+1)) + (t+1) f(\rX(t)) - (n-s+1) f(\rX(s)) - s f(\rX(s-1)). \]
\end{enumerate}
\end{lemma}
\begin{proof}
 For (\ref{lem:setup-basic-a}) note that if $s < t$ then $\EE[\rU(t) | \rX(s+1),\rX(s)] = 0$, and so $\EE[\rU(t) \rU(s)] = 0$. For similar reasons we have $\EE[\rD(t) \rD(s)] = 0$.

 For (\ref{lem:setup-basic-b}), note first that the $L_2$ triangle inequality implies that $\EE[\rU(s)^2] \leq 4\EE[(f(\rX(s+1)) - f(\rX(s)))^2]$, using the bound
\[
 \EE_{\rX(s)}[\EE[f(\rX(s+1))-f(\rX(s))|\rX(s)]^2] \leq \EE_{\rX(s)}[\EE[(f(\rX(s+1))-f(\rX(s)))^2|\rX(s)]] = \EE[(f(\rX(s+1))-f(\rX(s)))^2].
\]
 Since $f$ is multilinear, we have $f(\rX(s+1)) = f(\rX(s)) + \frac{\partial f}{\partial x_{\rpi(s+1)}}(\rX(s))$, and so
\begin{align*}
 \EE[\rU(s)^2] &\leq 4\EE\left[\left(\frac{\partial f}{\partial x_{\rpi(s+1)}}(\rX(s))\right)^2\right] \\ &=
 \frac{4}{n-s} \EE\left[\sum_{i\colon \rX(s)_i = 0} \left(\frac{\partial f}{\partial x_i}(\rX(s))\right)^2\right] \\ &\leq
 \frac{4}{n-s} \sum_{i=1}^n \EE\left[\left(\frac{\partial f}{\partial x_i}(\rX(s))\right)^2\right].
\end{align*}
 This implies our estimate for $\EE[\rU(s)^2]$. The estimate for $\EE[\rD(s)^2]$ is obtained in the same way.

 For (\ref{lem:setup-basic-c}), we again use the formula $f(\rX(s+1)) = f(\rX(s)) + \frac{\partial f}{\partial x_{\rpi(s+1)}}(\rX(s))$ to obtain
\[
 \EE[f(\rX(s+1)) - f(\rX(s)) | \rX(s)] = \frac{1}{n-s} \EE\left[\sum_{i\colon \rX(s)_i = 0} \frac{\partial f}{\partial x_i}(\rX(s))\right].
\]
 Similarly, using the formula $f(\rX(s-1)) = f(\rX(s)) - \frac{\partial f}{\partial x_{\rpi(s-1)}}(\rX(s))$ instead, we have
\[
 \EE[f(\rX(s-1)) - f(\rX(s)) | \rX(s)] = -\frac{1}{s} \EE\left[\sum_{i\colon \rX(s)_i = 1} \frac{\partial f}{\partial x_i}(\rX(s))\right].
\]
 Combining these formulas together, we obtain
\begin{align*}
 \rC(s) &= (n-s) f(\rX(s+1)) - (n-2s) f(\rX(s)) - s f(\rX(s-1)) + \EE\left[\sum_{i=1}^n \frac{\partial f}{\partial x_i}(\rX(s))\right] \\ &= (n-s) f(\rX(s+1)) - (n-2s) f(\rX(s)) - s f(\rX(s-1)),
\end{align*}
 since $f$ is harmonic.

Part~\eqref{lem:setup-basic-d} follows by simple computation, whose highlight is noticing that the coefficient of $f(\rX(u))$ for $s < u < t$ in the sum is $(n-u+1) - (n-2u) - (u+1) = 0$.
\end{proof}

The exact definition of $\rC(s)$ is aimed at the cancellation of the derivative terms in the proof of Lemma~\ref{lem:setup-basic}\eqref{lem:setup-basic-c}, in which we use the harmonicity of $f$. We also use the harmonicity of $f$ to bound the variance of $f$ with respect to various slices.

The next step is bounding the quantity appearing in Lemma~\ref{lem:setup-basic}\eqref{lem:setup-basic-b} in terms of~$d$, the degree of~$f$. This is a simple application of Lemma~\ref{lem:poincare-derivative} on page~\pageref{lem:poincare-derivative}.

\begin{lemma} \label{lem:derivative-L2}
 For every integer $0 \leq qn \leq n$ such that $d \leq q(1-q)n$ we have
\[
 \sum_{i=1}^n \left\|\frac{\partial f}{\partial x_i}\right\|^2_{\nu_{qn}} = O\left(\frac{d}{q(1-q)}\right) \VV[f]_{\nu_{qn}}.
\]
\end{lemma}
\begin{proof}
 Combine Lemma~\ref{lem:poincare-derivative} with Lemma~\ref{lem:norm-ratios}.
\end{proof}

Our work so far has not focused on any specific slice. Now we turn to the following question. Suppose that $f$ has unit variance on the slice $\binom{[n]}{pn}$. What can we say about its behavior on a nearby slice $\binom{[n]}{qn}$?

\begin{lemma} \label{lem:coupling}
 There exists a constant $K>0$ such that the following holds.
 Let $0 \leq pn,qn \leq n$ be integers such that $|p-q| \leq p(1-p)/(2d)$, and suppose that $\VV[f]_{\nu_{pn}} = 1$. If $d^2 \leq Kp(1-p)n$ then
\[
 \EE[(f(\rX(pn)) - f(\rX(qn)))^2] = O\left(\frac{|q-p|d}{p(1-p)}\right).
\]
\end{lemma}
\begin{proof}
 Let $k = pn$ and $\ell = qn$, and assume that $k < \ell$ (the other case is very similar). Note that $|(d/d\rho) \rho(1-\rho)| = |1-2\rho| \leq 1$ for $\rho \in [0,1]$, and so $|p(1-p)-q(1-q)| \leq |p-q| \leq p(1-p)/(2d)$. Therefore for $K \leq 1/2$ the assumption $d^2 \leq Kp(1-p)n$ implies that $d^2 \leq r(1-r)n$ for all $r \in [p,q]$, since $r(1-r) \geq p(1-p) - p(1-p)/(2d) \geq p(1-p)/2$.
 
 Lemma~\ref{lem:setup-basic}\eqref{lem:setup-basic-d} implies that
\[
 \Delta \eqdef \sum_{u=k}^\ell \rC(u) = (n-\ell) f(\rX(\ell+1)) + (\ell+1) f(\rX(\ell)) - (n-k+1) f(\rX(k)) - k f(\rX(k-1)).
\]
 As a first step, we bound $\EE[\Delta^2]$. The $L_2$ triangle inequality implies that
\[
 \EE[\Delta^2] \leq 2\underbrace{\EE\left[\left(\sum_{u=k}^\ell (n-u)\rU(u)\right)^2\right]}_{\Delta_U^2} + 2\underbrace{\EE\left[\left(\sum_{u=k}^\ell u\rD(u)\right)^2\right]}_{\Delta_D^2}.
\]
 Lemma~\ref{lem:setup-basic}\eqref{lem:setup-basic-a} and Lemma~\ref{lem:setup-basic}\eqref{lem:setup-basic-b} show that
\[
 \Delta_U^2 = \sum_{u=k}^\ell (n-u)^2 \EE[\rU(u)^2] \leq \sum_{u=k}^\ell 4(n-u) \sum_{i=1}^n \left\|\frac{\partial f}{\partial x_i}\right\|^2_{\nu_u}.
\]
 Applying Lemma~\ref{lem:derivative-L2} (using $d^2 \leq r(1-r)n$ for all $r \in [p,q]$), we obtain
\[
 \Delta_U^2 \leq \sum_{u=k}^\ell O\left(\frac{d(n-u)}{(u/n)(1-u/n)}\right) \VV[f]_{\nu_u} \leq \sum_{u=k}^\ell O\left(\frac{dn}{p(1-p)}\right) \VV[f]_{\nu_u}.
\]
 Putting $u = rn$, Corollary~\ref{cor:approximate-norm} together with Lemma~\ref{lem:norms} shows that
\[
 \VV[f]_{\nu_u} \leq \max_{e \leq d} \frac{(2r(1-r))^e(1 \pm O(\tfrac{e^2}{p(1-p)n}))}{(2p(1-p))^e(1 \pm O(\tfrac{e^2}{p(1-p)n}))},
\]
 using the observation $r(1-r) \geq p(1-p)/2$ as well as $\VV[f]_{\nu_{pn}} = 1$. For small enough $K>0$ we have $1 \pm O(\tfrac{e^2}{p(1-p)n}) = \Theta(1)$. As observed above, $|p(1-p)-r(1-r)| \leq |p-r| \leq p(1-p)/2d$, and so
\[
 \VV[f]_{\nu_u} \leq O(1) \cdot \max_{e \leq d} \left(1 + \frac{1}{2d}\right)^e = O(1).
\] 
Therefore
\[
 \Delta_U^2 \leq \sum_{u=k}^\ell O\left(\frac{dn}{p(1-p)}\right) = O\left(\frac{|\ell-k|dn}{p(1-p)}\right).
\]
 We can bound $\Delta_L^2$ similarly, and conclude that
\[
 \EE[\Delta^2] = O\left(\frac{|\ell-k|dn}{p(1-p)}\right).
\]

 The triangle inequality implies that
\[
 (n+1) \|f(\rX(\ell)) - f(\rX(k))\| \leq \|\Delta\| + (n-\ell) \|f(\rX(\ell+1))-f(\rX(\ell))\| + k \|f(\rX(k-1))-f(\rX(k))\|.
\]
 The latter two terms can be bounded in the same way that we bounded $\EE[\rU(u)^2]$ above, and we conclude that
\[
 (n+1) \|f(\rX(\ell)) - f(\rX(k))\| = O\left(\sqrt{\frac{|\ell-k|dn}{p(1-p)}}\right).
\]
 This implies that
\[
 \EE[(f(\rX(pn)) - f(\rX(qn)))^2] = O\left(\frac{|\ell-k|d}{p(1-p)n}\right) = O\left(\frac{|q-p|d}{p(1-p)}\right). \qedhere
\]
\end{proof}

Combining Lemma~\ref{lem:coupling} with Chernoff's bound, we obtain our main theorem.

\begin{proposition} \label{pro:chernoff}
 Let $0 < p,\epsilon < 1$ satisfy $\epsilon \leq p(1-p)$. Then
\[ \Pr[|\bin(n,p)-np| > \epsilon n] \leq 2e^{-\epsilon^2 n/(6p(1-p))}. \]
\end{proposition}
\begin{proof}
 Suppose first that $p \leq 1/2$, and let $X \sim \bin(n,p)$ and $\mu = np$. One common version of Chernoff's bound states that for $0 < \delta < 1$ we have
\[
 \Pr[|X-\mu| > \delta\mu] \leq 2 e^{-\delta^2\mu/3}.
\]
 Choose $\delta = \epsilon/p$, and note that $\delta \leq 1-p < 1$. Since $\delta^2\mu = (\epsilon^2/p^2)(np) = \epsilon^2n/p \geq \epsilon^2n/(2p(1-p))$, the bound follows in this case.

 When $p \geq 1/2$, we look at $X \sim \bin(n,1-p)$ instead, using the fact that $n - X \sim \bin(n,p)$.
\end{proof}

\begin{theorem} \label{thm:invariance}
 There exists a constant $K>0$ such that the following holds, whenever $p(1-p)n \geq 3\log n$.

 Let $f$ be a harmonic multilinear polynomial of degree $d$ satisfying $d \leq K\sqrt{p(1-p)n}$ such that $\VV[f]_{\nu_{pn}} = 1$. For any Lipschitz function $\varphi$,
\[
 |\EE_{\nu_{pn}}[\varphi(f)] - \EE_{\mu_p}[\varphi(f)]| = O\left(\sqrt{\frac{d}{\sqrt{p(1-p)n}}}\right).
\]
\end{theorem}
\begin{proof}
 We can assume, without loss of generality, that $\EE[f] = 0$ (Corollary~\ref{cor:approximate-norm} implies that the expectation is the same with respect to both $\nu_{pn}$ and $\mu_p$), so that $\EE[f^2] = 1$, and that $\varphi(0) = 0$. Since $\varphi$ is Lipschitz, it follows that $|\varphi(x)| \leq |x|$ for all $x$.

 Let $\delta := p(1-p)/(2d)$.
 Lemma~\ref{lem:coupling} implies that whenever $|q-p| \leq \delta$,
\begin{gather*}
 |\EE_{\nu_{pn}}[\varphi(f)] - \EE_{\nu_{qn}}[\varphi(f)]| =
 |\EE[\varphi(f(\rX(pn))) - \varphi(f(\rX(qn)))]| \leq
 |\EE[f(\rX(pn)) - f(\rX(qn))]| \\ \leq
 \|f(\rX(pn)) - f(\rX(qn))\| =
 O\left(\sqrt{\frac{|q-p| d}{p(1-p)}}\right).
\end{gather*}
 Expressing $\mu_p$ as a mixture of distributions of the form $\nu_{qn}$, we obtain
\begin{align*}
 |\EE_{\nu_{pn}}[\varphi(f)] - \EE_{\mu_p}[\varphi(f)]| &=
 \left| \EE_{\nu_{pn}}[\varphi(f)] - \sum_{\ell=0}^n \binom{n}{\ell} p^\ell (1-p)^{n-\ell} \EE_{\nu_\ell}[\varphi(f)]\right| \\ &\leq
 \sum_{\ell=0}^n \binom{n}{\ell} p^\ell (1-p)^{n-\ell} \bigl|\EE_{\nu_{pn}}[\varphi(f)] - \EE_{\nu_\ell}[\varphi(f)]\bigr| \\ &\leq
 O\left(\sqrt{\frac{d}{p(1-p)}}\right) \underbrace{\EE_{\ell \sim \bin(n,p)} \left[\sqrt{|\tfrac{\ell}{n}-p|}\right]}_{C} + \\ &\hspace{-1cm} \underbrace{\Pr[|\bin(n,p)-pn| > \delta n] \cdot |\EE_{\nu_{pn}}[\varphi(f)]|}_{\epsilon_1} + \underbrace{\EE_{x\sim\mu_p}[|\varphi(f(x))| \cdot \charf{|\Sigma_i x_i - pn| > \delta n}]}_{\epsilon_2}.
\end{align*}
 The coefficient $C$ can be bounded using Proposition~\ref{pro:chernoff} and the formula $\EE[X] = \int_0^\infty \Pr[X \geq t] \, dt$ (for $X \geq 0$):
\begin{multline*}
 C = \int_0^\infty \Pr[\sqrt{|\bin(n,p)/n - p|} \geq t] \, dt =
 \int_0^\infty \Pr[|\bin(n,p) - pn| \geq t^2n] \, dt \\ \leq
 \int_0^{\sqrt{p(1-p)}} 2e^{-t^4 n/(6p(1-p))} \, dt + \int_{\sqrt{p(1-p)}}^1 2e^{-p(1-p)n/6} \, dt.
\end{multline*}
 Substitute $s = \sqrt[4]{n/6p(1-p)} \cdot t$ to get
\[
 C \leq 2\sqrt[4]{\frac{6p(1-p)}{n}} \int_0^\infty e^{-s^4} \, ds + O(e^{-p(1-p)n/6}) = O\left(\sqrt[4]{\frac{p(1-p)}{n}}\right),
\]
 since $e^{-p(1-p)n/6} \leq \frac{1}{\sqrt{n}}$ whereas $\sqrt[4]{\frac{p(1-p)}{n}} = \sqrt[4]{\frac{p(1-p)n}{n^2}} \geq \frac{\sqrt[4]{3\log n}}{\sqrt{n}} \geq \frac{1}{\sqrt{n}}$.

 We proceed to bound the error terms $\epsilon_1,\epsilon_2$. With respect to $\nu_{pn}$, $\EE[|\varphi(f)|] \leq \EE[|f|] \leq \|f\| \leq 1$, and so Proposition~\ref{pro:chernoff} implies that
\[
 \epsilon_1 \leq 2e^{-\delta^2 n/(6p(1-p))}.
\]
 For the second error term, let $M > 0$ be a parameter to be determined. We have
\[
 \epsilon_2 \leq M \Pr[|\bin(n,p)-pn| > \delta n] + \EE_{x\sim\mu_p}[|\varphi(f(x))| \cdot \charf{|\varphi(f(x))| > M}].
\]
 If $|\varphi(f(x))| > M$ then certainly $|f(x)| \geq |\varphi(f(x))| > M$, and so
\[
 \EE_{x\sim\mu_p}[|\varphi(f(x))| \cdot \charf{|\varphi(f(x))| > M}] \leq
 \EE_{x\sim\mu_p}[|f(x)| \cdot \charf{|f(x)| > M}] \leq \frac{1}{M} \EE_{\mu_p}[f^2] = O\left(\frac{1}{M}\right),
\]
 since Lemma~\ref{lem:norms} and Corollary~\ref{cor:approximate-norm} imply that $\|f\|_{\mu_p} = O(\|f\|_{\nu_{pn}}) = O(1)$ for small enough $K$. Hence
\[
 \epsilon_2 \leq 2Me^{-\delta^2 n/(6p(1-p))} + O\left(\frac{1}{M}\right).
\]
 Choosing $M = e^{-\delta^2 n/(12p(1-p))}$, we conclude that
\[
 \epsilon_2 = O(e^{-\delta^2 n/(12p(1-p))}).
\]
 Putting everything together, we obtain
\[
 |\EE_{\nu_{pn}}[\varphi(f)] - \EE_{\mu_p}[\varphi(f)]| = O\left(\sqrt{\frac{d}{\sqrt{p(1-p)n}}} + e^{-\frac{\delta^2 n}{12p(1-p)}}\right).
\]
 Substituting $\delta = p(1-p)/(2d)$, we deduce
\[
  |\EE_{\nu_{pn}}[\varphi(f)] - \EE_{\mu_p}[\varphi(f)]| = O\left(\underbrace{\sqrt{\frac{d}{\sqrt{p(1-p)n}}}}_A + \underbrace{e^{-\frac{p(1-p) n}{48d^2}}}_B\right).
\]
 Note that $B = e^{-1/(48A^4)}$. When $A < 0.38$, calculation shows that $B < A$. Since $A \leq \sqrt{K}$, choosing $K \leq 0.38^2$ ensures that $B < A$.
\end{proof}

As a corollary, we can estimate the L\'evy distance between $f(\nu_{pn})$ and $f(\mu_p)$, along the lines of~\cite[Theorem 3.19(28)]{MOO}.

\begin{corollary} \label{cor:levy}
 Suppose that $p(1-p)n \geq 3\log n$, and let $f$ be a harmonic multilinear polynomial of degree $d$ satisfying $d \leq K\sqrt{p(1-p)n}$ such that $\VV[f]_{\nu_{pn}} = 1$, where $K>0$ is the constant from Theorem~\ref{thm:invariance}. The L\'evy distance between $f(\nu_{pn})$ and $f(\mu_p)$ is at most
\[
 \epsilon = O\left(\sqrt[4]{\frac{d}{\sqrt{p(1-p)n}}}\right).
\]
 That is, for all $y$ it holds that
\[
 \Pr_{\nu_{pn}}[f \leq y - \epsilon] - \epsilon \leq \Pr_{\mu_p}[f \leq y] \leq \Pr_{\nu_{pn}}[f \leq y + \epsilon] + \epsilon.
\]
\end{corollary}
\begin{proof}
 Given $y$, define a $(1/\epsilon)$-Lipschitz function $\varphi$ by
\[
 \varphi(t) = \begin{cases} 1 & \text{if } t \leq y, \\ \frac{y+\epsilon-t}{\epsilon} & \text{if } y \leq t \leq y + \epsilon, \\ 0 & \text{if } t \geq y + \epsilon. \end{cases}
\]
 It is easy to check that $\charf{t \leq y} \leq \varphi(t) \leq \charf{t \leq y+\epsilon}$, and so
\[
 \Pr_{\mu_p}[f \leq y] - \Pr_{\nu_{pn}}[f \leq y+\epsilon] \leq \EE_{\mu_p}[\varphi(f)] - \EE_{\nu_{pn}}[\varphi(f)] = \frac{1}{\epsilon} O\left(\sqrt{\frac{d}{\sqrt{p(1-p)n}}}\right) = \epsilon,
\]
 using Theorem~\ref{thm:invariance} and the correct choice of $\epsilon$. We get the bound in the other direction in the same way.
\end{proof}

We conjecture that $f(\nu_{pn})$ and $f(\mu_p)$ are also close in CDF distance. Unfortunately, the method of proof of~\cite[Theorem 3.19(30)]{MOO} relies on the anticoncentration of multivariate Gaussian distributions~\cite{CarberyWright}, whereas both $f(\nu_{pn})$ and $f(\mu_p)$ are discrete distributions.
We consider it an interesting open problem to extend Corollary~\ref{cor:levy} to CDF distance.

\begin{question} \label{q:CDF}
 Suppose that $p(1-p)n$ is ``large'' and $d$ is ``small'', compared to $n$. Is it true that for every harmonic multilinear polynomial $f$ of degree $d$ satisfying $\VV[f]_{\nu_{pn}} = 1$, the CDF distance between $f(\nu_{pn})$ and $f(\mu_p)$ is $o(1)$?
\end{question}

\subsection{High-degree functions} \label{sec:high-degree}

Theorem~\ref{thm:invariance} requires that $d = O(\sqrt{p(1-p)n})$. Indeed, Lemma~\ref{lem:norms}, which implies that the norm of a low-degree function is approximately the same under both $\mu_p$ and $\nu_{pn}$, already requires the degree to be $O(\sqrt{p(1-p)n})$. For $d = \omega(\sqrt{p(1-p)n})$ and constant $p \neq \tfrac{1}{2}$ we exhibit below a $0/\pm C$-valued function $f$ (for some $C$ depending on $p$ and $d$) which satisfies $\|f\|_{\mu_p} = 1$ while $\|f\|_{\nu_{pn}} = o(1)$. This shows that for constant $p \neq \tfrac{1}{2}$ the dependence on the degree is essential in Theorem~\ref{thm:invariance}, since $|\EE_{\nu_{pn}}[|f|] - \EE_{\mu_p}[|f|]| = \|f\|_{\mu_p}^2 - \|f\|_{\nu_{pn}}^2 = 1 - o(1)$. We do not know whether this dependence is necessary for $p = \tfrac{1}{2}$. Indeed, Lemma~\ref{lem:norms} can be extended above $\sqrt{n}$ in this case, as the calculation below shows.

Let $d = \omega(\sqrt{p(1-p)n})$, and assume further that $d = o((p(1-p)n)^{2/3})$. We consider the function $f = (2p(1-p))^{-d/2} (x_1 - x_2) \cdots (x_{2d-1} - x_{2d})$, whose $\mu_p$-norm is~$1$ according to Lemma~\ref{lem:norms}. The lemma also gives its $\nu_k$-norm (where $k = pn$) as
\[
 \|f\|_{\nu_k}^2 = (p(1-p))^{-d} \frac{k^{\underline{d}} (n-k)^{\underline{d}}}{n^{\underline{2d}}}.
\]
We estimate this expression using Stirling's approximation, starting with $k^{\underline{d}}$:
\[
 k^{\underline{d}} = \frac{k!}{(k-d)!} = \left(\frac{k}{k-d}\right)^{k-d+1/2} \frac{k^d}{e^d} e^{O(1/k) - O(1/(k-d))} = \left(1+\frac{d}{k-d}\right)^{k-d} \frac{k^d}{e^d} (1 \pm o(1)) .
\]
The Taylor series $\log (1+x) = x - x^2/2 + O(x^3)$ shows that
\[
 \left(1 + \frac{d}{k-d}\right)^{k-d} = \exp \left[ d - \frac{d^2}{2(k-d)} + o(1) \right] =
 \exp \left[ d - \frac{d^2}{2k} + o(1) \right],
\]
and so $k^{\underline{d}} = k^d e^{-d^2/2k} (1 \pm o(1))$. We can similarly estimate $(n-k)^{\underline{d}} = (n-k)^d e^{-d^2/2(n-k)} (1 \pm o(1))$ and $n^{\underline{2d}} = n^{2d} e^{-2d^2/n} (1 \pm o(1))$, concluding that
\begin{align*}
 \|f\|_{\nu_k}^2 &= (p(1-p))^{-d} \frac{k^d (n-k)^d}{n^{2d}} e^{-d^2/2k - d^2/2(n-k) + 2d^2/n} (1 \pm o(1)) \\ &= \exp \left[ \frac{d^2}{2p(1-p)n} \left(-1 + 4p(1-p)\right) \pm o(1) \right].
\end{align*}
If $p \neq \tfrac{1}{2}$ is fixed, we immediately conclude that $\|f\|_{\nu_k} = o(1)$.

The constant $C = (2p(1-p))^{-d/2}$ is unbounded as a function of $d$. We do not know whether Theorem~\ref{thm:invariance} can be extended to higher degrees for bounded functions.

\section{Approximately Boolean functions} \label{sec:approx-bool}

Theorem~\ref{thm:invariance} only applies to Lipschitz test functions, but in many applications we are interested in functions which grow faster, for example the squared-distance-from-$\{-1,1\}$ function $\varphi(x) = (|x|-1)^2$. In this section we show how to handle such functions using hypercontractivity.

\begin{proposition} \label{pro:hypercontractivity}
 Denote by $\|\cdot\|_r$ the $L_r$ norm. For a multilinear polynomial $f$, let $T_\rho$ denote the operator
\[ T_\rho f = \sum_{i=0}^n \rho^i f^{=i}. \]
 For $r \geq 2$ and with respect to $\mu_p$,
\[ \|T_\rho f\|_r \leq \|f\|_2, \text{ where } \rho = \sqrt{\frac{p(1-p)}{r-1}}. \]

 For a harmonic multilinear polynomial $f$, let $H_\rho$ denote the operator
\[ H_\rho f = \sum_{i=0}^{n/2} \rho^{i(1-(i-1)/n)} f^{=i}. \]
 For $r \geq 2$ and with respect to $\nu_{pn}$,
\[ \|H_\rho f\|_r \leq \|f\|_2, \text{ where } \rho = (r-1)^{\Theta(\log(p(1-p)))} = (p(1-p))^{\Theta(\log(r-1))}. \]
\end{proposition}
\begin{proof}
 The first result is classical, appearing in~\cite{Bonami,Beckner} for $p = 1/2$ and in~\cite{Talagrand,Oleszkiewicz} for general $p$. The second result is due to Lee and Yau~\cite{LeeYau}. They proved the corresponding log-Sobolev inequality, which implies hypercontractivity as shown in~\cite{DCS}.
\end{proof}

These results imply that for low-degree functions, the $L_r$ norm and the $L_2$ norm are comparable.

\begin{lemma} \label{lem:low-degree}
 Let $f$ be a multilinear polynomial of degree $d$, and let $r \geq 2$ be a constant. With respect to $\mu_p$,
\[
 \|f\|_r \leq O(p(1-p))^{-O(d)} \|f\|_2.
\]
 If $f$ is also harmonic, then with respect to $\nu_{pn}$,
\[
 \|f\|_r \leq O(p(1-p))^{-O(d)} \|f\|_2.
\]
\end{lemma}
\begin{proof}
 In both cases, we apply Proposition~\ref{pro:hypercontractivity} to $T_{1/\rho} f$ or to $H_{1/\rho} f$, where $\rho$ is given by the proposition. In the first case, we get
\[
 \|f\|_r^2 \leq \|T_{1/\rho} f\|_2^2 = \sum_{i=0}^d \rho^{-2i} \|f^{=i}\|_2^2 \leq \rho^{-2d} \|f\|_2^2.
\]
 In the second case, we get
\[
 \|f\|_r^2 \leq \|U_{1/\rho} f\|_2^2 = \sum_{i=0}^d \rho^{-2i(1-(i-1)/n)} \|f^{=i}\|_2^2 \leq \rho^{-2d} \|f\|_2^2.
\]
 We obtain the stated bounds by substituting the values of $\rho$.
\end{proof}

We can now obtain our invariance principle for $\varphi(x) = (|x|-1)^2$.

\begin{theorem} \label{thm:invariance-sq}
 Suppose that $p(1-p)n \geq 3\log n$, and let $f$ be a harmonic multilinear polynomial of degree $d$ satisfying $d \leq K\sqrt{p(1-p)n}$ such that $\|f\|_{\nu_{pn}} = 1$, where $K$ is the constant in Theorem~\ref{thm:invariance}. We have
\[
 |\EE_{\nu_{pn}}[(|f|-1)^2] - \EE_{\mu_p}[(|f|-1)^2]| = O\left(\sqrt[4]{\frac{d}{\sqrt{p(1-p)n}}}(p(1-p))^{-O(d)}\right).
\]
\end{theorem}
\begin{proof}
 Note that $\|f\|_{\nu_{pn}} = 1$ implies that $\VV[f]_{\nu_{pn}} \leq 1$ and $\|f\|_{\mu_p} = O(1)$, due to Corollary~\ref{cor:approximate-norm} and Lemma~\ref{lem:norms}.

 Let $M \geq 1$ be a parameter to be decided, and define the $2(M-1)$-Lipschitz function $\psi$ by
\[
 \psi(x) = \begin{cases} (|x|-1)^2 & \text{if } |x| \leq M, \\ (M-1)^2 & \text{if } |x| \geq M. \end{cases}
\]
 When $|x| \geq M$, we have $(|x|-1)^2 \leq x^2$, and so with respect to any measure we have
\[
 \EE[(|f|-1)^2] \leq \EE[\psi(f)] + \EE[f^2 \charf{|f| \geq M}].
\]
 With respect to either $\mu_p$ or $\nu_{pn}$, using Lemma~\ref{lem:low-degree} we can bound
\[
 \EE[f^2 \charf{|f| \geq M}] \leq \frac{1}{M}\EE[|f|^3] = \frac{1}{M}\|f\|_3^3 \leq \frac{1}{M} O(p(1-p))^{-O(d)}.
\]
 Theorem~\ref{thm:invariance} therefore implies that
\[
 |\EE_{\nu_{pn}}[(|f|-1)^2] - \EE_{\mu_p}[(|f|-1)^2]| \leq M O\left(\sqrt{\frac{d}{\sqrt{p(1-p)n}}}\right) + \frac{1}{M} O(p(1-p))^{-O(d)}.
\]
 Choosing $M$ to be the geometric mean of both terms appearing above results in the statement of the theorem.
\end{proof}

As an illustration of this theorem, we give an alternative proof of~\cite[Theorem 7.5]{FKMW}, a Kindler--Safra theorem for the slice.

\begin{definition} \label{def:boolean}
 A function $f$ on a given domain is \emph{Boolean} if on the domain it satisfies $f \in \{\pm 1\}$. If the domain is a cube, we use the term \emph{cube-Boolean}. If it is a slice, we use the term \emph{slice-Boolean}.
\end{definition}

\begin{proposition} \label{pro:kindler-safra}
 Let $f$ be a multilinear polynomial of degree $d$ such that $\EE_{\mu_p}[(|f|-1)^2] = \epsilon$. There exists a cube-Boolean function $g$ on $(p(1-p))^{-O(d)}$ coordinates such that $\|f-g\|^2_{\mu_p} = O((p(1-p))^{-O(d)} \epsilon)$.
\end{proposition}
\begin{proof}
 This is essentially proved in~\cite{KindlerSafra,Kindler}. Explicitly, they prove the same result without the guarantee that $g$ is cube-Boolean. In order to get our version, let $F = \sgn f$ and $G = \sgn g$. By definition $\|F-f\|^2 = \epsilon$, and so $\|F-g\|^2 = O((p(1-p))^{-O(d)} \epsilon)$. Since $F$ is cube-Boolean, this implies that $\|F-G\|^2 = O((p(1-p))^{-O(d)} \epsilon)$. We conclude that $\|f-G\|^2 = O((p(1-p))^{-O(d)} \epsilon)$.
\end{proof}

\begin{theorem} \label{thm:kindler-safra}

 Let $f$ be a slice-Boolean harmonic multilinear polynomial such that $\|f^{>d}\|^2_{\nu_{pn}} = \epsilon$, where $f^{>d} = \sum_{i>d} f^{=i}$. There exists a slice-Boolean harmonic multilinear polynomial $h$ depending on $(p(1-p))^{-O(d)}$ coordinates (that is, invariant to permutations of the other coordinates) satisfying
\[
 \|f-h\|^2_{\nu_{pn}} \leq O((p(1-p))^{-O(d)} \epsilon) + O_{p,d}\left(\frac{1}{n^{1/8}}\right).
\]
\end{theorem}

Before we can prove this theorem, we need an auxiliary result~\cite[Theorem 3.3]{FKMW}, which we prove here (simplifying the original proof) to make the proof self-contained. The proof uses Corollary~\ref{cor:harmonic-projection}, whose self-contained proof appears in Subsection~\ref{sec:harmonic-algebra}.

\begin{theorem} \label{thm:harmonic-projection}
 Let $f$ be a multilinear polynomial depending on $M \leq \min(p,1-p)n$ variables, and let $\tilde{f}$ be the unique harmonic multilinear polynomial of degree at most $\min(p,1-p)n$ agreeing with $f$ on the slice $\binom{[n]}{pn}$ (the \emph{harmonic projection} of $f$ on the slice $\binom{[n]}{pn}$). Then
\[
 \|f - \tilde{f}\|_{\mu_p}^2 = O\left(\frac{M^22^M}{p(1-p)n}\right)\|f\|_{\mu_p}^2.
\]
\end{theorem}
\begin{proof}
 The Fourier expansion of $f$ with respect to $\mu_p$ is
\[
 f = \sum_{S \subseteq M} \hat{f}(S) \omega_S, \text{ where } \omega_S = \prod_{i \in S} \frac{x_i - p}{\sqrt{p(1-p)}}.
\]
 The characters $\omega_S$ are orthogonal and have unit norm with respect to $\mu_p$. This shows that $\|f\|_{\mu_p}^2 = \sum_{S \subseteq M} \hat{f}(S)^2$. Harmonic projection is a linear operator, and so denoting the harmonic projection of $\omega_S$ by $\tilde{\omega}_S$, we have
\[
 f - \tilde{f} = \sum_{S \subseteq M} \hat{f}(S) (\omega_S - \tilde{\omega}_S).
\]
 We show below that
\begin{equation} \label{eq:hp}
 \|\omega_S - \tilde{\omega}_S\|_{\mu_p}^2 = O\left(\frac{|S|^2}{p(1-p)n}\right).
\end{equation}
 Assuming without loss of generality that $f$ depends on the first $M$ variables, the $L_2$ triangle inequality then implies that
\[
 \|f - \tilde{f}\|_{\mu_p}^2 \leq 2^M \sum_{S \subseteq [M]} \hat{f}(S)^2 \|\omega_S - \tilde{\omega}_S\|_{\mu_p}^2 \leq 
 O\left(\frac{M^22^M}{p(1-p)n}\right) \sum_{S \subseteq [M]} \hat{f}(S)^2 = O\left(\frac{M^22^M}{p(1-p)n}\right)\|f\|_{\mu_p}^2.
\]

 \smallskip
 
 It remains to prove~\eqref{eq:hp}. For definiteness, consider $S = \{1,\ldots,d\}$, where $d \leq \min(p,1-p)n$. The first step is to consider the related function $\chi_S = \prod_{i=1}^d x_i$ and its harmonic projection $\tilde\chi_S$. We will be particularly interested in the coefficient of the monomial $\prod_{i=1}^d x_i$ in $\tilde\chi_S$. According to Corollary~\ref{cor:harmonic-projection}, the coefficient of $\prod_{i=1}^d x_i$ in $\tilde\chi_S$ is independent of $p$, as long as $d \leq pn \leq n-d$. This suggests considering the harmonic projection of $\prod_{i=1}^d x_i$ on the slice $\binom{[n]}{d}$. On that slice, $\prod_{i=1}^d x_i$ is the indicator function of the point $\mathbf{p} = \{1,\ldots,d\}$, and we denote its harmonic projection on $\binom{[n]}{d}$ by $\charf{\mathbf{p}}$.
 The coefficient of $\prod_{i=1}^d x_i$ in $\charf{\mathbf{p}}$ clearly equals its coefficent in $\charf{\mathbf{p}}^{=d}$. Since all monomials in $\charf{\mathbf{p}}^{=d}$ have degree $d$, this coefficient is simply $\charf{\mathbf{p}}^{=d}(\mathbf{p})$.
 
 Let $E_d$ be the operator projecting a function $\phi$ on $\binom{[n]}{d}$ to $\phi^{=d}$. The trace of $E_d$ is clearly the dimension of the space of harmonic multilinear polynomials which are homogeneous of degree $d$, which is $\binom{n}{d} - \binom{n}{d-1}$ by Corollary~\ref{cor:dimension}. Symmetry shows that the diagonal elements of $E_d$ are all equal to
\[
 \frac{\binom{n}{d} - \binom{n}{d-1}}{\binom{n}{d}} = 1 - \frac{d}{n-d+1}.
\]
 This is exactly the value of $\charf{\mathbf{p}}^{=d}(\mathbf{p})$. We conclude that the coefficient of $\prod_{i=1}^d x_i$ in $\tilde\chi_S$ is $1 - \frac{d}{n-d+1}$, and so
\[
 \chi_S - \tilde\chi_S = \frac{d}{n-d+1} \chi_S + \cdots,
\]
 where the dots hide a linear combination of other monomials.
 
 If we substitute $x_i := \frac{x_i-p}{\sqrt{p(1-p)}}$ in $\tilde\chi_S$ then we get a harmonic multilinear polynomial of degree $d$ which agrees with $\omega_S$ on the slice $\binom{[n]}{pn}$, and so equals $\tilde\omega_S$. Substituting this in the preceding formula, we deduce that
\[
 \omega_S - \tilde\omega_S = \frac{d}{n-d+1} \omega_S + \psi,
\]
 where $\psi$ is orthogonal to $\omega_S$ with respect to $\mu_p$.
 
 Since $\tilde\omega_S$ agrees with $\omega_S$ on the slice, we can compute $\|\tilde\omega_S\|_{\nu_{pn}}^2$ explicitly. Simple estimates show that $\|\tilde\omega_S\|_{\nu_{pn}}^2 \leq 1 + O(\frac{d^2}{p(1-p)n})$ (see~\cite[Lemma 4.2]{FKMW} for the details). Corollary~\ref{cor:approximate-norm} and Lemma~\ref{lem:norms} imply the same bound on $\|\tilde\omega_S\|_{\mu_p}^2$. It follows that
\begin{multline*}
 \|\omega_S - \tilde\omega_S\|_{\mu_p}^2 = \frac{d}{n-d+1} + \|\psi\|_{\mu_p}^2 = \frac{d}{n-d+1} + \|\tilde\omega_S\|_{\mu_p}^2 - \left(1 - \frac{d}{n-d+1}\right) \leq \\
 \frac{2d}{n-d+1} + O\left(\frac{d^2}{p(1-p)n}\right) = O\left(\frac{d^2}{p(1-p)n}\right).
\end{multline*}
 This completes the proof.
\end{proof}

Armed with this result, we can prove Theorem~\ref{thm:kindler-safra}.

\providecommand{\lowf}{\breve{f}}
\begin{proof}
 Let $\lowf = f^{\leq d}$ (that is, $\lowf = \sum_{i \leq d} f^{=i}$), so that $\EE_{\nu_{pn}}[(|\lowf|-1)^2] \leq \EE_{\nu_{pn}}[(\lowf-f)^2] = \epsilon$. Notice that $\lowf$ is a harmonic multilinear polynomial of degree at most $d$.
 We would like to apply Theorem~\ref{thm:invariance-sq} to $\lowf$, which is possible if $n \geq d^2/(p(1-p)K^2)$. If this is not the case then $n \leq (p(1-p))^{-O(d)}$ (using $1/p(1-p) \geq 2$), and so the theorem is trivial (we can take $h = f$).
 
 Assume therefore that $n$ is large enough. Theorem~\ref{thm:invariance-sq} implies that
\[
 \EE_{\mu_p}[(|\lowf|-1)^2] \leq \underbrace{\epsilon + O\left(\sqrt[4]{\frac{d}{\sqrt{p(1-p)n}}}(p(1-p))^{-O(d)}\right)}_{\epsilon_1}.
\]
 Proposition~\ref{pro:kindler-safra} implies that there exists a cube-Boolean function $g$ on a set $J$ of $M = O((p(1-p))^{-O(d)})$ coordinates such that $\epsilon_2 \eqdef \EE_{\mu_p}[(\lowf-g)^2] = O((p(1-p))^{-O(d)} \epsilon_1)$. The function $g$ is also slice-Boolean, but it is not necessarily harmonic. Let $\tilde{g}$ be its harmonic projection on $\binom{[n]}{pn}$; this will be our choice for $h$. Note that $\tilde{g}$ also depends only on the coordinates in $J$, and in particular it has degree at most $M$ (in fact, Lemma~\ref{lem:harmonic-degree} implies that $\deg \tilde{g} \leq \deg g$). Invoking Theorem~\ref{thm:harmonic-projection}, we see that $\|g-\tilde{g}\|^2_{\mu_p} = O(\frac{M^2 2^M}{p(1-p)n})$, and so
\[
 \epsilon_3 \eqdef \|\lowf - \tilde{g}\|^2_{\mu_p} = O\left(\frac{M^2 2^M}{p(1-p)n} + \epsilon_2\right).
\]
 Corollary~\ref{cor:approximate-norm} and Lemma~\ref{lem:norms} imply that $\|\lowf - \tilde{g}\|^2_{\nu_{pn}} = O(\epsilon_3)$, using the fact that $\deg(\lowf-\tilde{g}) \leq M$. The $L_2$ triangle inequality shows that
\[
 \|f-\tilde{g}\|^2_{\nu_{pn}} \leq O(\|f-\lowf\|^2_{\nu_{pn}} + \epsilon_3) = 
 O((p(1-p))^{-O(d)}\epsilon) + O\left(\sqrt[4]{\frac{d}{\sqrt{p(1-p)n}}}(p(1-p))^{-O(d)}\right) + O\left(\frac{M^2 2^M}{p(1-p)n}\right). \qedhere
\]

\end{proof}

The proof of~\cite[Theorem 7.5]{FKMW} contains an additional argument guaranteeing that $\deg h \leq d$. The same argument can be applied here. The idea is that there are finitely many Boolean functions on $(p(1-p))^{-O(d)}$ coordinates, and each of them of degree larger than $d$ has (as $n\to\infty$) constant distance from all Boolean functions of degree at most $d$. Hence if $\epsilon$ is small enough, $g$ must have degree at most $d$. We refer the reader to~\cite{FKMW} for the complete details.

\section{Multilinear functions} \label{sec:multilinear}

Theorem~\ref{thm:invariance} only applies to harmonic multilinear polynomials. The harmonicity condition is crucial here. Indeed, a polynomial such as $\sum_{i=1}^n x_i$ behaves very differently on the Boolean cube (where it has a non-trivial distribution) and on a slice (where it is constant). Nevertheless, we are able to recover similar theorems by looking at several slices at once. Our first invariance result, which we call \emph{norm invariance}, states that if two low-degree multilinear polynomials $f,g$ are close in L2 norm for enough ``well-separated'' slices, then they are close in $L_2$ norm over the entire Boolean cube. Our second invariance result, which we call \emph{interpolation invariance}, gives a recipe for constructing the distribution of a low-degree multilinear polynomial $f$ from its profile on several ``well-separated'' slices, where by \emph{profile} we mean the distribution of $f$ on several \emph{coupled} slices, just as in Section~\ref{sec:invariance}.

Our main technical tool is a theorem of Blekherman~\cite{Blekherman}, which states that any degree~$d$ multilinear polynomial $P$ corresponds uniquely to a degree~$d$ polynomial $Q$ in the variables $x_1,\ldots,x_n,\xsum$ such that
\begin{enumerate}[(a)]
 \item $P(x_1,\ldots,x_n) = Q(x_1,\ldots,x_n,x_1+\cdots+x_n)$ for any point in the Boolean cube $\{0,1\}^n$.
 \item For each $e$, the coefficient of $\xsum^e$ is a \emph{harmonic} multilinear polynomial (of degree at most $d-e$).
\end{enumerate}
This theorem allows us to reduce the analysis of arbitrary multilinear polynomials to that of harmonic ones.

We state Blekherman's theorem in Section~\ref{sec:blekherman}. After preparatory work in Section~\ref{sec:vandermonde}, we prove our norm invariance theorems in Section~\ref{sec:multilinear-norm}, and our interpolation invariance theorems in Section~\ref{sec:multilinear-interpolation}.

\smallskip

There are several principal results in this section, which we now highlight.

\paragraph{Norm invariance (\S\ref{sec:multilinear-norm})}
The main results are Theorem~\ref{thm:norm-invariance-1}, which bounds the $L_2$ norm of a low-degree multilinear polynomial on the Boolean cube with respect to its $L_2$ norm on several well-separated slices, and Theorem~\ref{thm:norm-invariance-2}, which goes in the other direction, bounding the $L_2$ norm of a low-degree multilinear polynomial on a slice in terms of its $L_2$ norm on the Boolean cube and the centrality of the slice.

Both results are combined in Corollary~\ref{cor:norm-invariance}, which states (roughly) that a low-degree multilinear polynomial has small $L_2$ norm on the Boolean cube if and only if it has small $L_2$ norm on several well-separated slices. Another conclusion, Corollary~\ref{cor:norm-invariance-2}, states that two low-degree multilinear polynomials are close in $L_2$ norm on the Boolean cube if and only if they are close in $L_2$ norm on several well-separated slices.

\paragraph{Interpolation invariance(\S\ref{sec:multilinear-interpolation})}
The main results are Theorem~\ref{thm:interpolation-1} and Corollary~\ref{cor:interpolation-1-levy-1}, which show how to estimate the distribution of a low-degree multilinear polynomial on a given slice given its distribution on several coupled well-separated slices, and Theorem~\ref{thm:interpolation-2} and Corollary~\ref{cor:interpolation-2-levy-1}, which similarly show how to estimate the distribution of a low-degree multilinear polynomial on the entire Boolean cube given its distribution on several coupled well-separated slices.

These results imply that if two low-degree multilinear polynomials have a similar distribution on several coupled well-separated slices then they have a similar distribution on other slices and on the entire Boolean cube, as we show in Corollary~\ref{cor:interpolation-1-levy-2} and in Corollary~\ref{cor:interpolation-2-levy-2}, respectively.

\subsection{Blekherman's theorem} \label{sec:blekherman}

Our starting point is a theorem of Blekherman~\cite{Blekherman} quoted in Lee et al.~\cite{LPWY}. For completeness, we prove this theorem in Section~\ref{sec:harmonic-algebra}.

\begin{theorem} \label{thm:blekherman}
 Let $f$ be a multilinear polynomial over $x_1,\ldots,x_n$ of degree $d \leq n/2$, and define $\xsum := \sum_{i=1}^n x_i$.
 There exist harmonic multilinear polynomials $\fc{f}{0},\ldots,\fc{f}{d}$ over $x_1,\ldots,x_n$, where $\deg \fc{f}{i} \leq d-i$, such that
\[
 f(x_1,\ldots,x_n) \equiv \sum_{i=0}^d \fc{f}{i}(x_1,\ldots,x_n) \xsum^i \pmod{I}, \text{ where } I=\langle x_1^2-x_1,\ldots,x_n^2-x_n \rangle,
\]
 or equivalently, both sides agree on every point of the Boolean cube $\{0,1\}^n$.
 Moreover, this representation is unique.
\end{theorem}

For our purposes, it will be better to consider $f$ as a polynomial in $(\xsum - np)/\sqrt{np(1-p)}$ rather than in $\xsum$.

\begin{corollary} \label{cor:blekherman}
 Let $f$ be a multilinear polynomial over $x_1,\ldots,x_n$ of degree $d \leq n/2$, let $p \in (0,1)$, and define $\xstd := (\sum_{i=1}^n x_i - np)/\sqrt{np(1-p)}$.
 There exist harmonic multilinear polynomials $\fc{f}{0},\ldots,\fc{f}{d}$ over $x_1,\ldots,x_n$, where $\deg \fc{f}{i} \leq d-i$, such that
\[
 f(x_1,\ldots,x_n) \equiv \sum_{i=0}^d \fc{f}{i}(x_1,\ldots,x_n) \xstd^i \pmod{I}, \text{ where } I=\langle x_1^2-x_1,\ldots,x_n^2-x_n \rangle,
\]
 or equivalently, both sides agree on every point of the Boolean cube $\{0,1\}^n$.
 Moreover, this representation is unique.
\end{corollary}
\begin{proof}
 Follows from the fact that $\xsum$ and $\xstd$ are affine shifts of one another.
\end{proof}

We call the representation of $f$ in Corollary~\ref{cor:blekherman} its \emph{Blekherman expansion} with respect to $p$, and $\fc{f}{0},\ldots,\fc{f}{d}$ its \emph{Blekherman coefficients} with respect to $p$.
If we substitute $\xstd = \sigma$ in the Blekherman expansion then we get a harmonic multilinear polynomial of degree at most $d$ which agrees with $f$ on the slice $\binom{[n]}{k}$, where $k = np + \sigma\sqrt{np(1-p)}$. We denote this function by $\fs{f}{\sigma}$. Note that this notation depends on $p$.

The Blekherman expansion is \emph{linear} in the sense that if $h = \alpha f + \beta g$ then $\fc{h}{e} = \alpha\fc{f}{e} + \beta\fc{g}{e}$ and $\fs{h}{\sigma} = \alpha\fs{f}{\sigma} + \beta\fs{f}{\sigma}$. This immediately follows from its uniqueness and the fact that harmonic functions are closed under taking linear combinations.

\subsection{Vandermonde interpolation} \label{sec:vandermonde}

Our arguments will involve extracting the Blekherman coefficients $\fc{f}{i}$ given $\fs{f}{\sigma}$ for various values of $\sigma$. We will consider the simple setting in which we are given $d+1$ values of $\sigma$, and in that case the problem translates to solving a system of linear equations whose coefficient matrix is a Vandermonde matrix. The quality of the reconstruction will depend on the magnitude of the entries in the inverse matrix, which we estimate using a result of Turner~\cite{Turner}.

\begin{proposition}[{Turner~\cite{Turner}}] \label{pro:turner}
 Let $\xi_1,\ldots,\xi_n$ be arbitrary real numbers, and consider the $n\times n$ Vandermonde matrix $V$ given by $V_{ij} = \xi_i^{j-1}$, where $1 \leq i,j \leq n$. The inverse of $V$ is given by $V^{-1} = UL$, where
\[
 L_{ij} =
\begin{cases}
 0 & \text{if } i < j, \\
 \displaystyle\prod_{\substack{k=1 \\ k \neq j}}^i \frac{1}{\xi_j - \xi_k} & \text{otherwise},
\end{cases}
 \qquad
 U_{ij} =
 \begin{cases}
 0 & \text{if } i > j, \\
 1 & \text{if } i = j, \\
 U_{(i-1)(j-1)} - \xi_{j-1} U_{i(j-1)} & \text{if } i < j, \text{ where } U_{0(j-1)} = 0.
 \end{cases}
\]
\end{proposition}

This proposition implies the following interpolation result.

\begin{theorem} \label{thm:vandermonde}
 Suppose that $f$ is a multilinear polynomial over $x_1,\ldots,x_n$ of degree $d \leq n/2$ with Blekherman coefficients $\fc{f}{0},\ldots,\fc{f}{d}$ with respect to some $p \in (0,1)$.
 Let $\sigma_1,\ldots,\sigma_{d+1}$ be $d+1$ distinct values, and define
\[
 \eta = \min(1,\min_{i \neq j} |\sigma_i - \sigma_j|), \qquad M = \max(1,\max_i |\sigma_i|).
\]
 For $0 \leq e \leq d$ and $1 \leq i \leq d+1$ there exist coefficients $c_{ei}$ of magnitude $|c_{ei}| \leq (4M/\eta)^d$ such that for all $0 \leq e \leq d$,
\[
 \fc{f}{e} = \sum_{i=1}^{d+1} c_{ei} \fs{f}{\sigma_i}.
\]
\end{theorem}
\begin{proof}
 Let $V$ be the Vandermonde matrix for $\sigma_1,\ldots,\sigma_{d+1}$, so that
\[
 \begin{bmatrix} \fs{f}{\sigma_1} \\ \vdots \\ \fs{f}{\sigma_{d+1}} \end{bmatrix} = V \begin{bmatrix} \fc{f}{0} \\ \vdots \\ \fc{f}{d} \end{bmatrix}.
\]
 Inverting $V$, this shows that
\[
 V^{-1} \begin{bmatrix} \fs{f}{\sigma_1} \\ \vdots \\ \fs{f}{\sigma_{d+1}} \end{bmatrix} = \begin{bmatrix} \fc{f}{0} \\ \vdots \\ \fc{f}{d} \end{bmatrix}. 
\]
 We can thus take $c_{ei} = V^{-1}_{(e+1)i}$. It remains to bound the magnitude of the entries of $V^{-1}$.
 
 Let $L,U$ be the matrices in Proposition~\ref{pro:turner}. The formula for $L$ implies that all of its entries have magnitude at most $(1/\eta)^d$. As for $U$, we will prove by induction that when $i \leq j$, $|U_{ij}| \leq \binom{j-1}{i-1} M^{j-i}$. This is clearly true when $i = j$. When $i = 1$, $|U_{1j}| = \prod_{k=1}^{j-1} |\sigma_k| \leq M^{j-1}$.
 The inductive step follows from
\[
 |U_{ij}| \leq |U_{(i-1)(j-1)}| + M |U_{i(j-1)}| \leq \binom{j-2}{i-2} M^{j-i} + M \cdot \binom{j-2}{i-1} M^{j-1-i} = \binom{j-1}{i-1} M^{j-i}.
\]
 It follows that all entries of $U$ are bounded in magnitude by $2^d M^d \leq \frac{4^d}{d+1} \cdot M^d$. The theorem follows.
\end{proof}

\subsection{Norm invariance} \label{sec:multilinear-norm}

We are now ready to prove our norm invariance principle.
Our argument will require a few auxiliary results. We start with an estimate on the central moments of binomial distributions.

\begin{lemma} \label{lem:binomial-moments}
 Let $\xsum \sim \bin(n,p)$ and $\xstd = \frac{\xsum-np}{\sqrt{np(1-p)}}$. For all $d \geq 0$,
\[
 \EE[\xstd^{2d}] \leq \frac{2d!}{(2p(1-p))^d}.
\]
\end{lemma}
\begin{proof}
 Hoeffding's inequality states that
\[
 \Pr[|\xsum - np| \geq t] \leq 2e^{-2t^2/n} \Longrightarrow \Pr[|\xstd| \geq t] \leq 2e^{-2p(1-p)t^2}.
\]
 Plugging this in the general formula $\EE[|X|] = \int_0^\infty \Pr[|X| \geq t] \, \rmd t$, we get
\[
 \EE[\xstd^{2d}] = \int_0^\infty \Pr[\xstd^{2d} \geq t] \, \rmd t \leq 2\int_0^\infty e^{-2p(1-p)t^{1/d}} \, \rmd t.
\]
 Substituting $s = 2p(1-p)t^{1/d}$, we have $t = (s/2p(1-p))^d$ and so $\rmd t/\rmd s = d(s/2p(1-p))^{d-1}/2p(1-p)$, implying
\[
 \EE[\xstd^{2d}] \leq 2 \cdot (2p(1-p))^{-d} \int_0^\infty e^{-s} ds^{d-1} \, \rmd s = 2 \cdot (2p(1-p))^{-d} d!,
\]
using the classical integral formula for the $\Gamma$ function:
\[
 \int_0^\infty e^{-s} s^{d-1} \, \rmd s = \Gamma(d) = (d-1)!. \qedhere
\]
\end{proof}

We comment that for every fixed $p$ and $d$, as $n\to\infty$, $\EE[\xstd^{2d}]$ converges to $\EE[\nor(0,1)^{2d}] = \frac{(2d)!}{2^dd!}$.

We also need an anti-concentration result for binomial distributions, which follows from the Berry--Esseen theorem.

\begin{lemma} \label{lem:binomial-lb}
 Let $\xsum \sim \bin(n,p)$ and $\xstd = \frac{\xsum-np}{\sqrt{np(1-p)}}$. For all $a < b$,
\[
 \Pr[\xstd \in (a,b)] \geq \frac{b-a}{\sqrt{2\pi}} e^{-\max(a^2,b^2)/2} - \frac{1}{\sqrt{np(1-p)}}.
\]
\end{lemma}
\begin{proof}
 The Berry--Esseen theorem states that the cumulative distribution functions of $\xstd$ and of the standard normal distribution $\nor(0,1)$ differ by at most $C\rho/\sigma^3\sqrt{n}$, where $C<1$ is an absolute constant, $\rho = \EE[|\Ber(p)-p|^3] = p(1-p)(p^2+(1-p)^2)$, and $\sigma^2 = \VV[\Ber(p)] = p(1-p)$. The result follows from $\rho \leq p(1-p)/2$.
\end{proof}

We are now ready to tackle norm invariance. We start by bounding the norm of $f$ on the Boolean cube given its norm on several well-separated slices.

\begin{definition} \label{def:system}
 Fix $p \in (0,1)$ and $n$. A set $0 \leq k_1,\ldots,k_r \leq n$ is said to be an $(\eta,M)$-system, for $\eta \leq 1$ and $M \geq 1$, if the following two conditions hold for $\sigma_i = \frac{k_i-np}{\sqrt{np(1-p)}}$:
\begin{enumerate}[(a)]
 \item For every $i \neq j$, $|\sigma_i - \sigma_j| \geq \eta$.
 \item For every $i$, $|\sigma_i| \leq M$.
\end{enumerate}
\end{definition}

\begin{theorem} \label{thm:norm-invariance-1}
 Let $p \in (0,1)$, and let $f$ be a multilinear polynomial over $x_1,\ldots,x_n$ of degree $d \leq \sqrt{np(1-p)}$. Let $k_1,\ldots,k_{d+1}$ be an $(\eta,M)$-system, where $M \leq \sqrt{np(1-p)}/2$, and suppose that $\|f\|_{\nu_{k_i}} \leq 1$ for all $1 \leq i \leq d+1$. Then
\[
 \|f\|_{\mu_p} = O\left(\frac{d}{p(1-p)} \cdot \frac{M}{\eta}\right)^{O(d)}.
\]
 Moreover, the Blekherman coefficients $\fc{f}{0},\ldots,\fc{f}{d}$ of $f$ with respect to $p$ satisfy, for $0 \leq e \leq d$,
\[
 \|\fc{f}{e}\|_{\mu_p} \leq O(M/\eta)^d.
\]
\end{theorem}
\begin{proof}
 If $d = 0$ then $f$ is constant, so we can assume that $d \geq 1$.

 We are given that for each $i$, $\|\fs{f}{\sigma_i}\|_{\nu_{k_i}} \leq 1$. Since $\fs{f}{\sigma_i}$ is a harmonic multilinear polynomial of degree at most $d$, Corollary~\ref{cor:approximate-norm} and Lemma~\ref{lem:norms} show that for $q = k_i/n$,
\[
 \|\fs{f}{\sigma_i}\|_{\mu_p}^2 \leq \frac{(2p(1-p))^d}{(2q(1-q))^d} \left(1 - O\left(\frac{d^2}{q(1-q)n}\right)\right)^{-1} = O(1) \cdot \left(\frac{p(1-p)}{q(1-q)}\right)^d.
\]
 Since $M \leq \sqrt{np(1-p)}/2$, we have
\[
 |p-q| \leq \frac{M \sqrt{np(1-p)}}{n} \leq \frac{p(1-p)}{2},
\]
 which implies that
\[
 \frac{p}{q} \leq \frac{p}{p-p(1-p)/2} = \frac{2}{1+p} \leq 2.
\]
 Similarly, $(1-p)/(1-q) \leq 2$.
 This shows that
\[
 \|\fs{f}{\sigma_i}\|_{\mu_p}^2 \leq O(1)^d.
\]

 Theorem~\ref{thm:vandermonde} shows that there exist coefficients $c_{ei}$ of magnitude at most $(4M/\eta)^d$ such that for each $e$, $\fc{f}{e} = \sum_i c_{ei} \fs{f}{\sigma_i}$. It follows that for each $e$,
\[
 \|\fc{f}{e}\|_{\mu_p} \leq \sum_{i=1}^{d+1} |c_{ei}| \|\fs{f}{\sigma_i}\|_{\mu_p} \leq (d+1) \cdot O(M/\eta)^d \cdot O(1)^d = O(M/\eta)^d.
\]
 Since $f = \sum_{e=0}^d \xstd^e \fc{f}{e}$, the Cauchy--Schwarz inequality implies that
\[
 \|f\|_{\mu_p} \leq \sum_{e=0}^d \|\xstd^e \fc{f}{e}\|_{\mu_p} \leq \sum_{e=0}^d \sqrt{\|\xstd^{2e}\|_{\mu_p}} \sqrt{\|\fc{f}{e}^2\|_{\mu_p}}.
\]
 Lemma~\ref{lem:binomial-moments} shows that
\[
 \sqrt{\|\xstd^{2e}\|_{\mu_p}} = \sqrt[4]{\EE_{\mu_p}[\xstd^{4e}]} = O\left(\frac{d}{p(1-p)}\right)^{d/2}.
\]
 Hypercontractivity, in the form of Lemma~\ref{lem:low-degree}, implies that with respect to $\mu_p$,
\[
 \sqrt{\|\fc{f}{e}^2\|} = \|\fc{f}{e}\|_4 \leq O(p(1-p))^{-O(e)} \|\fc{f}{e}\|_2 = \left(\frac{M}{p(1-p)\eta}\right)^{O(d)}.
\]
 In total, we deduce that
\[
 \|f\|_{\mu_p} \leq (d+1) \cdot O\left(\frac{d}{p(1-p)}\right)^d \cdot \left(\frac{M}{p(1-p)\eta}\right)^{O(d)} = \left(\frac{d}{p(1-p)} \cdot \frac{M}{\eta}\right)^{O(d)}. \qedhere
\]
\end{proof}

We can also go in the other direction. In the statement of Theorem~\ref{thm:norm-invariance-2} and similar results below, we allow big O constants to depend on the fixed value of $p$ (and below, on the fixed value of other parameters).

\begin{theorem} \label{thm:norm-invariance-2}
 Fix $p \in (0,1)$, and let $f$ be a multilinear polynomial over $x_1,\ldots,x_n$ of degree $d \leq \sqrt{\log[np(1-p)/30]}$ satisfying $\|f\|_{\mu_p} \leq 1$. If $k = np + \sqrt{np(1-p)} \cdot \sigma$ then
\[
 \|f\|_{\nu_k} \leq e^{O(d^2)} (1+|\sigma|)^d e^{d|\sigma|/\sqrt{np(1-p)}}.
\]
\end{theorem}
\begin{proof}
 If $d = 0$ then $f$ is constant, so we can assume that $d \geq 1$.
 
 For $0 \leq i \leq d$, let $a_i = i-1/2$ and $b_i = i$. Lemma~\ref{lem:binomial-lb} shows that
\[
 \Pr[\xstd \in (a_i,b_i)] \geq \frac{1}{2\sqrt{2\pi}} e^{-d^2/2} - \frac{1}{\sqrt{np(1-p)}}.
\]
 The first summand is at least
\[
 \frac{1}{2\sqrt{2\pi}} e^{-\log[np(1-p)/30]/2} = \frac{1}{2\sqrt{2\pi}} \cdot \frac{1}{\sqrt{np(1-p)/30}} > \frac{1.09}{\sqrt{np(1-p)}},
\]
 and so
\[
 \Pr[\xstd \in (a_i,b_i)] = \Omega(e^{-d^2/2}).
\]
 In particular, the norm of $f$ restricted to $\xstd \in (a_i,b_i)$ is $O(e^{d^2/2})$, and so there must exist $k_i = np + \sqrt{np(1-p)} \cdot \sigma_i$ such that $\sigma_i \in (a_i,b_i)$ and $\|f\|_{\nu_{k_i}} = O(e^{d^2/2})$.
 
 The resulting system $k_0,\ldots,k_d$ is a $(1/2,d)$-system, and so Theorem~\ref{thm:norm-invariance-1} shows that the Blekherman coefficients $\fc{f}{0},\ldots,\fc{f}{d}$ satisfy
\[
 \|\fc{f}{e}\|_{\mu_p} \leq O(e^{d^2/2}) \cdot O(d)^d = e^{O(d^2)}.
\]
 Substituting a given value of $\sigma$, we deduce that
\[
 \|\fs{f}{\sigma}\|_{\mu_p} \leq \sum_{e=0}^d |\sigma|^e \|\fc{f}{e}\|_{\mu_p} \leq (1+|\sigma|)^d e^{O(d^2)}.
\]
 
 Corollary~\ref{cor:approximate-norm} and Lemma~\ref{lem:norms} show that for $q = k/n$,
\[
 \|\fs{f}{\sigma}\|_{\nu_k} \leq (1+|\sigma|)^d e^{O(d^2)} \frac{(2q(1-q))^{d/2}}{(2p(1-p))^{d/2}} \left(1 + O\left(\frac{d^2}{p(1-p)n}\right)\right)^{1/2} = (1+|\sigma|)^d e^{O(d^2)} \cdot \left(\frac{q(1-q)}{p(1-p)}\right)^{d/2}.
\]
 In order to estimate the final factor, note that $q = p + \sqrt{\frac{p(1-p)}{n}} \sigma$ and so
\[
 \frac{q}{p} = 1 + \sqrt{\frac{1-p}{p}} \frac{\sigma}{\sqrt{n}} \leq e^{\sqrt{\frac{1-p}{p}} \frac{\sigma}{\sqrt{n}}} \leq e^{\sigma/\sqrt{np(1-p)}}. 
\]
 Similarly $(1-q)/(1-p) \leq e^{-\sigma/\sqrt{np(1-p)}}$.
 Therefore
\[
 \left(\frac{q(1-q)}{p(1-p)}\right)^{d/2} \leq e^{d|\sigma|/\sqrt{np(1-p)}},
\]
 and the theorem follows.
\end{proof}

We can combine both results to obtain the following clean corollary.

\begin{corollary} \label{cor:norm-invariance}
 Fix $p \in (0,1)$ and $d \geq 1$, and suppose that $\cF$ is a collection of multilinear polynomials of degree at most $d$ on at least $\frac{30}{p(1-p)} e^{d^2}$ variables (different functions could depend on a different number of variables). The following three conditions are equivalent:
\begin{enumerate}[(a)]
 \item There exists a constant $C_1$ such that $\|f\|_{\mu_p} \leq C_1$ for all $f \in \cF$. \label{cor:norm-invariance:1}
 \item There exists a constant $C_2$ such that $\|f\|_{\nu_k} \leq \bigl(C_2(1+\frac{|k-np|}{\sqrt{n}})\bigr)^d$ for all $f \in \cF$ over $x_1,\ldots,x_n$ and for all $0 \leq k \leq n$. \label{cor:norm-invariance:2}
 \item There exists a constant $C_3$ such that $\|f\|_{\nu_k} \leq C_3$ for all $f \in \cF$ over $x_1,\ldots,x_n$ and for $k = \lfloor np + \sqrt{np(1-p)} \cdot \sigma \rfloor$ for $\sigma \in \{0,\ldots,d\}$. \label{cor:norm-invariance:3}
\end{enumerate}
\end{corollary}
\begin{proof}
 Suppose first that condition~\eqref{cor:norm-invariance:1} holds, and let $\sigma = \frac{k-np}{\sqrt{np(1-p)}}$, so that $|\sigma| \leq \sqrt{n/p(1-p)}$. Then Theorem~\ref{thm:norm-invariance-2} shows that
\[
 \|f\|_{\nu_k} \leq C_1 \cdot e^{O(d^2)} \cdot (1+|\sigma|)^d \cdot e^{d/p(1-p)} \leq (C_1 e^{O(d^2)} (1+|\sigma|))^d,
\]
 which implies condition~\eqref{cor:norm-invariance:2}.
 
 Suppose next that condition~\eqref{cor:norm-invariance:2} holds. In particular, for $\sigma_1,\ldots,\sigma_{d+1}=0,\ldots,d$ and $k_i = np+\sqrt{np(1-p)} \cdot \sigma_i$ it is the case that $\|f\|_{\nu_{k_i}} \leq \bigl(C_2(1+d\sqrt{p(1-p)})\bigr)^d$, which implies condition~\eqref{cor:norm-invariance:3} with $C_3 = \bigl(C_2(1+d\sqrt{p(1-p)})\bigr)^d$.
 
 Finally, suppose that condition~\eqref{cor:norm-invariance:3} holds.
 Since $0,1,\ldots,d$ is a $(1,d)$-system, Theorem~\ref{thm:norm-invariance-1} shows that
\[
 \|f\|_{\mu_p} = O\left(\frac{d^2}{p(1-p)}\right)^{O(d)} C_3,
\]
 which depends only on $p$ and $d$, implying condition~\eqref{cor:norm-invariance:1}.
\end{proof}


Here is a different interpretation of these results.

\begin{corollary} \label{cor:norm-invariance-2}
 Fix $p \in (0,1)$ and $d \geq 1$. There is a constant $C>0$ such that the following implications hold for any two multilinear polynomials of degree at most $d$ on $n \geq \frac{30}{p(1-p)} e^{d^2}$ variables:
\begin{enumerate}[(a)]
 \item If $\|f-g\|_{\mu_p} = \epsilon$ then for all $0 \leq k \leq n$, $\|f-g\|_{\nu_k} \leq \bigl(C(1+\frac{|k-np|}{\sqrt{n}})\bigr)^d \epsilon$.
 \item If $k_1,\ldots,k_{d+1}$ is an $(\eta,M)$-system and $\|f-g\|_{\nu_{k_i}} \leq \epsilon$ for $1 \leq i \leq d+1$ then $\|f-g\|_{\mu_p} \leq (CM/\eta)^{O(d)} \epsilon$.
\end{enumerate}
\end{corollary}
\begin{proof}
 The first statement follows directly from Theorem~\ref{thm:norm-invariance-2} (using $|\sigma| \leq \sqrt{n/p(1-p)}$), and the second statement follows directly from Theorem~\ref{thm:norm-invariance-1}. In both cases, we apply the theorems to $f-g$.
\end{proof}

\subsection{Interpolation invariance} \label{sec:multilinear-interpolation}

We move on to the interpolation invariance principle. Definition~\ref{def:setup} describes a coupling $\rX(0),\ldots,\rX(n)$ of the distributions $\nu_0,\ldots,\nu_n$, which we will use to define the \emph{profile} of a function.

\begin{definition} \label{def:profile}
 Fix parameters $p \in (0,1)$, $d,n \geq 0$, and let $k_1,\ldots,k_{d+1}$ be an $(\eta,M)$-system. The \emph{profile} of a multilinear function $f$ of degree at most $d$ with respect to the system $k_1,\ldots,k_{d+1}$ is the joint distribution of $f(\rX(k_1)),\ldots,f(\rX(k_{d+1}))$, which we denote by $\rf_1,\ldots,\rf_{d+1}$.
\end{definition}

The profile of a function $f$ allows us to recover its distribution on arbitrary slices, as reflected by Lipschitz functions.

\begin{theorem} \label{thm:interpolation-1}
 There exists a constant $K > 0$ such that the following holds.
 Let $p \in (0,1)$, let $f$ be a multilinear polynomial on $x_1,\ldots,x_n$ of degree $1 \leq d \leq \sqrt{Knp(1-p)}$, let $k_1,\ldots,k_{d+1}$ be an $(\eta,M)$-system for $M \leq \sqrt{np(1-p)}/2$, and let $\rf_1,\ldots,\rf_{d+1}$ be the profile of $f$ with respect to this system. Suppose that $\|f\|_{\nu_{k_i}} \leq 1$ for $1 \leq i \leq d+1$. For every slice $k = np + \sqrt{np(1-p)} \cdot \sigma$ such that $L := \max_i |\sigma-\sigma_i| \leq \frac{\sqrt{np(1-p)}}{8d}$, and for any Lipschitz function $\varphi$,
\[
 |\EE_{\nu_k}[\varphi(f)] - \EE[\varphi(\fs{\rf}{k})]| = O\left((1+|\sigma|)\frac{M}{\eta}\right)^d \sqrt{\frac{L}{\sqrt{np(1-p)}}}, \quad \text{where } \fs{\rf}{k} = \sum_{i=1}^{d+1} \gamma_{ki} \rf_i,
\]
 for some constants $\gamma_{ki}$ depending on $p,n,d$ satisfying
\[
 |\gamma_{ki}| = O\left((1+|\sigma|)\frac{M}{\eta}\right)^d.
\]
\end{theorem}
\begin{proof}
 Let $p_i = k_i/n = p + \sqrt{\frac{p(1-p)}{n}} \sigma_i$ (where $\sigma_i = \frac{k_i-np}{\sqrt{np(1-p)}}$), and let $q = k/n = p + \sqrt{\frac{p(1-p)}{n}} \sigma$. The condition on $L$ guarantees that $|p_i-q| = \sqrt{\frac{p(1-p)}{n}} |\sigma_i-\sigma| \leq \frac{p(1-p)}{8d}$. The condition on $M$ guarantees that $|p_i-p| \leq \sqrt{\frac{p(1-p)}{n}} M \leq p(1-p)/2$, and so $\tfrac{1}{4} p(1-p) \leq p_i(1-p_i) \leq \tfrac{9}{4} p(1-p)$. In particular, $|p_i-q| \leq \frac{p(1-p)}{8d} \leq \frac{p_i(1-p_i)}{2d}$.
 Lemma~\ref{lem:coupling} (applied with $p:=p_i$ and $q:=q$) shows that
\[
 \EE[|f(\rX(k_i)) - f(\rX(k))|]^2 \leq \EE[(f(\rX(k_i)) - f(\rX(k)))^2] = O\left(\frac{dL}{\sqrt{np(1-p)}}\right).
\]

 Let $\fc{f}{0},\ldots,\fc{f}{d}$ be the Blekherman coefficients of $f$, and let $c_{ei}$ be the coefficients given by Theorem~\ref{thm:vandermonde}, so that $|c_{ei}| \leq (4M/\eta)^d$. Theorem~\ref{thm:vandermonde} shows that for $0 \leq e \leq d$,
\[
 \EE[|\fc{f}{e}(\rX(k)) - \sum_{i=1}^{d+1} c_{ei} f(\rX(k_i))|] \leq \sum_{i=1}^{d+1} |c_{ei}| \EE[|f(\rX(k_i)) - f(\rX(k))|] \leq (d+1) \cdot \left(\frac{4M}{\eta}\right)^d \cdot O\left(\sqrt{\frac{dL}{\sqrt{np(1-p)}}}\right).
\]
 Since $\fs{f}{\sigma} = \sum_{e=0}^d \sigma^e \fc{f}{e}$, we conclude that
\[
 \EE[|f(\rX(k)) - \sum_{e=0}^d \sigma^e \sum_{i=1}^{d+1} c_{ei} f(\rX(k_i))|] \leq (d+1)^2 \cdot (1+|\sigma|^d) \cdot \left(\frac{4M}{\eta}\right)^d \cdot O\left(\sqrt{\frac{dL}{\sqrt{np(1-p)}}}\right).
\]
 Rearrangement yields the statement of the theorem, with $\gamma_{ki} = \sum_{e=0}^d \sigma^e c_{ei}$.
\end{proof}

We immediately obtain a corollary for the L\'evy distance.

\begin{corollary} \label{cor:interpolation-1-levy-1}
 Under the setting of Theorem~\ref{thm:interpolation-1}, the L\'evy distance between $f(\nu_k)$ and $\fs{\rf}{k}$ is at most
\[
 O\left((1+|\sigma|)\frac{M}{\eta}\right)^{d/2} \sqrt[4]{\frac{L}{\sqrt{np(1-p)}}}.
\]
\end{corollary}
\begin{proof}
 The proof is identical to the proof of Corollary~\ref{cor:levy}.
\end{proof}

A striking form of this corollary compares two different profiles.

\begin{definition} \label{def:levy-metric}
 Let $X = (X_1,\ldots,X_r), Y = (Y_1,\ldots,Y_r)$ be two $r$-dimensional real distributions. The \emph{L\'evy--Prokhorov distance} between these distributions is the infimum value of $\epsilon$ such that for all Borel sets $A$,
\begin{enumerate}
\item $\Pr[X \in A] \leq \Pr[Y \in A^\epsilon] + \epsilon$, and
\item $\Pr[Y \in A] \leq \Pr[X \in A^\epsilon] + \epsilon$,
\end{enumerate}
where $A^\epsilon$ consists of all points at distance at most $\epsilon$ from $A$ in the $L^\infty$ metric.
\end{definition}

In fact, for all our results it suffices to consider \emph{halfspaces} for $A$, that is, sets of the form $\{ (x_1,\ldots,x_r) : \sum_{i=1}^r a_i x_i \leq b \}$. If we restrict the definition of L\'evy--Prokhorov distance in this way, then it coincides with the usual L\'evy distance in the one-dimensional case.

\begin{corollary} \label{cor:interpolation-1-levy-2}
 Under the setting of Theorem~\ref{thm:interpolation-1} for \emph{two} functions $f,g$, if the L\'evy--Prokhorov distance between the profiles of $f$ and $g$ is $\epsilon$, then the L\'evy distance between $f(\nu_k)$ and $g(\nu_k)$ is at most
\[
 O\left((1+|\sigma|)\frac{M}{\eta}\right)^{d/2} \sqrt[4]{\frac{L}{\sqrt{np(1-p)}}} + O\left((1+|\sigma|)\frac{M}{\eta}\right)^d \epsilon.
\]
\end{corollary}
\begin{proof}
 We start by bounding the L\'evy distance between $\rf[k]$ and $\rg[k]$. Given $t$, define the halfspace $A_t$ by
\[
 A_t = \{ (x_1,\ldots,x_{d+1}) : \sum_{i=1}^{d+1} \gamma_{ki} x_i \leq t \},
\]
 and notice that $A_t^\epsilon \subseteq A_{t + B\epsilon}$, where
\[
 B = \sum_{i=1}^{d+1} |\gamma_{ki}|= O\left((1+|\sigma|)\frac{M}{\eta}\right)^d.
\]
 Since the L\'evy--Prokhorov distance between the profiles of $f$ and $g$ is at most $\epsilon$, it follows that
\begin{multline*}	
 \Pr[\rf[k] \leq t] = \Pr[(\rf_1,\ldots,\rf_{d+1}) \in A_t] \leq \\
 \Pr[(\rg_1,\ldots,\rg_{d+1}) \in A_t^\epsilon] + \epsilon \leq
 \Pr[(\rg_1,\ldots,\rg_{d+1}) \in A_{t+B\epsilon}] + \epsilon =
 \Pr[\rg[k] \leq t+B\epsilon] + \epsilon.
\end{multline*}
 This shows that the L\'evy distance between $\rf[k]$ and $\rg[k]$ is at most $B\epsilon$. 

The result now follows from Corollary~\ref{cor:interpolation-1-levy-1} and the triangle inequality for the L\'evy distance.
\end{proof}

By combining different slices, we can obtain a similar result for $\mu_p$.

\begin{theorem} \label{thm:interpolation-2}
 There exists a constant $K > 0$ such that the following holds.
 Let $p \in (0,1)$, let $f$ be a multilinear polynomial on $x_1,\ldots,x_n$ of degree $1 \leq d \leq K\sqrt{np(1-p)/\log[np(1-p)]}$, let $k_1,\ldots,k_{d+1}$ be an $(\eta,M)$-system for $M \leq \sqrt{np(1-p)}/(9d)$, and let $\rf_1,\ldots,\rf_{d+1}$ be the profile of $f$ with respect to this system. Suppose that $\|f\|_{\nu_{k_i}} \leq 1$ for $1 \leq i \leq d+1$.
 
 For $s := \sqrt{3\log[np(1-p)]}$, define a distribution $\rf$ as follows: choose $\sigma \sim \frac{\bin(n,p)-np}{\sqrt{np(1-p)}}$ conditioned on $|\sigma| \leq s$, and let $\rf = \fs{\rf}{\sigma}$, as in Theorem~\ref{thm:interpolation-1}. Then for any Lipschitz function $\varphi$,
\[
 |\EE_{\mu_p}[\varphi(f)] - \EE[\varphi(\rf)]| = O\left(\frac{d}{p(1-p)} \cdot \frac{M}{\eta}\right)^{O(d)} \frac{1}{\sqrt[4]{np(1-p)}}.
\]
\end{theorem}
\begin{proof}
 We can assume, without loss of generality, that $\varphi(0) = 0$.

 Since $M+s \leq \frac{\sqrt{np(1-p)}}{8d}$ due to the bound on $M$ and our choice of $s$, Theorem~\ref{thm:interpolation-1} implies that
\[
 \sum_{\substack{k = np+\sqrt{np(1-p)}\cdot\sigma\colon\\ |\sigma| \leq s}} \Pr[\bin(n,p)=k] |\EE_{\nu_k}[\varphi(f)] - \EE[\varphi(\fs{\rf}{\sigma})]| =
 O\left(\frac{M}{\eta}\right)^d \frac{1}{\sqrt[4]{np(1-p)}} \EE[\underbrace{(1+|\sigma|)^d \sqrt{M+|\sigma|}}_C],
\]
 where $\sigma \sim \frac{\bin(n,p)-np}{\sqrt{np(1-p)}}$. When $|\sigma| \leq 1$, $C \leq 2^{d+1} \sqrt{M}$, and when $|\sigma| \geq 1$, $C \leq 2^d |\sigma|^d (\sqrt{M} + \sqrt{|\sigma|}) \leq 2^{d+1} |\sigma|^{2d} \sqrt{M}$. Therefore Lemma~\ref{lem:binomial-moments} implies that
\[
 \EE[C] \leq 2^{d+1} \sqrt{M} (1 + \EE[|\sigma|^{2d}]) = O\left(\frac{d}{p(1-p)}\right)^d \sqrt{M}.
\]
 We conclude that
\[
 \biggl|\sum_{\substack{k = np+\sqrt{np(1-p)}\cdot\sigma\colon\\ |\sigma| \leq s}} \Pr[\bin(n,p)=k] \EE_{\nu_k}[\varphi(f)] - (1-\epsilon_1) \EE[\varphi(\rf)]\biggr| =
 O\left(\frac{d}{p(1-p)} \cdot \frac{M}{\eta}\right)^d \frac{1}{\sqrt[4]{np(1-p)}},
\]
 where $\epsilon_1 := \Pr[|\bin(n,p)-np| > \sqrt{np(1-p)} \cdot s]$. Dividing this bound by $1-\epsilon_1$ and using $1-\frac{1}{1-\epsilon_1} = O(\epsilon_1)$, we deduce that
\[
 |\EE_{\mu_p}[\varphi(f)] - \EE[\varphi(\rf)]| \leq O(\epsilon_1) \EE[\varphi(f)] + \underbrace{\EE_{x \sim \mu_p}[\varphi(f)\charf{|\sum_i x_i - np| > s\sqrt{np(1-p)}}]}_{\epsilon_2} + O\left(\frac{d}{p(1-p)} \cdot \frac{M}{\eta}\right)^d \frac{1}{\sqrt[4]{np(1-p)}}.
\]
 
 As in the proof of Theorem~\ref{thm:invariance}, Proposition~\ref{pro:chernoff} implies that $\epsilon_1 \leq 2e^{-s^2/6}$.
 Theorem~\ref{thm:norm-invariance-1} shows that
\[
 \|f\|_{\mu_p} = O\left(\frac{d}{p(1-p)} \cdot \frac{M}{\eta}\right)^{O(d)}.
\]
 This allows us to bound the first error term above, since $\EE[\varphi(f)] \leq \EE[|f|] \leq \|f\|$.
 The other error term is bounded in the proof of Theorem~\ref{thm:invariance} by $O(e^{-s^2/12})$. Altogether, we obtain
\[
 |\EE_{\nu_k}[\varphi(f)] - \EE[\varphi(\rf)]| \leq O\left(\frac{d}{p(1-p)} \cdot \frac{M}{\eta}\right)^{O(d)} \left[\frac{1}{\sqrt[4]{np(1-p)}} + e^{-s^2/12}\right].
\]
 Substituting the value for $s$, we deduce that
\[
 |\EE_{\nu_k}[\varphi(f)] - \EE[\varphi(\rf)]| \leq O\left(\frac{d}{p(1-p)} \cdot \frac{M}{\eta}\right)^{O(d)} \frac{1}{\sqrt[4]{np(1-p)}}. \qedhere
\]
\end{proof}

Just as before, we can obtain corollaries for the L\'evy distance.

\begin{corollary} \label{cor:interpolation-2-levy-1}
 Under the setting of Theorem~\ref{thm:interpolation-2}, the L\'evy distance between $f(\mu_p)$ and $\rf$ is
\[
 O\left(\frac{d}{p(1-p)} \cdot \frac{M}{\eta}\right)^{O(d)} \frac{1}{\sqrt[8]{np(1-p)}}.
\]
\end{corollary}
\begin{proof}
 The proof is identical to the proof of Corollary~\ref{cor:interpolation-1-levy-1}.
\end{proof}


\begin{corollary} \label{cor:interpolation-2-levy-2}
 Under the setting of Theorem~\ref{thm:interpolation-2} for \emph{two} functions $f,g$, if the L\'evy distance between the profiles of $f$ and $g$ is $\epsilon \leq 1/2$, then the L\'evy distance between $f(\mu_p)$ and $g(\mu_p)$ is at most
\[
  O\left(\frac{d}{p(1-p)} \cdot \frac{M}{\eta}\right)^{O(d)} \frac{1}{\sqrt[8]{np(1-p)}} + O\left(\sqrt{\frac{\log (1/\epsilon)}{p(1-p)}}\frac{M}{\eta}\right)^d \epsilon.
\]
\end{corollary}
\begin{proof}
 As in the proof of Corollary~\ref{cor:interpolation-1-levy-2}, the idea is to bound the L\'evy distance between $\rf$ and $\rg$. Let $\sigma$ have the distribution in Theorem~\ref{thm:interpolation-2}, and let $\tau$ be a threshold to be determined. The argument of Corollary~\ref{cor:interpolation-1-levy-2} shows that
\[
 \Pr[\rf[\sigma] \leq t] \leq \Pr[\rg[\sigma] \leq t + B_\sigma \epsilon] + \epsilon, \text{ where } B_\sigma = O\left((1+|\sigma|) \frac{M}{\eta}\right)^d.
\]
 Therefore
\begin{align*}
 \Pr[\rf \leq t] = \EE_{\sigma}[\Pr[\rf[\sigma] \leq t]] &\leq 
 \EE_{\sigma}[\Pr[\rg[\sigma] \leq t + B_\sigma \epsilon] \charf{|\sigma| \leq \tau}] + \Pr[|\sigma| > \tau] + \epsilon \\ &\leq
 \Pr[\rg \leq t + B_\tau \epsilon] + \Pr[|\sigma| > \tau] + \epsilon.
\end{align*}
 As in the proof of Lemma~\ref{lem:binomial-moments}, Hoeffding's inequality shows that
\[
 \Pr[|\sigma| > \tau] = O(\Pr[|\xstd| > \tau]) = O(e^{-2p(1-p)\tau^2}),
\]
 showing that the L\'evy distance between $\rf$ and $\rg$ is at most $\max(B_\tau\epsilon, O(e^{-2p(1-p)\tau^2}) + \epsilon)$. Choosing $\tau = \sqrt{\frac{\log(1/\epsilon)}{2p(1-p)}}$, we conclude that the L\'evy distance between $\rf$ and $\rg$ is at most
\[
 O\left(\sqrt{\frac{\log (1/\epsilon)}{p(1-p)}}\frac{M}{\eta}\right)^d \epsilon.
\]
 The result now follows from Corollary~\ref{cor:interpolation-2-levy-1} and the triangle inequality for the L\'evy distance.
\end{proof}

Theorem~\ref{thm:interpolation-2} states that we can recover the distribution of a low-degree multilinear polynomial on the Boolean cube from its distribution on a few coupled slices. Interpolation in the other direction is not possible: the distribution of a low-degree polynomial function on the Boolean cube doesn't determine its distribution on the various slices. For example, consider the following two functions, for even $n$:
\[
 f_1 = \frac{\sum_{i=1}^n x_i - np}{\sqrt{np(1-p)}}, \qquad f_2 = \frac{\sum_{i=1}^{n/2} (x_{2i-1} - x_{2i})}{\sqrt{np(1-p)}}.
\]
The central limit theorem shows that with respect to $\mu_p$, the distribution of both functions is close to $\nor(0,1)$. However, on the slice $\nu_{pn}$, the first function vanishes, while the second function also has a distribution close to $\nor(0,1)$, due to Theorem~\ref{thm:invariance}.

\section{Functions depending on few coordinates} \label{sec:few-coordinates}

In this section we prove an invariance principle for bounded functions depending on $o(n)$ coordinates. In contrast to our work so far, the functions in question need not be harmonic multilinear polynomials.
The invariance principle immediately follows from the following bound on total variation distance.

\begin{lemma} \label{lem:tv-few-coordinates}
 Let $p(1-p)n \to \infty$ and $m = o(p(1-p)n)$. Denote the projection of $\mu_p$ and $\nu_{pn}$ into the first $m$ coordinates by $\mu'_p$ and $\nu'_{pn}$.
 The total variation distance between $\mu'_p$ and $\nu'_{pn}$ is $o(1)$.
\end{lemma}
\begin{proof}
 Let $k = pn$, and consider the ratio $\rho(\ell)$ between the probability of a set of size $\ell$ under $\nu'_k$ and under $\mu'_p$:
 \[
  \rho(\ell) = \frac{\nu'_k([\ell])}{\mu'_p([\ell])} = \frac{\binom{n-m}{k-\ell}/\binom{n}{k}}{p^\ell (1-p)^{m-\ell}}.
 \]
 To understand the behavior of $\rho(\ell)$, we compute the ratio $\rho(\ell+1)/\rho(\ell)$:
 \[
  \frac{\rho(\ell+1)}{\rho(\ell)} = \frac{1-p}{p} \frac{k-\ell}{n-m-k+\ell+1}.
 \]
 Thus $\rho(\ell+1) > \rho(\ell)$ iff $(1-p)(k-\ell) > p(n-m-k+\ell+1)$ iff $\ell < p(m-1)$. We deduce that the largest value of $\rho(\ell)$ is obtained for $\ell_0 = pm$ (assuming for simplicity that this is indeed an integer).

 At the point $\ell_0$ we can estimate, using Stirling's approximation,
 \begin{align*}
  \rho(\ell_0) &= \frac{\binom{n-m}{p(n-m)}}{\binom{n}{pn}} p^{-pm} (1-p)^{-(1-p)m} \\ &=
  \frac{\frac{p^{-p(n-m)}(1-p)^{-(1-p)(n-m)}}{\sqrt{2\pi p(1-p) (n-m)}}e^{O(1/p(1-p)(n-m))}}{\frac{p^{-pn}(1-p)^{-(1-p)n}}{\sqrt{2\pi p(1-p) n}}e^{O(1/p(1-p)n)}} p^{-pm} (1-p)^{-(1-p)m} \\ &=
  \sqrt{\frac{n}{n-m}} e^{O(1/p(1-p)(n-m) - O(1/p(1-p)n)} = 1 + O\left(\frac{m}{p(1-p)n}\right).
 \end{align*}
 Altogether, we deduce that for all $\ell$, $\rho(\ell) \leq 1 + O(m/p(1-p)n)$. Therefore the total variation distance is
 \[
  \sum_{\ell\colon \rho(\ell)>1} \binom{m}{\ell} (\nu'_k([\ell]) - \mu'_p([\ell])) =
  \sum_{\ell\colon \rho(\ell)>1} \binom{m}{\ell} \mu'_p([\ell]) (\rho(\ell) - 1) = O\left(\frac{m}{p(1-p)n}\right). \qedhere
 \] 
\end{proof}

As an immediate corollary, we obtain an invariance principle for bounded functions depending on $o(n)$ coordinates.

\begin{theorem} \label{thm:invariance-few-coordinates}
 Let $f$ be a function on $\{0,1\}^n$ depending on $o(p(1-p)n)$ coordinates and satisfying $\|f\|_\infty \leq 1$. As $p(1-p)n \to \infty$,
 \[
  |\EE_{\nu_{pn}}[f] - \EE_{\mu_p}[f]| = o(1).
 \]
\end{theorem}
\begin{proof}
 Suppose that $f$ depends on $m = o(p(1-p)n)$ coordinates. Applying the triangle inequality shows that
\[
 |\EE_{\nu_{pn}}[f] - \EE_{\mu_p}[f]| \leq \sum_{x \in \{0,1\}^m} |\nu'_{pn}(x) - \mu'_p(x)| |f(x)| \leq \sum_{x \in \{0,1\}^m} |\nu'_{pn}(x) - \mu'_p(x)|,
\]
 the last expression being exactly the total variation distance between $\nu'_{pn}$ and $\mu'_p$.
\end{proof}

\begin{corollary} \label{cor:invariance-few-coordinates}
 Let $f$ be a function on $\{0,1\}^n$ depending on $o(p(1-p)n)$ coordinates. As $p(1-p)n \to \infty$, the CDF distance between $\nu_{pn}(f)$ and $\mu_p(f)$ tends to zero, that is,
\[
 \sup_{t \in \RR} |\Pr_{\nu_{pn}}[f < t] - \Pr_{\mu_p}[f < t]| = o(1).
\]
\end{corollary}
\begin{proof}
 Consider the functions $f_t = \charf{f < t}$, which satisfy $\|f_t\|_\infty = 1$ for all $t \in \RR$.
\end{proof}

\section{Functions with low total influence} \label{sec:low-influence}

In this section we prove an invariance principle for Boolean functions whose total influence is $o(\sqrt{p(1-p)n})$. In other words, we show that Boolean functions with total influence $o(\sqrt{p(1-p)n})$ cannot distinguish the slice from the cube.

For us a Boolean function is a function $f\colon \{0,1\}^n \to \{0,1\}$, and total influence is defined as follows: $\Inf[f] = p(1-p)n\Pr[f(x) \neq f(y)]$, where $x \sim \mu_p$ and $y$ is a random neighbor of $x$, chosen uniformly from all $n$ neighbors (this differs from the normalization in~\cite{FKMW} by a factor of $p(1-p)$). As in Section~\ref{sec:few-coordinates}, the functions we consider in this section need not be harmonic multilinear polynomials.

The upper bound $\sqrt{p(1-p)n}$ is necessary. Indeed, consider a threshold function with threshold $pn$. This function has total influence $\Theta(\sqrt{p(1-p)n})$, is constant on the slice, and is roughly balanced on the cube. The condition that the function is Boolean is also necessary (total influence can be extended to arbitrary functions, see for example~\cite[Section 8.4]{ODonnell}). The function $(x_1+\cdots+x_n)/\sqrt{p(1-p)n}$ has unit variance and unit total influence on the cube, is constant on the slice, and has a non-trivial distribution on the cube. Invariance does hold, however, for non-Boolean harmonic multilinear polynomials of degree $o(\sqrt{p(1-p)n})$, as Theorem~\ref{thm:invariance} shows.

We will use part of the setup of Definition~\ref{def:setup}, namely the random variables $\rX(s) \in \binom{[n]}{s}$ for $0 \leq s \leq n$. We also define two new random variables: $\rs \sim \bin(n,p)$ and $\rt \sim \bin(n-1,p)$, where $\bin(n,p)$ is the binomial distribution with $n$ trials and success probability $p$. The basic observation is that $\rX(\rs)$ is distributed according to the measure $\mu_p$ on the cube $\{0,1\}^n$ whereas $\rX(np)$ is distributed uniformly on the slice $\binom{[n]}{pn}$, and so our goal would be to bound the probability $\Pr[f(\rX(\rs)) \neq f(\rX(np))]$.

The following lemma bounds $\Pr[f(\rX(\rs)) \neq f(\rX(np))]$ using a hybrid argument.

\begin{lemma} \label{lem:hybrid}
 For every Boolean function $f$,
\begin{multline*}
 \Pr[f(\rX(\rs)) \neq f(\rX(np))] \leq \\ \sum_{s=np}^{n-1} \Pr[\rs \geq s+1] \Pr[f(\rX(s)) \neq f(\rX(s+1))] + \sum_{s=0}^{np-1} \Pr[\rs \leq s] \Pr[f(\rX(s)) \neq f(\rX(s+1))].
\end{multline*}
\end{lemma}
\begin{proof}
 For an event $E$, let $\ib{E}$ denote the corresponding indicator. Let $\rI$ denote the interval whose endpoints are $\rs$ and $np$ (inclusive). Then
\begin{multline*}
 \ib{f(\rX(\rs)) \neq f(\rX(np))} \leq \sum_{\substack{s\colon\\ \{s,s+1\} \subseteq \rI}} \ib{f(\rX(s)) \neq f(\rX(s+1))} = \\ \sum_{s=0}^{n-1} \ib{\{s,s+1\} \subseteq \rI} \cdot \ib{f(\rX(s)) \neq f(\rX(s+1))}.
\end{multline*}
 Since $\rI$ is independent of $\rX$, taking expectations we get
\begin{align*}
 \Pr[f(\rX(\rs)) \neq f(\rX(np))] &\leq \sum_{s=0}^{n-1} \Pr[\{s,s+1\} \subseteq \rI] \Pr[f(\rX(s)) \neq f(\rX(s+1))] \\ &=
 \sum_{s=np}^{n-1} \Pr[\rs \geq s+1] \Pr[f(\rX(s)) \neq f(\rX(s+1))] \\ &+ \sum_{s=0}^{np-1} \Pr[\rs \leq s] \Pr[f(\rX(s)) \neq f(\rX(s+1))]. \qedhere
\end{align*}
\end{proof}

We can write the total influence in a very similar form.

\begin{lemma} \label{lem:total-influence}
 For every Boolean function $f$,
\[
 \Inf[f] = \sum_{s=0}^{n-1} p(1-p)n\Pr[\rt = s] \Pr[f(\rX(s)) \neq f(\rX(s+1))].
\]
\end{lemma}
\begin{proof}
 Let $(x,y)$ be a random edge of the cube oriented so that $|x| < |y|$, where $|x|$ is the Hamming weight of $x$. We can pick $(x,y)$ in the following way. Pick $z \sim \mu_p$ and a random coordinate $i$, and let $x = z|_{i=0}$ and $y = z|_{i=1}$ (that is, $x$ is obtained from $z$ by setting the $i$th coordinate to zero, and $y$ by setting it to one). The Hamming weight of $x$ thus has the same distribution of $\rt$. Therefore $(x,y) \sim (\rX(\rt),\rX(\rt+1))$, which directly implies the formula given above.
\end{proof}

In order to bound $\Pr[f(\rX(\rs)) \neq f(\rX(np))]$, it remains to analyze the ratio between the coefficients in the two lemmas.

\begin{lemma} \label{lem:low-influence}
 If $p(1-p)n \to \infty$ then for every Boolean function $f$,
\[
 \Pr[f(\rX(\rs)) \neq f(\rX(np))] = O\left(\frac{1}{\sqrt{p(1-p)n}}\Inf[f]\right).
\]
\end{lemma}
\begin{proof}
 In view of Lemma~\ref{lem:hybrid} and Lemma~\ref{lem:total-influence}, it remains to bound from above the following two ratios:
\begin{align*}
 \rho_1(s) &= \frac{\Pr[\rs \geq s+1]}{p(1-p)n\Pr[\rt = s]} \;\; (pn \leq s \leq n-1), \\
 \rho_2(s) &= \frac{\Pr[\rs \leq s]}{p(1-p)n\Pr[\rt = s]} \;\; (0 \leq s \leq pn-1).
\end{align*}
 It is not hard to check that
\[
 \rho_1(s) \leq \frac{\Pr[\rt \geq s]}{p(1-p)n\Pr[\rt = s]}, \quad
 \rho_2(s) \leq \frac{\Pr[\rt \leq s]}{p(1-p)n\Pr[\rt = s]}.
\]
 Log-concavity of the binomial coefficients shows that $\rho_1$ is decreasing while $\rho_2$ is increasing, and so in the given ranges, $\rho_1(s) \leq \rho_1(np)$ and $\rho_2(s) \leq \rho_2(np)$. It is known that the median of the binomial distribution $\bin(n-1,p)$ is one of $np-1,np$, and so the numerator of both $\rho_1(np)$ and $\rho_2(np)$ is $1/2 \pm o(1)$. The local limit theorem shows that the common denominator is $(1\pm o(1))\sqrt{p(1-p)n/2\pi}$. Therefore
\[
 \rho_1(s), \rho_2(s) \leq \sqrt{\frac{\pi}{2} \frac{1}{p(1-p)n}}. \qedhere
\]
\end{proof}

Our main result in this section immediately follows.

\begin{theorem} \label{thm:invariance-low-influence}
 If $p(1-p)n \to \infty$ and $f$ is a Boolean function satisfying $\Inf[f] = o(\sqrt{p(1-p)n})$ then $|\EE_{\mu_p}[f] - \EE_{\nu_{pn}}[f]| = o(1)$.
\end{theorem}

As a corollary, we prove that balanced symmetric Boolean functions cannot be approximated by Boolean functions whose influence is $o(\sqrt{p(1-p)n})$. 

\begin{proposition} \label{prop:balanced-symmetric}
 There exists a constant $\delta > 0$ such that for all $p \in (0,1)$ and large enough $n$, if $f$ is a symmetric Boolean function such that $\frac{1}{3} \leq \EE_{\mu_p}[f] \leq \frac{2}{3}$ then $\Pr_{\mu_p}[f \neq g] > \delta$ for every Boolean function $g$ satisfying $\Inf[g] = o(\sqrt{p(1-p)n})$.
\end{proposition}
Our result can be considered a generalization of two recent results (with much weaker bounds):
\begin{itemize} 
\item O'Donnell and Wimmer~\cite{ODonnellWimmer} proved much tighter results (in terms of $\delta$) in the special case of the majority function.
\item Tal~\cite{Tal} proved tighter bounds (in terms of $\delta$) for general symmetric functions assuming $g$ is an $\mathsf{AC}^0$ circuit whose influence sum is $o(\sqrt{p(1-p)n})$.
\end{itemize} 

\begin{proof}
 It is known that the total variation distance between $\mu_p$ and $\mu_q$ in the regime $|p-q| = o(1)$ is $O(|p-q| \sqrt{\frac{n}{p(1-p)}})$. Choose $C > 0$ so that $|p-q| \leq C\sqrt{p(1-p)/n}$ implies that the variation distance between $\mu_p$ and $\mu_q$ is at most $1/4$. Call a $q$ satisfying this bound \emph{good}.

 Let $g$ be a Boolean function satisfying $\Inf[g] = o(\sqrt{p(1-p)n})$.
 Suppose that $\Pr_{\mu_p}[f \neq g] \leq \delta$. If $x \sim \mu_p$ then $|x|/n$ is good with constant probability, and so there exists a slice $\binom{[n]}{qn}$ with $q$ good such that $\Pr_{\nu_{qn}}[f \neq g] = O(\delta)$.

 Theorem~\ref{thm:invariance-low-influence} implies that $|\EE_{\mu_q}[g] - \EE_{\nu_{qn}}[g]| = o(1)$. On the one hand, $|\EE_{\nu_{qn}}[g] - \EE_{\nu_{qn}}[f]| = O(\delta)$, and so the symmetry of $f$ implies that either $\EE_{\nu_{qn}}[g] = O(\delta)$ or $\EE_{\nu_{qn}}[g] = 1-O(\delta)$. On the other hand,
$|\EE_{\mu_q}[g] - \EE_{\mu_p}[f]| \leq |\EE_{\mu_q}[g] - \EE_{\mu_p}[g]| + |\EE_{\mu_p}[g] - \EE_{\mu_p}[f]| \leq 1/4 + \delta$. We conclude that either $\EE_{\mu_p}[f] \leq 1/4 + O(\delta) + o(1)$ or $\EE_{\mu_p}[f] \geq 3/4 - O(\delta) - o(1)$. For an appropriate choice of $\delta$, this contradicts the assumption $1/3 \leq \EE_{\mu_p}[f] \leq 2/3$ for large enough $n$.
\end{proof}

\section{Decomposing the slice} \label{sec:decomposition}

In this section we describe two ways of decomposing the space of functions over the slice. One decomposition arises naturally from the harmonic representation: $f = \sum_{d \leq n/2} f^{=d}$. The corresponding decomposition of the space of functions is into the subspaces of homogeneous harmonic multilinear polynomials of degree $d$, for $0 \leq d \leq n/2$. We call this decomposition the \emph{coarse decomposition}, and discuss it further in Section~\ref{sec:coarse-decomposition}. This decomposition is the underlying reason for results such as Theorem~\ref{thm:spherical}, Lemma~\ref{lem:transpositions}, and Lemma~\ref{lem:spherical-derivative}

The coarse decomposition can be refined into a distinguished basis of the space of functions over the slice, known as the \emph{Gelfand--Tsetlin basis}, which is described in Section~\ref{sec:gelfand-tsetlin}. This decomposition has been used by the first author~\cite{F} to simplify the proof of Friedgut's junta theorem for the slice, due to Wimmer~\cite{Wimmer}.

This section assumes basic knowledge of the representation theory of the symmetric group, such as the one provided by Chapter~2 and Chapter~8 of Diaconis~\cite{Diaconis}.

\subsection{Coarse decomposition} \label{sec:coarse-decomposition}

Every harmonic function $f$ can be decomposed as a sum of homogeneous parts:
\[
 f = \sum_{d \leq n/2} f^{=d}.
\]
This decomposition naturally arises from the representation theory of the slice. In this subsection we discuss this connection, and obtain a representation-theoretic proof of Theorem~\ref{thm:spherical} as a result.

So far we have viewed the slice as a subset of the Boolean cube. However, it is also possible to view the slice $\binom{[n]}{k}$ as the set of cosets of $S_k \times S_{n-k}$ inside $S_n$. In other words, we identify a subset $A \in \binom{[n]}{k}$ with the set of permutations $\pi$ such that $\pi(\{1,\ldots,k\}) = A$. From this point of view it is natural to consider the action of $S_n$ on $\binom{[n]}{k}$ by permutation of the coordinates. The resulting module $\RR[\binom{[n]}{k}]$ is isomorphic to the permutation module $M^{(n-k,k)}$ (assuming $k \leq n/2$) whose decomposition into irreducible modules (Specht modules) can easily be computed to be
\[
 \RR\left[\binom{[n]}{k}\right] \approx M^{(n-k,k)} \approx S^{(n)} \oplus S^{(n-1,1)} \oplus \cdots \oplus S^{(n-k,k)}.
\]
Since each irreducible representation appears once, $(S_n,S_k \times S_{n-k})$ forms a \emph{Gelfand pair}. Classical representation theory of the symmetric group (for example,~\cite[\S 2.9--2.10]{Sagan}) allows us to identify $S^{(n-d,d)}$ with the space of all homogeneous harmonic multilinear polynomials of degree $d$ over the slice.

Consider now the middle slice $\binom{[n]}{\lfloor n/2 \rfloor}$. We can uniquely identify every harmonic multilinear polynomial on $x_1,\ldots,x_n$ with a function over the middle slice. If $f,g$ are harmonic multilinear polynomials and $\alpha$ is an exchangeable measure then $\EE_\alpha[fg] = f^TAg$, where $A$ commutes with the action of $S_n$ (since $\alpha$ is exchangeable). Since $A$ commutes with the action, Schur's lemma implies that it acts on each irreducible $S^{(n-d,d)}$ by scalar multiplication (since each irreducible representation appears only once). This implies that for some constants $\lambda_0,\ldots,\lambda_{\lfloor n/2 \rfloor}$ depending only on $\alpha$,
\[
 f^TAg = \sum_{d \leq n/2} \lambda_d \langle f^{=d}, g^{=d} \rangle,
\]
where $f^{=d}$ is the component of $f$ in $S^{(n-d,d)}$ (defined uniquely since there is only one copy of this irreducible representation) and the inner product is given by $\langle x,y \rangle = x^Ty$; there are no mixed summands $\langle f^{=d}, g^{=e} \rangle$ since the decomposition into irreducible representations is orthogonal. Theorem~\ref{thm:spherical} follows by taking, for each $d \leq n/2$,
\[
 C_{f,g} = \frac{\langle f^{=d}, g^{=d} \rangle}{\| (x_1-x_2) \ldots (x_{2d-1}-x_{2d}) \|^2}.
\]

A similar argument explains why the decomposition into homogeneous parts appears in some of our other results such as Lemma~\ref{lem:transpositions} and Lemma~\ref{lem:spherical-derivative}. More generally, if $A$ is any operator on $\RR[\binom{[n]}{k}]$ which commutes with the action of $S_n$ then
\[
 A f = \sum_{d \leq \min(k,n-k)} \lambda_d f^{=d},
\]
where $\lambda_0,\ldots,\lambda_d$ are the eigenvalues of $A$.

\medskip

The decomposition into homogeneous parts appears in several other guises:

\begin{description}
 \item[Association schemes] The operators on $\RR[\binom{[n]}{k}]$ which commute with the action of $S_n$ form a commutative algebra knows as the \emph{Bose--Mesner algebra} of the \emph{Johnson association scheme}. This algebra has two important bases: the spatial basis corresponds to the characterization of the algebra as the set of matrices $A$ in which $A(S,T)$ depends only on $|S \cap T|$, and the spectral basis corresponds to the decomposition into homogeneous parts. For more on association schemes, see Bannai and Ito~\cite{BannaiIto}.
 \item[Regular semilattices and spherical posets] For each $k$, we can consider the subset of the $n$-dimensional Boolean cube consisting of all sets of cardinality at most $k$ as a lattice under the set inclusion order. The truncated Boolean lattice is an example of a \emph{regular semilattice} and, when $k \leq n/2$, of a \emph{spherical lattice}. Delsarte~\cite{Delsarte} and Stanton~\cite{Stanton} gave general constructions which recover the decomposition of $\RR[\binom{[n]}{k}]$ into its irreducible components. For more, see Ceccherini-Silberstein et al.~\cite[\S8.2--8.4]{CSST}.
 \item[Differential posets] We can also consider the entire Boolean cube as a lattice. The Boolean lattice is an example of a \emph{$\mu$-differential poset}. Stanley~\cite{Stanley90} gave a general construction, which in the special case of the Boolean lattice decomposes it into the various slices, and each slice into the irreducible components. This decomposition corresponds to the \emph{Terwilliger algebra} of the \emph{binary Hamming association scheme}. For more on this area, see Engel~\cite[\S6.2]{Engel}.
\end{description}

\subsection{Gelfand--Tsetlin basis} \label{sec:gelfand-tsetlin}

So far we have discussed the decomposition of each function on the slice into its homogeneous parts. This decomposition has a similar counterpart for functions on the Boolean cube, the so-called \emph{levels} of the Fourier expansion. For functions on the Boolean cube, the decomposition can be refined to the \emph{Fourier basis}. The Fourier basis is uniquely defined (up to scalar multiplication) as the set of characters of the group $\mathbb{Z}_2^n$. In the case of the slice, we can obtain such a canonical basis by specifying an order on the coordinates, say the standard order $x_1,\ldots,x_n$.

Recall the decomposition into irreducible $S_n$-modules of $\RR[\binom{[n]}{k}]$ discussed in the preceding subsection:
\[
 \RR\left[\binom{[n]}{k}\right] \approx M^{(n-k,k)} \approx S^{(n)} \oplus \cdots \oplus S^{(n-k,k)}.
\]
If we consider $S^{(n-d,d)}$ (the space of homogeneous harmonic multilinear polynomials of degree $d$) as a module over $S_{n-1}$ rather than over $S_n$, then it is no longer irreducible. Instead, it decomposes into two irreducibles: $S^{(n-1-d,d)} \oplus S^{(n-d,d-1)}$. If we then consider each of these irreducibles as modules over $S_{n-2}$ they decompose even further, and continuing in this way, eventually we get a decomposition into irreducibles of dimension~$1$. The corresponding basis (defined up to scalar multiplication) is known as the \emph{Gelfand--Tsetlin basis}.

Srinivasan~\cite{Srinivasan} described an inductive construction of this basis using ideas from Sperner theory, which are related to the theory of differential posets mentioned above. Ambainis et al.~\cite{ABRW} gave an inductive construction which closely follows the definition, in the context of quantum computation. Filmus~\cite{F} gave an explicit formula for the basis elements, which we now describe.

A sequence $(a_1,b_1),\ldots,(a_d,b_d)$ is \emph{admissible} if:
\begin{enumerate}
 \item All $2d$ numbers are distinct and belong to $[n]$.
 \item $b_1 < \cdots < b_d$.
 \item $a_i < b_i$ for all $i$.
\end{enumerate}
A set $B = \{b_1,\ldots,b_d\}$ (where $b_1 < \cdots < b_d$) is \emph{admissible} if it can be completed to an admissible sequence (equivalently, it is the bottom row of a standard Young tableau of shape $(n-d,d)$).
There are $\binom{n}{d} - \binom{n}{d-1}$ admissible sets of size $d$. For each admissible set $B$, we define a basis element
\[
 \chi_B = \sum_{\substack{a_1,\ldots,a_d\colon \\ (a_1,b_1),\ldots,(a_d,b_d) \text{ admissible}}} (x_{a_1} - x_{b_1}) \cdots (x_{a_d} - x_{b_d}).
\]
In total there are $\binom{n}{k}$ admissible sets of size at most $k$, and the corresponding basis elements constitute the Gelfand--Tsetlin basis for $\binom{[n]}{k}$. Furthermore, if $\alpha$ is any exchangeable measure then
\[
 \|\chi_B\|_\alpha^2 = \binom{b_1}{2} \binom{b_2-2}{2} \cdots \binom{b_d-2(d-1)}{2} \|(x_1-x_2)\cdots(x_{2d-1}-x_{2d})\|_\alpha^2.
\]
For proofs, see Filmus~\cite{F}.

The Gelfand--Tseltlin basis can also be characterized (up to scalar multiplication) as the common eigenvectors of the Gelfand--Tsetlin algebra, which is generated by the Young--Jucys--Murphy elements
\[
 X_m = \sum_{i < m} (i\;m).
\]
Filmus~\cite[Lemma 24]{F} gives the following explicit formula:
\[
 \left[\sum_{\ell=1}^m X_\ell\right] \chi_B = \left[\binom{m}{2} - |B \cap [m]| (m + 1 - |B \cap [m]|)\right] \chi_B.
\]

Finally, this basis can be constructed in the same way as Young's orthogonal basis. For each $d$, consider all standard Young tableaux of shape $(n-d,d)$, arranged in lexicographic order of the second row. With each tableau having first row $a_1,\ldots,a_{n-d}$ and second row $b_1,\ldots,b_d$, associate the function $(x_{a_1}-x_{b_1})\cdots(x_{a_d}-x_{b_d})$. Running the Gram--Schmidt orthogonalization process on these functions yields the degree~$d$ part of the Gelfand--Tsetlin basis described above.

\section{Why harmonic functions?} \label{sec:harmonic-use}

Our invariance principle compares the distribution of harmonic multilinear polynomials on the slice and on the Boolean cube (with respect to an appropriate measure). This empirically validates the following informal claim:
\emph{
  The correct way to lift a function from the slice to the Boolean cube is through its unique representation as a bounded degree harmonic multilinear polynomial.
}

In the same way, the classical invariance principle of Mossel et al.~\cite{MOO} suggests that the correct way to lift a function from the Boolean cube to Gaussian space it through its unique representation as a multilinear polynomial. Alternatively, given a function on the Boolean cube, we lift it to a function on Gaussian space by linear interpolation in all axis-parallel directions.

In this section we discuss two other justifications for using harmonic multilinear polynomials, from commutative algebra and from representation theory. We briefly mention that James' intersecting kernel theorem~\cite{James} (see also~\cite[\S10.9]{CSST} and~\cite[\S6.2]{CSST2}) draws together both points of view, and generalizes them from the slice to the \emph{multislice}, in which the object of study is an arbitrary permutation module of the symmetric group. We hope to explore this connection in future work.

\subsection{Commutative algebra} \label{sec:harmonic-algebra}

Let us consider first the case of the Boolean cube $\{-1,1\}^n$. The classical invariance principle~\cite{MOO} compares the distribution of multilinear polynomials over the Boolean cube and over Gaussian space. Multilinear polynomials arise naturally when we consider $\{-1,1\}^n$ as a variety defined by the ideal $I = \langle x_1^2-1, \ldots, x_n^2-1 \rangle$. The coordinate ring over the variety consists of polynomials over $x_1,\ldots,x_n$ modulo the ideal $I$. It is immediately clear that every polynomial is equivalent to a multilinear polynomial modulo $I$, and a dimension computation shows that this representation is unique.

We can treat the case of the slice $\binom{[n]}{k}$ in a similar way. The slice is a variety defined by the ideal
\[
 I_k = \bigl\langle x_1^2 - x_1, \ldots, x_n^2 - x_n, \sum_{i=1}^n x_i - k \bigr\rangle.
\]
Theorem~\ref{thm:harmonic-representation} shows that the coordinate ring is isomorphic to the space of harmonic multilinear polynomials of degree at most $\min(k,n-k)$.
Intuitively speaking, the existence of $x_i^2 - x_i$ in the ideal allows us to reduce every polynomial to a multilinear polynomial; the existence of $\sum_{i=1}^n x_i - k$ corresponds to the harmonicity constraint; and the degree constraint follows from the fact that $x_{i_1} \ldots x_{i_{k+1}} = 0$ (when $k \leq n/2$).

We can formalize this intuition to give another proof of Theorem~\ref{thm:harmonic-representation}, which uses arguments due to Blekherman~\cite{Blekherman} quoted in Lee et al.~\cite{LPWY}. Similar arguments can be found in Engel~\cite[\S 6.2]{Engel}.
\begin{theorem}[Reformulation of Theorem~\ref{thm:harmonic-representation}] \label{thm:harmonic-representation-2}
 Let $0 \leq k \leq n$. For every polynomial $P$ over $x_1,\ldots,x_n$ there exists a unique harmonic multilinear polynomial $Q$ of degree at most $\min(k,n-k)$ such that $P \equiv Q \pmod{I_k}$.
\end{theorem}
\begin{proof}
 We start by showing that every polynomial is equivalent modulo $I_k$ to a harmonic multilinear polynomial of degree at most $k$. We will use the following definitions:
\begin{itemize}
 \item Given $S \subseteq [n]$, $x_S = \prod_{i \in S} x_i$.
 \item We denote by $P_d$ the linear space of all homogeneous multilinear polynomials of degree $d$, and by $H_d \subset P_d$ its subspace consisting of harmonic polynomials.
\end{itemize}

 The first step is replacing $x_i^r$ (for $r > 1$) by $x_i$ using the equation $x_i^2 = x_i$, which holds modulo $I_k$. We are left with a multilinear polynomial.

 The second step is removing all monomials of degree larger than $\min(k,n-k)$. If $k \leq n/2$, then monomials of degree larger than $k$ always evaluate to zero on the slice\footnote{Formally speaking, this step requires us to work with $\sqrt{I_k}$. However, the ideal $I_k$ is radical, see for example~\cite[Lemma 6.1]{RW}.}. If $k \geq n/2$, then we get a similar simplification by replacing $x_i$ with $1-(1-x_i)$, and using the fact that the product of more than $n-k$ different factors $1-x_i$ always evaluates to zero on the slice. We are left with a multilinear polynomial of degree at most $\min(k,n-k)$.
 
 Definition~\ref{def:harmonic} defines a differential operator $\Delta = \sum_{i=1}^n \frac{\partial}{\partial x_i}$. When applied to multilinear polynomials, we can also think of it as a formal operator which maps $x_S$ to $\sum_{i \in S} x_{S \setminus \{i\}}$, and extends linearly to multilinear polynomials (this is the Lefschetz lowering operator in the Boolean lattice). 
 When restricted to $P_d$, the range of $\Delta$ is $P_{d-1}$. Linear algebra shows that $P_d$ decomposes into $\ker \Delta$ and $\im \Delta^T$, where $\Delta^T$ (the Lefschetz raising operator) is the operator on $P_{d-1}$ which maps $x_S$ to $\sum_{i \in [n] \setminus S} x_{S \cup \{i\}}$. Modulo $I_k$,
 \[
  \Delta^T x_S = \sum_{i \in [n] \setminus S} x_{S \cup \{i\}} = x_S \sum_{i \in [n] \setminus S} x_i = (k-|S|) x_S = (k-d+1) x_S,
 \]
 the third equality following logically since $x_S \neq 0$ only if $x_i = 1$ for all $i \in S$, and formally using $x_S x_i = x_S$ for all $i \in S$.
 Thus $\im \Delta^T$ can be identified with $P_{d-1}$. We conclude that any polynomial in $P_d$ can be written (modulo $I_k$) as the sum of a polynomial in $H_d$ (belonging to $\ker \Delta$) and a polynomial in $P_{d-1}$ (belonging to $\im \Delta^T$).
 
 Applying this construction recursively, we see that every homogeneous multilinear polynomial is equivalent modulo $I_k$ to a harmonic multilinear polynomial, and moreover this operation doesn't increase the degree. This completes the proof that every polynomial is equivalent modulo $I_k$ to a harmonic multilinear polynomial of degree at most $\min(k,n-k)$.
 
 It remains to show that this representation is unique. We first show that $\Delta^T$, considered as an operator from $P_{d-1}$ to $P_d$ for $d \leq n/2$, has full rank $\dim P_{d-1}$. Notice that for $S \subseteq [n]$,
\begin{align*}
 \Delta \Delta^T x_S &= 	\Delta \sum_{i \in [n] \setminus S} x_{S \cup \{i\}} = (n-|S|) x_S + \sum_{i \in [n] \setminus S} \sum_{j \in S} x_{S \cup \{i\} \setminus \{j\}}, \\
 \Delta^T \Delta x_S &= \Delta^T \sum_{j \in S} x_{S \setminus \{j\}} = |S| x_S + \sum_{i \in [n] \setminus S} \sum_{j \in S} x_{S \cup \{i\} \setminus \{j\}}.
\end{align*}
 We conclude that on $P_{d-1}$, the following holds:
\[
 \Delta \Delta^T = \Delta^T \Delta + (2n-d+1) I,
\]
 where $I$ is the identity operator. Note that $\Delta^T \Delta$ is positive semidefinite. Since $d \leq n/2$, we see that $2n-d+1 > 0$, and so $\Delta \Delta^T$ is positive definite, and in particular regular. It follows that $\Delta^T$ is also regular, and so has full rank.
 
 The decomposition $\ker \Delta \oplus \im \Delta^T$ of $P_d$ shows that when $d > 0$, $H_d$ has dimension $\dim \ker \Delta = \binom{n}{d} - \dim \im \Delta^T = \binom{n}{d} - \binom{n}{d-1}$. When $d=0$, the dimension is clearly~$1$. It follows that the space of all harmonic multilinear polynomials of degree at most $\min(k,n-k)$ has dimension $1 + \sum_{d=1}^{\min(k,n-k)} \left[\binom{n}{d} - \binom{n}{d-1}\right] = \binom{n}{k}$, matching the dimension of the space of functions on the slice.
\end{proof}

Using very similar ideas, we can prove Blekherman's theorem stated in Section~\ref{sec:blekherman}.

\begin{theorem}[Reformulation of Theorem~\ref{thm:blekherman}] \label{thm:blekherman-2}
 Every multilinear polynomial $f$ over $x_1,\ldots,x_n$ of degree at most $d \leq n/2$ can be represented uniquely in the form
\[
 f(x_1,\ldots,x_n) \equiv \sum_{i=0}^d f_i(x_1,\ldots,x_n) S^i \pmod{I},
\]
 where $I = \langle x_1^2-x_1,\ldots,x_n^2-x_n \rangle$, $S = x_1+\cdots+x_n$, and $f_i$ is a harmonic multilinear polynomial of degree at most $d-i$.
\end{theorem}
\begin{proof}
 Throughout the proof we use the definitions of $x_S,P_d,H_d,\Delta,\Delta^T$ in the proof of Theorem~\ref{thm:harmonic-representation-2}. We also define $P_{\leq d}$ to be the space of all multilinear polynomials of degree at most $d$, and $H_{\leq d}$ to be the space of all harmonic multilinear polynomials of degree at most $d$.
 
 We start by considering the case in which $f$ is a homogeneous multilinear polynomial. We prove by induction on $e$ that if $\deg f = e$ then for some $\fc{f}{0},\ldots,\fc{f}{e}$,
\[
 f \equiv \sum_{i=0}^e \fc{f}{i} \xsum^i \pmod{I}, \text{ where } \fc{f}{i} \in H_{\leq e-i}.
\]
 When $e = 0$ this is clear, since $f$ is already harmonic, so suppose that $e > 0$.
 Recall that $\Delta$ maps $P_e$ to $P_{e-1}$. The conjugate operator $\Delta^T$ maps $x_A \in P_{e-1}$ to
\[
 \sum_{i \notin A} x_{A \cup \{i\}} = x_A \cdot \sum_{i \notin A} x_i = x_A \cdot \left(\xsum - \sum_{i \in A} x_i\right) \equiv (\xsum - (e-1)) x_A \pmod{I}.
\]
 Since $P_e = \ker \Delta \oplus \im \Delta^T = H_d \oplus \im \Delta^T$, it follows that $f \equiv h + (\xsum - (e-1)) g \pmod{I}$ for some $h \in H_e$ and $g \in P_{e-1}$. By the induction hypothesis, for some $\fc{g}{0},\ldots,\fc{g}{e-1}$,
\[
 g \equiv \sum_{i=0}^{e-1} \fc{g}{i} \xsum^i \pmod{I}, \text{ where } \fc{g}{i} \in H_{\leq e-1-i}.
\]
 Substituting this in the equation for $f$, we obtain
\[
 f \equiv h + \sum_{i=0}^{e-1} (\xsum - (e-1)) \fc{g}{i} \xsum^i \equiv [h - (e-1) \fc{g}{0}] + \sum_{i=1}^e [\fc{g}{i-1} - (e-1) \fc{g}{i}] \xsum^i \pmod{I},
\]
 where $g_e = 0$.
 Since $\deg [h - (e-1)\fc{g}{0}] \leq \max(e,e-1) = e$ and $\deg [\fc{g}{i-1} - (e-1)\fc{g}{i}] \leq \max((e-1)-(i-1),e-1-i) = e-i$, we have obtained the required representation.
 
 When $f$ is an arbitrary multilinear polynomial of degree $d$, we get the required representation by summing the representations of $f^{=0},\ldots,f^{=d}$.
 
 Finally, we prove that the representation is unique by comparing dimensions:
\[
 \dim P_{\leq d} = \sum_{i=0}^d \binom{n}{i} = \sum_{i=0}^d \dim H_{\leq i} = \sum_{i=0}^d \dim H_{\leq d-i},
\]
 since $H_{\leq e}$ has dimension $\binom{n}{e}$ for $e \leq n/2$ by Theorem~\ref{thm:harmonic-representation}.
\end{proof}

This implies a surprising corollary for \emph{harmonic projections}.

\begin{definition} \label{def:harmonic-projection}
 Let $f$ be a function over $x_1,\ldots,x_n$. Its \emph{harmonic projection} on the slice $\binom{[n]}{k}$ is the unique harmonic multilinear polynomial of degree at most $\min(k,n-k)$ which agrees with $f$ on the slice.
\end{definition}

\begin{corollary} \label{cor:harmonic-projection}
 Let $f$ be a multilinear polynomial over $x_1,\ldots,x_n$ of degree $d$, and let $F_k$ be its harmonic projection on the slice $\binom{[n]}{k}$, where $d \leq k \leq n-d$. For $k$ in this range, $F_k^{=d}$ does not depend on $k$.
\end{corollary}
\begin{proof}
 Let $f_0,\ldots,f_d$ be the Blekherman decomopsition of $f$, where $f_i$ is a harmonic multilinear polynomial of degree at most $d-i$. The function $\sum_{i=0}^d f_i k^i$ is thus a harmonic multilinear polynomial which agrees with $f$ on $\binom{[n]}{k}$.
 When $d \leq k \leq n-d$, this polynomial has degree at most $d \leq \min(k,n-k)$, and so it is the harmonic projection $F_k$ of $f$ on the slice $\binom{[n]}{k}$. Notice that $F_k^{=d} = f_0^{=d}$ does not depend on $k$.
\end{proof}

\subsection{Representation theory} \label{sec:harmonic-rep}

As we have explained in Subsection~\ref{sec:coarse-decomposition}, representation theory naturally leads to the decomposition of each function on the slice to its homogeneous parts. From the point of view of the representation theory of the symmetric group, it is natural to order the irredicuble representations in increasing order of complexity: $S^{(n)},S^{(n-1,1)},\ldots,S^{(n-k,k)}$. This suggests defining the \emph{degree} of a function $f$ on the slice as the maximal $d$ such that $f^{=d} \neq 0$; note that this definition doesn't make use of the harmonic multilinear representation of $f$.

Another definition which leads to the same parameter is the \emph{junta degree}, which is the minimal $d$ such that $f$ can be written as a linear combination of functions depending on at most $d$ coordinates. The cometric property (Corollary~\ref{cor:cometric}), which (apart from a few degenerate cases) only holds for this ordering of the irreducible representations, gives further credence to the notion of degree.

A similar notion of degree exists also for functions on the Boolean cube: the degree of a function on $\{0,1\}^n$ is both its degree as a multilinear polynomial and its junta degree. 

The following result shows that the lifting defined by the harmonic multilinear representation is the unique way to couple both notions of degree.

\begin{theorem} \label{thm:harmonic-lifting}
 Let $0 \leq k \leq n/2$. Suppose $L$ is a linear operator from $\RR[\binom{[n]}{k}]$ to $\RR[\{0,1\}^n]$ such that
\begin{enumerate}[(a)]
 \item For every function $f$ over $\binom{[n]}{k}$, $Lf$ is an extension of $f$, that is, $(Lf)(x) = f(x)$ for every $x \in \binom{[n]}{k}$.
 \item For every function $f$ over $\binom{[n]}{k}$, $\deg Lf \leq \deg f$.
 \item For every function $f$ over $\binom{[n]}{k}$ and for every permutation $\pi \in S_n$, $L(f^\pi) = (Lf)^\pi$. In other words, $L$ commutes with the action of $S_n$.
\end{enumerate}
 Then $Lf$ is given by the unique harmonic multilinear polynomial of degree at most $k$ which agrees with $f$ on $\binom{[n]}{k}$.
\end{theorem}
\begin{proof}
 In view of Lemma~\ref{lem:harmonic-elementary}, it suffices to show that $Lf = f$ for $f = (x_1 - x_2) \cdots (x_{2d-1} - x_{2d})$.
 Suppose that $Lf = g$, where $g$ is a multilinear polynomial of degree at most $d$. Since $f^{(1\;2)} = -f$, also $g^{(1\;2)} = -g$. If we write
\[
 g = x_1x_2 r + x_1 s + x_2 t + u,
\]
 where $r,s,t,u$ don't involve $x_1,x_2$, then $g^{(1\;2)} = x_1x_2 r + x_1 t + x_2 s + u$, showing that $r = u = 0$ and $s = -t$. In other words, $g = (x_1 - x_2) s$ is a multiple of $x_1 - x_2$. Similarly $g$ is a multiple of $x_3-x_4,\ldots,x_{2d-1}-x_{2d}$. Unique factorization forces $g = C(x_1-x_2) \cdots (x_{2d-1}-x_{2d}) = Cf$ for some constant $C$. Since $g$ extends $f$, the constant must be $C = 1$.
\end{proof}

\bibliographystyle{plain}
\bibliography{harmonicity}

\begin{thebibliography}{10}

\bibitem{ABRW}
Andris Ambainis, Aleksandrs Belovs, Oded Regev, and Ronald de~Wolf.
\newblock Efficient quantum algorithms for (gapped) group testing and junta
  testing.
\newblock In {\em Proceedings of the twenty-seventh annual ACM-SIAM symposium
  on discrete algorithms (SODA '16)}, pages 903--922, 2016.

\bibitem{BannaiIto}
Eiichi Bannai and Tatsuro Ito.
\newblock {\em Algebraic Combinatorics I: Association Schemes}.
\newblock Mathematics lecture notes series. Benjamin / Cummings, 1984.

\bibitem{Beckner}
William Beckner.
\newblock Inequalities in {F}ourier analysis.
\newblock {\em Ann. Math.}, 102:159--182, 1975.

\bibitem{Bergeron}
Fran\'cois Bergeron.
\newblock {\em Algebraic Combinatorics and Coinvariant Spaces}.
\newblock CMS Treatises in Mathematics. A~K~Peters, 2009.

\bibitem{Blekherman}
Greg Blekherman.
\newblock Symmetric sums of squares on the hypercube, 2015.
\newblock Manuscript in preparation.

\bibitem{Bonami}
Aline Bonami.
\newblock {\'E}tude des coefficients {F}ourier des fonctions de {$L^p(G)$}.
\newblock {\em Ann. Inst. Fourier}, 20(2):335--402, 1970.

\bibitem{Bop}
Ravi~B Boppana.
\newblock The average sensitivity of bounded-depth circuits.
\newblock {\em Information Processing Letters}, 63(5):257--261, 1997.

\bibitem{CarberyWright}
Anthony Carbery and James Wright.
\newblock Distributional and {$L^q$} norm inequalities for polynomials over
  convex bodies in {$\mathbb{R}^n$}.
\newblock {\em Math. Res. Lett.}, 3(8):233--248, 2001.

\bibitem{CSST}
Tullio Ceccherini-Silberstein, Fabio Scarabotti, and Filippo Tolli.
\newblock {\em Harmonic analysis on finite groups}, volume 108 of {\em
  Cambridge studies in advanced mathematics}.
\newblock Cambridge University Press, 2008.

\bibitem{CSST2}
Tullio Ceccherini-Silberstein, Fabio Scarabotti, and Filippo Tolli.
\newblock {\em Representation theory of the symmetric groups}, volume 121 of
  {\em Cambridge studies in advanced mathematics}.
\newblock Cambridge University Press, 2010.

\bibitem{Delsarte}
Phillipe Delsarte.
\newblock Association schemes and {$t$}-designs in regular semilattices.
\newblock {\em J. Comb. Theory Ser. A}, 20(2):230--243, 1976.

\bibitem{Diaconis}
Persi Diaconis.
\newblock {\em Group representations in probability and statistics}, volume~11
  of {\em Institute of Mathematical Statistics Lecture Notes---Monograph
  Series}.
\newblock Institute of Mathematical Statistics, Hayward, CA, 1988.

\bibitem{DCS}
Persi Diaconis and Laurent Saloff-Coste.
\newblock Logarithmic {S}obolev inequalities for finite {M}arkov chains.
\newblock {\em Ann. Appl. Prob.}, 6(3):695--750, 1996.

\bibitem{Dunkl76}
Charles~F. Dunkl.
\newblock A {K}rawtchouk polynomial addition theorem and wreath products of
  symmetric groups.
\newblock {\em Indiana Univ. Math. J.}, 25:335--358, 1976.

\bibitem{Dunkl79}
Charles~F. Dunkl.
\newblock Orthogonal functions on some permutation groups.
\newblock In {\em Relations between combinatorics and other parts of
  mathematics}, volume~34 of {\em Proc. Symp. Pure Math.}, pages 129--147,
  Providence, RI, 1979. Amer. Math. Soc.

\bibitem{Engel}
Konrad Engel.
\newblock {\em Sperner Theory}, volume~65 of {\em Encyclopedia of Mathematics
  and its Applications}.
\newblock Cambridge University Press, 1997.

\bibitem{F}
Yuval Filmus.
\newblock An orthogonal basis for functions over a slice of the boolean
  hypercube.
\newblock {\em Elec. J. Comb.}, 23(1):P1.23, 2016.

\bibitem{FKMW}
Yuval Filmus, Guy Kindler, Elchanan Mossel, and Karl Wimmer.
\newblock Invariance principle on the slice.
\newblock In {\em 31st Conf. Comp. Comp.}, 2016.

\bibitem{Has}
Johan H{\aa}stad.
\newblock Almost optimal lower bounds for small depth circuits.
\newblock In Silvio Micali, editor, {\em Randomness and Computation}, volume~5
  of {\em Advances in Computing Research}, pages 143--170. JAI Press, 1989.

\bibitem{Hoe}
Wassily Hoeffding.
\newblock Probability inequalities for sums of bounded random variables.
\newblock {\em Journal of the American statistical association},
  58(301):13--30, 1963.

\bibitem{James}
Gordon~D. James.
\newblock A characteristic free approach to the representation theory of
  {$S_n$}.
\newblock {\em J. Algebra}, 46:430--450, 1977.

\bibitem{KellerKlein}
Nathan Keller and Ohad Klein.
\newblock A structure theorem for almost low-degree functions on the slice.
\newblock Manuscript.

\bibitem{Kindler}
Guy Kindler.
\newblock {\em Property testing, {PCP} and Juntas}.
\newblock PhD thesis, Tel-Aviv University, 2002.

\bibitem{KindlerSafra}
Guy Kindler and Shmuel Safra.
\newblock Noise-resistant {B}oolean functions are juntas, 2004.
\newblock Unpublished manuscript.

\bibitem{LPWY}
Troy Lee, Anupam Prakash, Ronald de~Wolf, and Henry Yuen.
\newblock On the sum-of-squares degree of symmetric quadratic functions.
\newblock In {\em Proceedings of the 31st Conference on Computational
  Complexity (CCC 2016)}, pages 17:1--17:31, 2016.

\bibitem{LeeYau}
Tzong-Yau Lee and Horng-Tzer Yau.
\newblock Logarithmic {S}obolev inequality for some models of random walks.
\newblock {\em Ann. Prob.}, 26(4):1855--1873, 1998.

\bibitem{LMN}
N.~Linial, Y.~Mansour, and N.~Nisan.
\newblock Constant depth circuits, fourier transform and learnability.
\newblock {\em Journal of the ACM}, 40(3):607--620, 1993.

\bibitem{LZ}
Terry~J. Lyons and T.~S. Zhang.
\newblock Decomposition of {D}irichlet processes and its application.
\newblock {\em Ann. Probab.}, 22(1):494--524, 1994.

\bibitem{MOO}
Elchanan Mossel, Ryan O'Donnell, and Krzysztof Oleszkiewicz.
\newblock Noise stability of functions with low influences: Invariance and
  optimality.
\newblock {\em Ann. Math.}, 171:295--341, 2010.

\bibitem{NPSS}
Assaf Naor, Yuval Peres, Oded Schramm, and Scott Sheffield.
\newblock Markov chains in smooth {B}anach spaces and {G}romov-hyperbolic
  metric spaces.
\newblock {\em Duke Math J.}, 134(1):165--197, 2006.

\bibitem{ODonnell}
Ryan O'Donnell.
\newblock {\em Analysis of {B}oolean functions}.
\newblock Cambridge University Press, 2014.

\bibitem{ODonnellWimmer}
Ryan O'Donnell and Karl Wimmer.
\newblock Approximation by {DNF}: Examples and counterexamples.
\newblock In {\em Automata, Languages and Programming}, volume 4596 of {\em
  Lecture Notes in Computer Science}, pages 195--206. Springer Berlin
  Heidelberg, 2007.

\bibitem{Oleszkiewicz}
Krzysztof Oleszkiewicz.
\newblock On a nonsymmetric version of the {K}hinchine--{K}ahane inequality.
\newblock {\em Progr. Probab.}, 56:157--168, 2003.

\bibitem{RW}
Prasad Raghavendra and Ben Weitz.
\newblock On the bit complexity of sum-of-squares proofs.
\newblock {\em CoRR}, abs/1702.05139, 2017.

\bibitem{Sagan}
Bruce~E. Sagan.
\newblock {\em The Symmetric Group: Representations, Combinatorial Algorithms,
  and Symmetric Functions}, volume 203 of {\em Graduate Texts in Mathematics}.
\newblock Springer New York, 2001.

\bibitem{Srinivasan}
Murali~K. Srinivasan.
\newblock Symmetric chains, {G}elfand--{T}setlin chains, and the {T}erwilliger
  algebra of the binary {H}amming scheme.
\newblock {\em J. Algebr. Comb.}, 34(2):301--322, 2011.

\bibitem{Stanley90}
Richard~P. Stanley.
\newblock Variations on differential posets.
\newblock {\em IMA Vol. Math. Appl.}, 19:145--165, 1990.

\bibitem{Stanton}
Dennis Stanton.
\newblock Harmonics on posets.
\newblock {\em J. Comb. Theory Ser. A}, 40(1):136--149, 1985.

\bibitem{Tal}
Avishay Tal.
\newblock Tight bounds on the {F}ourier spectrum of {$AC^0$}.
\newblock Manuscript, 2017.

\bibitem{Talagrand}
Michel Talagrand.
\newblock On {R}usso's approximate zero-one law.
\newblock {\em Ann. Prob.}, 22(3):1576--1587, 1994.

\bibitem{Turner}
L.~Richard Turner.
\newblock Inverse of the {V}andermonde matrix with applications.
\newblock Technical Report NASA TN D-3547, Lewis Research Center, NASA,
  Cleveland, Ohio, August 1966.

\bibitem{Wimmer}
Karl Wimmer.
\newblock Low influence functions over slices of the {B}oolean hypercube depend
  on few coordinates.
\newblock In {\em Conference on Computational Complexity (CCC 2014)}, pages
  120--131, 2014.

\end{thebibliography}

\end{document}